\documentclass[11pt]{amsart}
\usepackage[utf8]{inputenc}
\usepackage{amsmath,amssymb, mathrsfs}
\usepackage{bbm}
\usepackage{graphicx}

\usepackage{marvosym}
\usepackage{wasysym}%
\newcommand{\arian}{\textbf{\currency}}

\usepackage{bbold}
\usepackage{centernot}
\usepackage{cite}
\usepackage[dvipsnames]{xcolor}

\usepackage{tikz}
\usetikzlibrary{arrows}
\usetikzlibrary{patterns}
\usetikzlibrary{shapes}
\usepackage{mathrsfs}
\usepackage{bbm}
\usepackage{enumerate}
\usepackage[shortlabels]{enumitem}
\usepackage{tikz}

\usepackage{amssymb}
\usepackage{tikz-cd}
\usetikzlibrary{cd}

\newcommand{\R}{\mathbb{R}}
\newcommand{\N}{\mathbb{N}}
\newcommand{\Z}{\mathbb{Z}}

\newcommand{\de}{\partial}
\renewcommand{\-}{\smallsetminus}
\newcommand{\D}{\mathbb{D}}
\newcommand{\B}{\mathbb{B}}

\renewcommand{\a}{\alpha}

\newcommand{\f}{\varphi}
\newcommand{\e}{\varepsilon}
\newcommand{\w}{{\omega}}

\usepackage{pifont}

\newcommand{\transv}{\mathrel{\text{\tpitchfork}}}
\makeatletter
\newcommand{\tpitchfork}{%
  \raise-0.1ex\vbox{
    \baselineskip\z@skip
    \lineskip-.52ex
    \lineskiplimit\maxdimen
    \m@th
    \ialign{##\crcr\hidewidth\smash{$-$}\hidewidth\crcr$\pitchfork$\crcr}
  }%
}
\makeatother

\newcommand{\en}{\E\#_{X\in W}}
\newcommand{\vol}{\mathrm{vol}}
\newcommand{\ci}{\mathcal{C}^1}
\newcommand{\RP}{\mathbb{R}\mathrm{P}}

\newcommand{\Prob}{(\Omega, \mathfrak{S},\P)}
\renewcommand{\P}{\mathbb{P}}
\newcommand{\E}{\mathbb{E}}
\newcommand{\nrw}{\Rightarrow}
\newcommand{\spt}{\text{supp}}

\newcommand{\g}[3]{\mathcal{G}^{#1}(#2,\R^{#3})}

\newcommand{\goo}[3]{\mathcal{G}^{#1}(#2|#3)}

\newcommand{\coo}[3]{\mathcal{C}^{#1}(#2,#3)}

\newcommand{\mC}{\mathcal{C}}
\newcommand{\mS}{\mathcal{S}}

\newcommand{\Gam}{\Gamma(X,W)}

\newcommand{\HPK}{Let $(O_p,N_0,\phi,g)$ be a KROK model for a KROK couple $(X,W)$}
\newcommand{\HP}{Let $(X,W)$ be KROK}

\usepackage{hyperref}
\hypersetup{
	colorlinks=true,       
	linkcolor=blue,         
	citecolor=blue}

 \usepackage[usenames,dvipsnames]{pstricks}
 \usepackage{epsfig}
 \usepackage{pst-grad} 
 \usepackage{pst-plot} 

\newtheorem{thm}{Theorem}
\newtheorem{lemma}[thm]{Lemma}
\newtheorem{cor}[thm]{Corollary}
\newtheorem{prop}[thm]{Proposition}
\newtheorem{claim}[thm]{Claim}

\theoremstyle{definition}

\newtheorem{defi}[thm]{Definition}
\newtheorem{remark}[thm]{Remark}

\newtheorem{example}[thm]{Example}
\newcommand{\be}{\begin{equation}}
\newcommand{\ee}{\end{equation}}

\newcommand{\bega}{\begin{equation}\begin{aligned}}
\newcommand{\eega}{\end{aligned}\end{equation}}

\newcommand{\MW}{\mathcal{M}}
\newcommand{\XW}{\mathcal{X}}
\newcommand{\NW}{\mathcal{N}}
\newcommand{\DElta}{\mathcal{W}}
\usepackage{mathtools} 
\mathtoolsset{showonlyrefs,showmanualtags} 

\numberwithin{equation}{section}

\title{Kac-Rice formula for transverse intersections}
\author{Michele Stecconi}

\begin{document}
\begin{abstract}
We prove a generalized Kac-Rice formula that, in a well defined regular setting, computes the expected cardinality of the preimage of a submanifold via a random map, by expressing it as the integral of a density. Our proof starts from scratch and although it follows the guidelines of the standard proofs of Kac-Rice formula, it contains some new ideas coming from the point of view of measure theory. Generalizing further, we extend this formula to any other type of counting measure, such as the intersection degree. 

We discuss in depth the specialization to smooth Gaussian random sections of a vector bundle. Here, the formula computes the expected number of points where the section meets a given submanifold of the total space, it holds under natural non-degeneracy conditions and can be simplified by using appropriate connections. Moreover, we point out a class of submanifolds, that we call sub-Gaussian, for which the formula is locally finite and depends continuously  with respect to the covariance of the first jet. In particular, this applies to any notion of singularity of sections that can be defined as the set of points where the jet prolongation meets a given semialgebraic submanifold of the jet space.

Various examples of applications and special cases are discussed. In particular, we report a new proof of the Poincaré kinematic formula for homogeneous spaces and we observe how the formula simplifies for isotropic Gaussian fields on the sphere.
\end{abstract}
\maketitle
\tableofcontents
%
%
%


\subsection{Overview}

What motivates this work is the interest in studying the expected number of realizations of a geometric condition. This topic is at the heart of stochastic geometry and geometric probability and, in recent years, it has gained a role also in subjects of more deterministic nature like enumerative geometry (see \cite{Kostlan:93}, \cite{shsm}, \cite{GaWe1,GaWe2,GaWe3}, \cite{SarnakWigman}, \cite{NazarovSodin1,NazarovSodin2}, \cite{FyLeLu,Letwo,Lerarioshsp,Lerariolemniscate}) and physics (see for instance \cite{marinucci_peccati_2011}, \cite{Wig,wigEBNRF},\cite{park2013betti}).

As a first example, let us consider a random $\mC^1$ function $X\colon M\to \R^n$, where $M\subset \R^m$ is an open subset and let $t\in \R^n$. Then, under appropriate assumptions on $X$, we have the so called \emph{Kac-Rice formula} for the expected cardinality of the set of solutions of the equation $X=t$.
\be\label{eq:KR}
\E\#X^{-1}(t)=\int_M\E\left\{\left|\det(d_uX)\right|\Big|X(u)=t\right\}\rho_{X(u)}(t)du,
\ee
where $\rho_{X(u)}$ is the density of the random variable $X(u)$, meaning that for every event $A\subset \R^n$ we have $\P\{X(u)\in A\}=\int_{A}\rho_{X(u)}(t)dt$. 
First appeared in the independent works by M. Kac \cite[1943]{kac43} and S. O. Rice \cite[1944]{rice44}, this formula is today one of the most important tool in the application of smooth stochastic processes (here called ``random maps'' or ``random fields'') both in pure and applied maths. In fact, the ubiquity of this formula is suggested by its name and birth in that Kac's paper is about random algebraic geometry, while Rice's one deals with the analysis of random noise in electronic engineering. We refer to the book \cite{AdlerTaylor} for a detailed treatment of Kac-Rice formula. We were also inspired by the book \cite{Wschebor} containing a simpler proof in the Gaussian case.

The main result of this paper is the generalization of Kac-Rice formula to one that computes the expected number of points in the preimage of a submanifold $W$, namely the number of solutions of $X\in W$,  rather than just $X=t$. 
\begin{multline}\label{eq:mainformula}
\E\#(X^{-1}(W))= \\=\int_{M}\left(\int_{ S_p\cap W}\E\left\{J_pX\frac{\sigma_q(X,W)}{\sigma_q(S_p,W)}\bigg|X(p)=q\right\}\rho_{X(p)}(q)d(S_p\cap W)(q)\right)dM(p),
\end{multline}
This is the content of Theorems \ref{thm:main}, \ref{thm:maindens} and  \ref{thm:maingraph} reported in Section \ref{sec:main}, after a brief introduction to the problem in Section \ref{sec:intro}. Such theorems are essentially equivalent alternative formulations of the same result. In presenting them, we  pay a special attention to their hypotheses, aiming to propose a setting, that we call KROK\footnote{Stands for ``Kac Rice OK''.} hypotheses (Definition \ref{def:krok}), that appears frequently in random geometry and that is easy to recognize, especially in the Gaussian case.
\begin{remark}[A comment on the proof]
The first idea that comes to mind is to write, locally, the submanifold $W$ as the preimage $\f^{-1}(0)$ of a smooth function $\f\colon N\to \R^m$ and then apply the standard Kac-Rice formula to the random map $\f\circ X$. After that, however, one wants to get rid of $\f$ since it is desirable to have an intrinsic statement, independent from the arbitrary choice of this auxiliary function. In fact, this is the key issue, but it ends up being ugly. So, we chose to reprove everything from the beginning, instead. In doing this, we aim also at proposing an alternative reference for the proof of the standard case. 

Specializing the proof of Theorem \ref{thm:main} to the case in which $W$ is a point, one obtains a proof for the standard Kac-Rice formula, in the KROK setting. Although this setting is very general, the complexity of the proof is comparable to that of Azais and Wschebor \cite{Wschebor} for the Gaussian case and quite simple if compared to the one reported in Adler and Taylor's book \cite{AdlerTaylor}. Moreover, we use an argument that is new in this context: instead of dividing the domain in many little pieces, we interpret the expectation as a measure and use Lebesgue Differentiation Theorem. This makes the hard step of the proof (the ``$\ge$'' part) a little more elegant.
\end{remark}
In Section \ref{sec:gausscase} we focus on the case of Gaussian random sections of a vector bundle. Here, the formula specializes to Theorem \ref{thm:maingau}, where the hypotheses reduce to simple non-degeneracy conditions, thanks also to the Probabilistic Transversality theorem from \cite{dtgrf}. We also provide alternative ways of writing the formula \eqref{eq:mainformula}: as a measure on the submanifold $W$ (Corollary \ref{cor:maingauW}), or using the canonical connection defined by the Gaussian field (Corollary  \ref{cor:connectedform}), see \cite{Nicolaescu2016, AdlerTaylor}. Moreover, in this case we establish a continuity property of the expected number of singular points of a Gaussian random section, with respect to the corresponding covariance tensor (Theorem \ref{thm:mainEgau}). 
This last result has a nice application in the study of semialgebraic singularities of Gaussian random fields (Corollary \ref{cor:mainjgau}).

We also discuss, in Section \ref{sec:weights}, the problem of counting solutions with ``weights'', for instance the intersection degree of $X$ and $W$. Here we show (Theorem \ref{thm:mainalph}) that, under KROK hypotheses, the formula can be directly generalized to hold for any counting measure with measurable weights.

Finally, in section \ref{sec:ex}, we test our formula in two prominent instances of random geometry. First, we show that it can be used to obtain a new quite elementary proof of Poincaré kinematic formula for homogeneous spaces (Theorem \ref{thm:howard}), in the case of zero dimensional intersection; then, we deduce a simple but general formula for isotropic Gaussian random fields on the sphere (Theorem \ref{thm:mainIsotrop}).

\subsection{Structure of the paper}
Sections \ref{sec:intro}-\ref{sec:weights}  contain the presentation of the results of the paper, without proofs. All of their proofs are contained in the Sections \ref{sec:proflemmameas}-\ref{sec:proofgauss}. In particular, Section \ref{sec:gen} is devoted to deduce from the coarea formula that the identity \eqref{eq:mainformula} holds for almost every $W$, under very general assumptions. This essentially allows to prove the ``$\le$'' part of \eqref{eq:mainformula} in the KROK setting, while the opposite inequality is proved in Section \ref{sec:mainproof}. 
Section \ref{sec:ex} contains minor results (and their proof) obtained from applications of the main formula. In the appendix we report some details regarding a few notions of which we make extensive use throughout this paper.

\begin{remark}
 The reader who wants to grasp the meaning of the generalized formula, without going into its more abstract aspects, may just skip Section \ref{sec:main} and go directly to the Gaussian case, discussed in Section \ref{sec:gausscase}. Enough references are provided so that this is a safe practice. However, we recommend that you take a look at Section \ref{sec:intro} first.
\end{remark}

\subsection{Aknowledgements}
The author wishes to thank Antonio Lerario for his useful suggestions indicating the most interesting directions; Léo Mathis and Riccardo Tione for being ``the one with the answer'' and ready to help on multiple occasions.

\section{Introduction}\label{sec:intro}
\subsection{Notations}

\begin{enumerate}[$\bullet$]
 \item We write $\#(S)$ for the cardinality of the set $S$.
\item We use the symbol $A\transv B$ to say that objects $A$ and $B$ are in transverse position, in the usual sense of differential topology (as in \cite{Hirsch}).
    \item The space of $\mC^r$ functions between two manifolds $M$ and $N$ is denoted by $\coo rMN$.
    If $E\to M$ is a vector bundle, we denote the space of its $\mC^r$ sections by $\mathcal{C}^r(M|E)$. In both cases, we consider it to be a topological space endowed with the weak Whitney's topology (see \cite{Hirsch}).
    \item We call $X$ a \emph{random element} (see \cite{Billingsley}) of the topological space $T$ if $X$ is a measurable map $X\colon \Omega\to T$, defined on some probability space $\Prob$ and we denote by $[X]=\P X^{-1}$ the Borel probability measure on $T$ induced by pushforward. We will alternatively use the following equivalent notations:
\be 
\P\{X\in U\}:=[X](U)=\P\left(X^{-1}(U)\right)=\int_{U}d[X],
\ee 
to denote the probability that $X\in U$, for some measurable subset $U\subset T$, and
\be 
\E\{f(X)\}:=\int_{T}f(t)d[X](t),
\ee
to denote the expectation of a measurable vector-valued function $f\colon T\to \R^k$.
Two random elements $X_1,X_2$ are said to be \emph{equivalent} and treated as if they were equal if $[X_1]=[X_2]$. We might call $X$ a \emph{random variable}, \emph{random vector} or \emph{random map} if $T$ is the real line, a vector space or a space of functions $\coo rMN$, respectively.
    \item The sentence: ``$X$ has the property $\mathcal{P}$ almost surely'' (abbreviated ``a.s.'') means that the set  $S=\{f\in \coo1MN| f \text{ has the property }\mathcal{P}\}$ contains a Borel set of $[X]$-measure $1$. It follows, in particular, that the set $S$ is $[X]$-measurable, i.e. it belongs to the $\sigma$-algebra obtained from the completion of the measure space $(\coo1MN,\mathcal{B},[X])$.
    
    \item The bundle of densities of a manifold $M$ is denoted by $\Delta M$, see Appendix \ref{app:densities} for details. If $M$ is a Riemannian manifold, we denote its volume density by $dM$. The subset of positive density elements is denoted by $\Delta^+M$. We denote by $B^+(M)$ the set of positive Borel functions $M\to [0,+\infty]$ and by $L^+(M)$ the set of positive densities, i.e. densities of the form $\rho dM$, where $\rho\in B^+(M)$ and $dM$ is the volume density of some Riemannian metric on $M$. In other words 
\be 
L^+(M)=\left\{ \text{measurable functions }M\ni p\mapsto \delta(p)\in\Delta^+_pM \cup \{+\infty\}\right\}
\ee 
is the set of all non negative, non necessarily finite Borel measurable densities.
The integral of a density $\delta\in L^+(M)$ is written as $\int_M \delta=\int_M \delta(p) dp$.
\item The Jacobian of a $\mC^1$ map $f\colon M\to N$ between Riemannian manifolds (see Definition \ref{def:jacob}), evaluated at a point $p\in M$ is denoted by $J_pf\in [0,+\infty)$. The Jacobian density is then $\delta_pf=J_pf dM(p)$ (see Appendix \ref{app:densities}). If moreover $M$ and $N$ have the same dimension then we may stress this fact by writing $J_pf=|\det (d_pf)|$. In case $f$ is a linear map between Euclidean spaces, then we will just write $Jf:=J_0f=J_pf$. 
    \item
Given a finite dimensional Euclidean space $E$, the expression $\sigma_E(V,W)$ denotes the ``angle''\footnote{Actually a better analogy is with the sine of the angle.} between two vector subspaces $V,W\subset E$, see Appendix \ref{app:angle}. If $f\colon M\to N$ is a $\ci$ map between Riemannian manifolds and $W\subset N$ is a submanifold, we will write shortly
\be 
\sigma_x(W,f):=\sigma_x(W,d_pf):=\sigma_{T_xN}(T_xW, d_pf(T_pM)),
\ee
whenever $f(p)=x\in W$. If $S\subset N$ is another submanifold and $x\in S\cap W$, then
\be 
\sigma_x(W,S):=\sigma_{T_xN}(T_xW, T_xS)).
\ee
\item If $E$ is a Euclidean space, we write $\Pi_V\colon E\to V$ for the orthogonal projection onto a subspace $V\subset E$.
\end{enumerate}
\subsection{The expected counting measure}\label{subs:expcm}
Let us start by considering the following setting. 
\begin{enumerate}[i.]\label{setting}
    \item $M, N$ smooth manifolds ($\mC^\infty$ and without boundary) of dimension $m,n$.
    \item $W\subset N$ smooth submanifold (image of a smooth embedding) of codimension $m$.
    \item $X\colon M\to N$ random $\mC^1$ map, i.e. it represents a Borel probability measure $[X]$ on the topological space $\coo1MN$ endowed with the (weak) Whitney $\mC^1$ topology.
    \item $X\transv W$ almost surely.
\end{enumerate}
If moreover $W$ is closed (this assumption can and will be removed with Lemma \ref{lemma:meas}), then the random set $X^{-1}(W)$ is almost surely a discrete  subset of $M$, so that for every $U\subset M$ relatively compact open set, the number 
\be 
\#_{X\in W}(U):=\#(X^{-1}(W)\cap U)
\ee
is almost surely finite (if $W$ is not closed, this number can be $+\infty$) and it is a continuous function with respect to $X\in \{f\in \coo1MN \colon f\transv W\}$, thus it defines an integer valued random variable. Now, Lemma \ref{lemma:meas} below guarantees that its mean value
\be 
\E\#_{X\in W}(U)=\E\{\#(X^{-1}(W)\cap U)\}
\ee
can be extended to a Borel measure on $M$.

\begin{lemma}\label{lemma:meas}
Let $X\colon M\to N\supset W$ satisfy \emph{i-iv}. For any $A\in \mathcal{B}(M)$, the number $\#_{X\in W}(A)=\#(X^{-1}(W)\cap A)$ is a measurable random variable 
and the set function
\be 
\E\#_{X\in W}\colon\mathcal{B}(M)\to [0,+\infty],\qquad A\mapsto \E\#_{X\in W}(A)
\ee
is a Borel (not necessarily finite) measure on $M$.
\end{lemma}

The proof of this Lemma is postponed to Section \ref{sec:proflemmameas}. At this point, a couple of curiosities about this measure naturally arise: 
when is it a Radon measure\footnote{A Borel measure that is finite on compact sets.}? When is it absolutely continuous (in the sense of Definition \ref{def:abscont})? 
In this paper we are going to address these questions giving sufficient conditions for $\E\#_{X\in W}$ to be an absolutely continuous Radon measure and a formula to compute it in this case.

\section{KROK hypotheses and the main result.}\label{sec:main}

By considering the following particularly simple examples, that we should always bear in mind, we can observe that the setting i-iv described above is far too general to allow to give a yes/no answer to the questions raised at the end of the last subsection.
\begin{enumerate}[$\bullet$]
\item Let $X$ be \emph{deterministic}, in the sense that it is constantly equal to a function ${f\transv W}$. Then $\E\#_{X\in W}$ is the counting measure of the set $f^{-1}(W)$. This measure is Radon if and only if the set has no accumulation point, which is a consequence of transversality when $W$ is closed. In this situation the only case where $\E\#_{X\in W}$ is absolutely continuous is if $f^{-1}(W)=\emptyset$.
\item Let $\dim M=0$, i.e. $M=\{p_i\}_{i\in\N}$, and let $W\subset N$ be an open subset. In this case $X$ is a random element in the product space $X\in N^M$ and it can be easily checked that
\be\label{eq:stupid} 
\E\#\{p\in M\colon X_p\in W\}=\sum_{p\in M}\P\{X_p\in W\}.
\ee
This is a rather stupid case, however, the above formula \eqref{eq:stupid} is close in spirit to the one we are going to prove (in fact it is a special case of Theorem \ref{thm:main}), in that the right hand side depends only on the marginal probabilities of the random variables $\{X_p\}_{p\in N}$.
\item Let $X\colon M\to M\times M$ be the map $p\mapsto (p,\xi)$, for some random element $\xi\in M$ and let $W=\Delta$ be the diagonal. Then the measure $\E\#_{X\in W}$ is the law of $\xi$. Since the hypotheses i-iv are satisfied for every random variable $\xi$, this example shows that certainly every Borel probability measure on $M$ can be realized in this way (it is more difficult to realize an arbitrary measure with total mass greater than $1$).
\item Let $G$ be a group. Let $X\colon M\to E$ be a random section of a $G$-equivariant\footnote{Meaning that $G$ acts on the left on both $E$ and $M$ and the action commutes with the projection: $\pi(g\cdot x)=g\cdot \pi(x)$, for any $x\in E$ and $g\in G$. Thus, the function $gXg^{-1}$ such that $p\mapsto g\cdot X(g^{-1}\cdot p)$ is a section.} bundle $E\to M$
such that $gXg^{-1}$ has the same law (on the space of $\mC^1$ sections of $E$) as $X$, then the measure $\E\#_{X\in W}$ is $G$-invariant. This condition, in many situations, implies that $\E\#_{X\in W}$ is a constant multiple of the volume measure of some Riemannian metric on $M$ and therefore it is absolutely continuous. 
\item Let $M\subset \R^m$ be an open subset, $W=\{t\}$ a point of $N=\R^n$ and $X\colon M\to \R^n$  a $\mC^1$ Gaussian random field such that $X(p)$ is non degenerate for every $p\in M$. Then Kac-Rice formula \eqref{eq:KR} holds (see \cite{AdlerTaylor}) 
\end{enumerate}
To be able to say something meaningful we need to restrict our field of investigation. We will now make a series of assumption on the random map $X$ and on the submanifold $W$ under which the measure $ \E\#_{X\in W}\colon \mathcal{B}(M)\to [0,+\infty]$ is absolutely continuous and we can write a formula for its density. 
In doing so, one of our aim is to propose a setting that is easy to recognize in contexts involving differential topology and smooth random maps.
Although such hypotheses do not reach the highest level of generality in which Kac-Rice formula holds (see \cite{AdlerTaylor}), they describe a much more general setting than that in which the random map is assumed to be Gaussian and at the same time allow to give a proof whose simplicity is comparable to those for the Gaussian case.

\begin{defi}\label{def:krok}
 Let $M$ and $N$ be two smooth manifolds ($\mC^\infty$ and without boundary) of dimension $m$ and $n$. Let $W\subset N$ be a smooth submanifold (without boundary) and $X\colon M\to N$ a random map. We will say that $(X,W)$ is a \emph{KROK}
  couple if the following hypotheses are satisfied.
\begin{enumerate}[ , wide, left=0pt]
    \item Properties of $X$:
\begin{enumerate}[(i)]
\item\label{itm:krok:2} \boldmath{$X\in \mC^1$: }\unboldmath

$X\colon M\to N$ is a random $\mC^1$ map, i.e. it represents a Borel probability measure $[X]$ on the topological space $\coo1MN$ endowed with the weak Whitney $\mC^1$ topology (see \cite{Hirsch}).
\item\label{itm:krok:3} \boldmath{$d[X(p)]=\rho_{X(p)}dS_p$: }\unboldmath 

Let $N$ be endowed with a Riemannian metric. Assume that for each $p\in M$, the probability measure $[X(p)]$ is absolutely continuous with respect to the Riemannian volume density of a certain smooth submanifold $S_p$. 
In other words, there exists a measurable function $\rho_{X(p)}\colon S_p\to [0,+\infty]$ such that
    \be 
    \E\{F(X(p))\}=\int_{S_p}F(q)\rho_{X(p)}(q)dS_p(q).
    \ee
    for every Borel function $F\colon N\to\R$. ($\rho_{X(p)}$ is allowed to vanish on $S_p$.)
\end{enumerate}
\item Properties of $W$ (transversality):
\begin{enumerate}[(i)]
\setcounter{enumii}{2}
\item\label{itm:kroktransvas} \boldmath{$X\transv W$ }\unboldmath

 almost surely.
\item\label{itm:krok:6} \boldmath{$S_p\transv W$ }\unboldmath  

for every $p\in M$. 
\item\label{itm:krok:4} \boldmath{$\dim X^{-1}(W)=0$: }\unboldmath 

The codimension of $W$ is $m=\dim M$.
\end{enumerate}
\item Continuity properties:
\begin{enumerate}[(i)]
\setcounter{enumii}{5}
\item\label{itm:krok:7} \boldmath{$S_{(\cdot)}\in\mC^\infty$: }\unboldmath

The set $\mathcal{S}=\{(p,q)\colon q\in S_p\}\subset M\times N$ is a 
closed smooth submanifold. 
This, together with \ref{itm:krok:6}, implies that the set defined as $\MW  =\{(p,q)\in M\times N\colon q\in S_p\cap W\}$ is a smooth manifold. 
\item \boldmath{$\rho_{X(\cdot)}(\cdot)\in\mC^0$: }\unboldmath

The function $M\times N\ni (p,q)\mapsto \rho_{X(p)}(q)\in \R_+$ is continuous at all points of $\MW  $.
\item{\label{itm:krokEcont}}\boldmath{$[X|X(\cdot)=\cdot]\in\mC^0$: }\unboldmath 

There exists a regular conditional probability\footnote{See \cite{dudley}, or wait for the next subsection \ref{intro:krok9}, for an explanation of this concept.} ${[X|X(p)=\cdot]}$ such that the expectation
$M\times N\ni (p,q)\mapsto \E\big\{\a(X,p)J_pX\big|X(p)=q\big\}$ is continuous \emph{at}\footnote{i.e. continuous at every point of $\MW  $.
} $\MW  $, for any bounded and continuous function $\a\colon \coo1MN\times M\to \R$. 
\end{enumerate}
\end{enumerate}
In the following, we will refer to this hypotheses as KROK.\ref{itm:krok:2}, KROK.\ref{itm:krok:3}, etc\dots
\end{defi}

\begin{thm}[Generalized Kac-Rice formula]\label{thm:main}
Let $X\colon M \to N$ be a random $\mC^1$ map between two Riemannian manifolds and let $W\subset N$ such that $(X,W)$ is a KROK couple. Then for every Borel subset $A\subset M$ we have
\begin{multline}\label{eq:formrho}
\E\#_{X\in W}(A)=\int_A \rho_{X\in W}(p)dM(p)= \\=\int_{A}\left(\int_{ S_p\cap W}\E\left\{J_pX\frac{\sigma_q(X,W)}{\sigma_q(S_p,W)}\bigg|X(p)=q\right\}\rho_{X(p)}(q)d(S_p\cap W)(q)\right)dM(p),
\end{multline}
where $d(S_p\cap W)$ and $dM$ denote the volume densities of the corresponding Riemannian manifolds and $J_pX$ is the jacobian of $X$ (see Definition \ref{def:jacob}); besides, $\sigma_q(X,W)$ and $\sigma_x(S_p,W)$ denote the ``angles'' (in the sense of Definition \ref{defi:angle}) made by $T_qW$ with, respectively, $d_pX(T_pM)$ and $T_qS_p$.
\end{thm}
\begin{remark}[Special cases]
 The standard Kac-Rice formula corresponds to the situation when $S_p=N$ and $W=\{q\}$. Here, the term $\frac{\sigma_q(X,W)}{\sigma_q(S_p,W)}$ disappears, since both angles are equal to $1$. 

 When $T_qS_p=T_qN$, then $\sigma(S_p,W)=1$.

 When $S_p\cap W=\{q\}$, there is no integration $\int_{S_p\cap W}$.

 When $W$ is an open subset, then $\sigma_q({X,W})=1$ and $\sigma_q(S_p,W)=1$, but the dimension hypothesis KROK.\ref{itm:krok:4} falls, unless $m=0$. In such case, the above formula reduces to equation \eqref{eq:stupid}.

 If $S_p=\{f(p)\}$ it means that $X=f$ is deterministic. Unless $m=0$, the couple $(f,W)$ is KROK only if  $f^{-1}(W)=\emptyset$ because of KROK.\ref{itm:krok:3}. Indeed, as we previously observed in the first of the examples above, in the deterministic case the measure $\E\#_{f\in W}=\#_{f\in W}$ is not absolutely continuous, for obvious reasons, unless it is zero.
This is one of the reason why we can't change KROK.\ref{itm:krok:6} into ``$S_p\transv W$ for a.e. $p$''.
\end{remark}
In formula \eqref{eq:formrho} above, $\rho_{X\in W}\in B^+(M)$ is a not necessarily finite Borel measurable function $\rho_{X\in W}\colon M\to [0,+\infty]$. It is precisely the Radon-Nykodim derivative of $\E\#_{X\in W}$ with respect to the Riemannian volume measure of $M$.

We can write the above formula in another equivalent way, using the jacobian density
\be\label{eq:intro:indensity}
\delta_pX=J_pXdM(p)\in \Delta_pM
\ee
defined in \eqref{eq:indensity}, which is a more natural object in that it doesn't depend on the Riemannian  structure of $M$. 

By using the notion of \emph{density} we can write a more intrinsic formula, without involving a Riemannian metric on $M$. 
A density is a section of the vector bundle $\Delta M=\wedge^m(T^*M)\otimes L$ obtained by twisting the bundle of top degree forms with the orientation bundle $L$ (see \cite[Section 7]{botttu}). 
  The peculiarity of densities is that they can be integrated over $M$ in a canonical way. In particular, the volume density of a Riemannian manifold $M$ is a density in all respects and we denote it by $p\mapsto dM(p)\in\Delta_pM$. We collected some details and notations regarding densities in Appendix \ref{app:densities}.
Although the function $\rho_{X\in W}$ appearing in \eqref{eq:formrho} depends on the Riemannian structure of $M$, the expression $\rho_{X\in W}(p)dM(p)=:\delta_{X\in W}(p)$ defines a positive measurable density ($\delta_{X\in W}\in L^+(M)$) that is independent from the metric. This is clarified in the subsection  \ref{sec:indepmetr}, but it is actually a consequence of Theorem \ref{thm:main}, since the left hand side of \eqref{eq:formrho} depends merely on the ``set theoretic nature'' of the objects in play.

\begin{cor}[Main Theorem/Definition]\label{thm:maindens}
Let $X\colon M \to N$ be a random $\mC^1$ map between two Riemannian manifolds and let $W\subset N$ such that $(X,W)$ is a KROK couple. Then the measure $\en$ is absolutely continuous on $M$ with density $\delta_{X\in W}\in L^+(M)$ defined as follows.
\begin{multline}\label{eq:formdel}
\delta_{X\in W}(p):=\left(\int_{ S_p\cap W}\E\left\{\delta_pX\frac{\sigma_q(X,W)}{\sigma_q(S_p,W)}\bigg|X(p)=q\right\}\rho_{X(p)}(q)d(S_p\cap W)(q)\right)\in\Delta_pM,
\end{multline}
where $d(S_p\cap W)$ denotes the volume densities of the corresponding Riemannian manifold and $\delta_pX$ is the Jacobian density of $X$ defined as in \eqref{eq:intro:indensity}; besides, $\sigma_q(X,W)$ and $\sigma_q(S_p,W)$ denote the ``angles'' (in the sense of Definition \ref{defi:angle}) made by $T_qW$ with, respectively, $d_pX(T_pM)$ and $T_qS_p$. Therefore, for every Borel subset $A\subset M$,
\be 
\en(A)=\int_A \delta_{X\in W}.
\ee
\end{cor}
\begin{remark}\label{rem:altforms}
Other alternative forms of the above formula can be obtained from the identities: 
\be \delta_p X\sigma(X,W)=J(\Pi_{TW^\perp}\circ dX)dM(p)=|X^*\nu|,
\ee where $\nu=\nu^1\wedge \dots \wedge \nu^r$ for some orthonormal basis on $TW^\perp$. The first identity follows from Proposition \ref{prop:chi}, while in the second we are representing the density element as the modulus of a differential form, via the function $|\cdot |\colon \wedge^mT^*_pM\to \Delta_pM$, defined in Appendix \ref{app:densities}.
\end{remark}
\begin{remark}
Theorem \ref{thm:main} does not guarantee that $\E\#_{X\in W}$ is a Radon measure. This condition corresponds to the local integrability of the density: $\delta_{X\in W}\in L^1_{loc}(M)$, while in general $\rho_{X\in W}\in B^+(M)$ and $\delta_{X\in W}\in L^+(M)$. This issue, as it will be clear from Theorem \ref{thm:maingraph} below, comes from the integration over $W\cap S_p$, which can be a non-compact submanifold.
\end{remark}

The strength of this formula is that, exactly as in the standard Kac-Rice case (when $W$ is a point), although the left hand side depends, a priori, on the whole probability $[X]$ on $\coo 1MN)$, the right hand side depends only on the pointwise law of the first jet $j^1_pX=(X(p),d_pX)$. This is a significant simplification in that the former is the joint probability of all the random variables $\{X(p)\}_{p\in M}$, while the latter is the collection of the marginal probabilities $\{[j^1_pX]\}_{p\in M}$, which is a simpler data. 

\subsection{Explanation of condition KROK.\ref{itm:krokEcont}}\label{intro:krok9}
Given a random element $X\in \mC^1(M,N)$ as in KROK.\ref{itm:krok:3} and a point $p \in M$,
a \emph{regular conditional probability}\footnote{
See \cite{dudley} or \cite{Erhan}. In the latter the same object is called a \emph{regular version of the conditional probability}.
} for $X$ given $X(p)$ is a function 
\be
\P\{X\in \cdot\ |X(p)=\cdot\}\colon \mathcal{B}(\mC^1(M,N))\times N\to [0,1],
\ee
that satisfies the following two properties.
\begin{enumerate}[(a)]
\item For every $B\in\mathcal{B}(\mC^1(M,N))$, the function $\P(X\in B|X(p)=\cdot)\colon N\to [0,1]$ is Borel and for every $V\in \mathcal{B}(N)$, we have
\be 
\P\{X\in B ; X(p)\in V\}=\int_V\P(X\in B|X(p)=q\}d[X(p)](q).
\ee
\item For all $q\in N$, $\P\{X\in \cdot\ |X(p)=q\}$ is a probability measure on $\mC^1(M,N)$.
\end{enumerate} 
The notation that we use is what we believe to be the most intuitive one and consistent with the other used in this paper. Given $p\in M$ and $q\in N$, we write ${[X|X(p)=q]}:={\P\{X\in \cdot\ |X(p)=q\}}$ for the probability measure and ${\E\{\a(X)|X(p)=q\}}$ for the expectation/integral of a function $\a\colon \mC^1(M,N)\to \R$ with respect to the probability measure ${[X|X(p)=q]}$.

The fact that the space $\mC^1(M,N)$ is Polish ensures that, for every $p\in M$, a regular conditional probability measures $[X|X(p)=\cdot]$ exists (see \cite[Theorem 10.2.2]{dudley}) and it is unique up to $[X(p)]$-a.e. equivalence on $N$. However, strictly speaking, it is not a well defined function, although the notation can mislead to think that. 

In our case such ambiguity may be traumatic, since we are interested in evaluating $\E\{\dots|X(p)=q\}$ for $q\in W$ which, under KROK.\ref{itm:krok:3} and KROK.\ref{itm:krok:6}, is negligible for the measure $[X(p)]$, i.e. $\P\{X(p)\in W\}=0$. Therefore it is essential to choose a regular conditional probability that has some continuity property at $W$, otherwise formula \eqref{eq:formrho} doesn't make sense, as well as all of its siblings. This is the motivation for the hypothesys KROK.\ref{itm:krokEcont}.

\begin{remark}
In the common statements of Kac-Rice formula ($W$ is a point), one finds the analogous hypothesis that there exists a density for the measure $[J_pX|X(p)=q]$ that is continuous at $q\in W$ (see for instance \cite{AdlerTaylor}). This is different than KROK.\ref{itm:krokEcont}, in that we don't need to assume that $[J_pX|X(p)=q]$ has a density.
\end{remark}

To have a complete perspective, let us rewrite the hypothesis KROK.\ref{itm:krokEcont} in a more suggestive way.
Let $\mathcal{F}=\mC^1(M,N)$. 
Consider the space $\mathscr{P}(\mathcal{F})$ of all Borel probability measures on $\mathcal{F}$, endowed with the narrow topology (also called weak topology: see \cite{dtgrf}), namely the one induced by the inclusion $\mathscr{P}(\mathcal{F})\subset \mC_b(\mathcal{F})^*$. A sequence of measures $[X_n]$ converges in this topology: $X_n\nrw X$, if and only if $\E\{\a(X_n)\}\to\E\{\a(X)\}$ for every $\a\in \mC_b(\mathcal{F})$, see \cite{Parth,Billingsley}.

Let $[X|X(p)=\cdot]$ be a regular conditional probability. Consider, for each $p\in M$ and $q\in N$, the probability $\mu(p,q)\in\mathscr{P}(\mathcal{F})$, given by
\be 
\mu(p,q)(B)=\E\{1_B(X)J_pX|X(p)=q\}=\int_B (J_pf) d[X|X(p)=q](f).
\ee
In other words, $\mu(p,q)=J_p\cdot [X|X(p)=q]$ is the multiplication of the measure ${[X|X(p)=q]}$ by the positive function $J_p\colon \mathcal{F}\to \R$, such that $J_p(f)=J_pf$.
This defines a function $\mu\colon M\times N\to \mathscr{P}(\mathcal{F})$.
\begin{prop}\label{prop:CazzoDuro}
KROK.\ref{itm:krokEcont} holds if and only if $\mu$ is continuous at $\MW  $.
\end{prop}
\begin{proof}
If the function $\a$ in KROK.\ref{itm:krokEcont} was not allowed to depend on $p\in M$, this fact would be obvious from the definition of the topology on $\mathscr{P}(\mathcal{F})$.  This, in particular, implies the \emph{only if} part of the statement.

Let us show the converse.
Fix $\a\in \mC_b(\mathcal{F}\times M)$ and let $(p_n,q_n)\in M\times N$ be any  sequence of points  such that $(p_n,q_n)\to (p,q)\in \MW  $. Then $\mu_n:=\mu(p_n,q_n)\nrw \mu=\mu(p,q)$ in $\mathscr{P}(\mathcal{F})$. By the Skorohod theorem (see \cite{Billingsley, Parth}) there exists a representation $\mu_n=[Y_n]$, $\mu=[Y]$ for some sequence of random elements $Y_n,Y\in \mathcal{F}$ such that $Y_n\to Y$ in $\mathcal{F}$ almost surely. It follows that $\a(Y_n,p_n)\to \a(Y,p)$ almost surely and since $\a$ is bounded, we conlcude by dominated convergence that $\E\{\a(Y_n,p_n)\}\to \E\{\a(Y,p)\}$. This concludes the proof, since for all $n\in\N\cup\{\infty\}$:
\be 
\E\{\a(X,p_n)J_{p_n}X|X(p_n)=q_n\} =\int_{\mathcal{F}}\a(f,p_n)d\mu(p_n,q_n)(f)=\E\{\a(Y_n,p_n)\}.
\ee 
\end{proof}

When dealing with a KROK couple $(X,W)$, we will always implicitely assume that the function $(p,q)\mapsto [X|X(p)=q]$ is chosen among those for which $\mu$ is continuous at $\MW$. This arbitrary choice does not influence the final result, in that formula \eqref{eq:mainformula} depends only on $\mu|_{\MW}$.
\subsection{A closer look to the density}
In order to have a better understanding of the density $\delta_{X\in W}$, it is convenient to adopt a more general point of view. Let us consider, for any $V\subset M\times N$, the random number $\#_{\Gam}(V)=\#(\Gam\cap V)$, where $\Gam$ is the graph of the map $X|_{X^{-1}(W)}$, that is:
\be 
\Gamma(X,W):=\left\{(p,q)\in M\times W\colon X(p)=q\right\}, \quad \#_{\Gam}(V):=\#(\Gam\cap V).
\ee
The expectation $\E\#_{\Gam}$ of such random variable can be proven\footnote{The argument is exactly the same as that used to prove Lemma \ref{lemma:meas}} to be a Borel measure on $M\times N$ and by viewing it as an extension of the measure $\E\#_{X\in W}$, we can deduce Theorem \ref{thm:main} from the following slightly more general result.
\begin{thm}\label{thm:maingraph}
Let $(X,W)$, be a KROK couple, then the measure $\E\#_{\Gam}$ is supported on $\MW  =\{(p,q)\in M\times W\colon q\in S_p\}$ and it is an absolutely continuous measure on it, with a continuous density
\be\label{eq:gamdens}
\delta_{\Gam}(p,q)=\E\left\{J_pX{\sigma_q(X,W)}\bigg|X(p)=q\right\}\rho_{X(p)}(q)\delta \MW  (p,q),
\ee
where $\delta \MW  $ is the density on $\MW  $ defined below\footnote{We are implicitely making the identification $\Delta_{(p,q)}\MW  \cong \Delta_q(S_p\cap W)\otimes \Delta_pM$. By the KROK hypotheses \ref{def:krok}, $\MW  $ is a smooth submanifold of $M\times N$. However, $\delta {\MW  }$ is not the volume density of the metric induced by inclusion in the product Riemannian manifold $M\times N$.}.
\be 
\delta \MW  (p,q)=\frac{1}{\sigma_q(S_p,W)}d(S_p\cap W)(q)dM(p),
\ee  
\end{thm}
Precisely, this means that $ \E\#_{\Gam}(V)= \int_{V\cap \MW  }\delta_{\Gam}$,
for any Borel subset $V\subset M\times N$. In particular, if $V=A\times B$ we get
\be\label{eq:AperB}
\E\#_{X\in W\cap B}(A)=\E\#_{\Gam}(A\times B)=\int_A\left(\int_{B\cap W\cap S_p}\delta_{\Gam} (p,q)dq\right) dp.
\ee
for every $A\in \mathcal{B}(M)$ and $B\in \mathcal{B}(N)$.
\begin{remark}
The density $\delta_{X\in W}\in L^+(M)$ of the measure $\E\#_{X\in W}$ is obtained from the continuous density $\delta_{\Gam}\in \mathscr{D}^0(\MW  )$, by integration over the fibers of the projection map $\MW  \to M$.
\be\label{eq:twodens}
\delta_{X\in W}(p,q)=\int_{S_p\cap W}\delta_{\Gam}(p,q)dq.
\ee
This has to be intended as follows. The splitting $T_{(p,q)}\MW  \cong T_q(S_p\cap W)\oplus T_pM$ yields a natural identification $\Delta_{(p,q)} \MW  \cong \Delta_q (S_p\cap W)\otimes \Delta_p M$, allowing to define the \emph{partial integral} $\int_{S_p\cap W}\colon \mathscr{D}(\MW  )\to \Delta_p(M)$. 
\end{remark}
\begin{remark}
If $\mu=\E\#_{\Gamma(X,W)}$ on $M\times N$, then the integral of a measurable function $f\colon M\times N\to \R$ is given by the formula
\be 
\int fd\mu=\int_{\MW  }f\cdot \delta_{\Gamma(X,W)}=\E\left\{\sum_{p\in X^{-1}(q),\ q\in W}f(p,q)\right\}.
\ee
The proof of this fact, by monotone convergence, can be reduced to the case of characteristic functions $f=1_{A\times B}$, case in which  the formula is equivalent to equation \eqref{eq:AperB}.
\end{remark}

\subsection{The case of fiber bundles and the meaning of $\delta \MW$}

Let us consider the situation in which $\pi: N\to M$ is a smooth fiber bundle with fiber $S_p=\pi^{-1}(p)$
and let $W\subset N$ be a smooth submanifold such that $W\transv S_p$ for every $p\in M$.
Assume that $X\colon M\to N$ is a $\mC^1$ random section of $\pi$ 
and that $(X,W)$ is a KROK couple. 

In this case, the projection $\MW  =\{(p,q)\in M\times W\colon q\in S_p\cap W\}\to W$ is a bijection and we can identify the two spaces $\MW  \cong W$. Assume that both manifolds are endowed with Riemannian metrics in such a way that $\pi$ is a Riemannian submersion, meaning that the next map is an isometry\footnote{Such pair of metrics, always exists. To construct them, first define any metrics on $M$ and $N$. Then consider the subbundle $H\subset TN$ given by the orthogonal complement of the vertical one, namely $H_q=\ker(d_q\pi)^\perp$ (alternatively, take $H$ to be any Ehresmann connection). Now modifiy the metric on $H_q$ by declaring the map $d_q\pi\colon H_q\to T_pM$ a linear isometry.},
\be 
d_q\pi \colon \ker(d_q\pi)^\perp\to T_pM .
\ee
Then the formula \eqref{eq:gamdens} for the density, given in Theorem \ref{thm:maingraph} becomes easier and more meaningful.
\begin{thm}\label{thm:megafica}
Let $\pi\colon N\to M$ be a fiber bundle and Riemannian submersion. Let $(X,W)$ be a KROK couple such that $S_p=\pi^{-1}(p)$ for each $p\in M$. Then $\delta_{\Gam}$ is the continuous density on $W$ defined by the formula
\be\label{eq:megafica} 
\delta_{\Gam}(q)=\E\left\{J_{\pi(q)}X\sigma_q(X,W)\bigg|X(\pi(q))=q\right\}\rho_{X(\pi(q))}(q)dW(q).
\ee
\end{thm}
This is due to the fact that in this case we have $\delta \MW  =dW$.

\subsection{Independence on the metric}\label{sec:indepmetr}
It is important to note that the Riemannian structure on $N$ is just an auxiliary object that allows to write the formulas \eqref{eq:formdel}, \eqref{eq:gamdens}. In fact, the densities $\delta_{\Gam}$ and $\delta_{X\in W}$ must be independent of the chosen metric on $N$, since the corresponding measures have nothing to do with the Riemannian structure. 
Indeed, let us define a notation for the following expression:
\be 
d_pX\ \underset{T_qS_p}{\overset{T_qW}{\llcorner}}\ \delta_{X(p)}(q):=J_pX\frac{\sigma_q(X,W)}{\sigma_q(S_p,W)}\rho_{X(p)}(q)d(S_p\cap W)(q)dM(p).
\ee
This defines a density element in $\Delta_{(p,q)}\MW  $ depending only on the transverse vector subspaces $T_qW, T_qS_p\subset T_qN$, on the linear map $d_pX\colon T_pM\to T_qN$, and on the density element $\delta_{X(p)}(q)=\rho_{X(p)}(q)dS_p(q)\in \Delta_{q}S_p$. 
With this notation we can give a totally intrinsic version of formula \eqref{eq:gamdens}:
\be 
\delta_{\Gam}(p,q)=\E\left\{d_pX\ \underset{T_qS_p}{\overset{{T_qW}}{\llcorner}}\ \delta_{X(p)}(q)\ \Bigg|X(p)=q\right\}.
\ee

\section{The Gaussian case}\label{sec:gausscase}
\subsection{Smooth Gaussian random sections}\label{sec:gausscaseintro}
The first type of random maps that one encounters in random geometry are, with a very high probability, \emph{Gaussian random fields}, which are random maps $X\colon M\to \R^n$, whose evaluations at points are Gaussian (we refer to \cite{dtgrf} for a systematic treatment of smooth Gaussian random fields). In this section we are going to deal with a slight generalization of this concept, namely Gaussian random sections of a vector bundle. 

Precisely, let $\pi\colon E\to M$ be a smooth vector bundle of rank $s$ over a smooth $m$-dimensional manifold $M$ and let $X\colon M\to E$ be a random $\mC^r$ section of $\pi$. The random section $X$ is said to be \emph{Gaussian} if for any finite set of points $p_1,\dots, p_j$ the random vector 
\be 
\left(X(p_1), \dots , X(p_j)\right)\in E_{p_1}\oplus \dots \oplus E_{p_j}
\ee
is Gaussian. For simplicity in this paper we will assume all Gaussian variables to be centered, although this assumption is not necessary. Taking up the notation of \cite{dtgrf}, we will denote as $\goo rME$ the set of $\mC^r$ Gaussian Random Sections (GRS) of a vector bundle $E$ over $M$. As for every Gaussian stochastic process, a GRS $X\in \goo rME$ is completely determined by its covariance tensor, which is the section $K_X\colon M\times M\to  \text{pr}_1^*E\otimes \text{pr}_2^*E$, defined by the following identity holding for every $\lambda_1,\lambda_2 \in E^*$:
\be 
K_X(p,q)\langle\lambda_1,\lambda_2\rangle=\E\left\{\lambda_1(X(p))\lambda_2(X(q))\right\}.
\ee
In particular, $K_X(p,p)=K_{X(p)}$ is a symmetric, semipositive, bilinear form on $E_p^*$.
\begin{defi}
 If $K_{X(p)}$ is positive definite (equivalently, $\spt[X(p)]=E_p$) for every $p\in M$, then we say that $X$ is \emph{non-degenerate}. 
 \end{defi}
In this case, if moreover $E$ is endowed with a bundle metric $g\colon E\cong E^*$, 
one can define the inverse covariance tensor, which is a bilinear form on $E_p$ that we denote by $K_{X(p)}^{-1}\left\langle \cdot,\cdot\right\rangle$. Then
we have $d[X(p)]=\rho_{X(p)}dE_p$ (in the sense of point \ref{itm:krok:3} of Definition \ref{def:krok}), where $dE_p$ is the Riemannian volume density of the fiber $E_p$ and
\be\label{eq:dgauvec}
\rho_{X(p)}(x)= \frac{\exp\left(-\frac{1}{2}K_{X(p)}^{-1}\left\langle x,x\right\rangle\right)}{\pi^\frac s2\sqrt{\det\left(K_{X(p)}\right)}}.
\ee
(The same formula is true in coordinates, if $K_{X(p)}$ denotes the covariance matrix.)

We want to apply Theorem \ref{thm:main} to compute the average number of points $p\in M$ such that $X(p)$ belongs to a given smooth submanifold $W\subset E$ of the total space $E$, having codimension $m$. In the Gaussian case it is particularly easy to verify the hypotheses of the theorem, indeed with the help of the (Gaussian) Probabilistic Transversality theorem from \cite{dtgrf}.
\begin{thm}[Theorem $7$ from \cite{dtgrf}]\label{thm:probtransv}
Let $X\in\goo \infty ME$. Assume that for every $p\in M$
\be 
\spt[X(p)]=E_p.
\ee
Then for any smooth submanifold $W\subset E$, we have that $\P\{X\transv W\}=1$.
\end{thm}
From this, one deduces easily that the couple $(X,W)$ is KROK (Definition \ref{def:krok}) if $X$ is non-degenerate and $W\transv E_p$ for every $p\in M$. The only non obvious condition to check is KROK.\ref{itm:krokEcont}, which turns out to be a consequence of the Gaussian regression formula. This argument is used also in the proof of the standard Kac-Rice formula given in \cite{Wschebor}. Here, it is proved in Lemma \ref{lem:gioiellino}.
\begin{remark}
If $X\in \goo 1 ME$ is non-degenerate and $W\subset E$ is a submanifold such that $W\transv E_p$ for every $p\in M$, then $(X,W)$ is a KROK couple provided that $X\transv W$ almost surely. In the smooth case, the last hypothesis is redundant, due to Theorem \ref{thm:probtransv}.
This result holds only for sufficiently smooth fields, as well as its finite dimensional analogue, because it relies on Sard's Theorem. For this reason, we chose to focus on smooth GRS.
\end{remark}
The following theorem is the translation of the main Theorem \ref{thm:main} in the Gaussian setting. Although it is stated in a simpler way, it actually holds whenever the couple $(X,W)$ is KROK.
\begin{thm}\label{thm:maingau}
Let $\pi\colon E\to M$ and let $X\in\goo \infty ME$ be a non-degenerate $\mC^\infty$ Gaussian random section. Let $W\subset E$ be a smooth submanifold of codimension $m$ such that $W\transv E_p$ for every $p\in M$ and let $W_p=W\cap E_p$. Let the total space of $E$ be endowed with a Riemannian metric that is euclidean on fibers. 
Then for any Borel subset $A\subset M$
\begin{multline}\label{eq:formgaus}
\E\#_{X\in W}(A)= \int_{A}\delta_{X\in W}=\\
=\int_{A}\int_{W_p}\E\left\{J_pX\frac{\sigma_x(X,W)}{\sigma_x(E_p,W)}\bigg|X(p)=x\right\}\frac{e^{\left(-\frac{1}{2}K_p^{-1}
\left\langle x,x\right\rangle\right)}}{\pi^\frac s2\sqrt{\det(K_p)}}dW_p(x)dM(p).
\end{multline}
Here $K_p=K_X(p,p)$
; $dW_p$ and $dM$ denote the volume densities of the corresponding Riemannian manifolds; $J_pX$ is the Jacobian of $d_pX:T_pM\to T_xE$ (see Definition \ref{def:jacob}); besides, $\sigma_x(X,W)$ and $\sigma_x(E_p,W)$ denote the ``angles'' (in the sense of Definition \ref{defi:angle}) made by $T_xW$ with, respectively, $d_pX(T_pM)$ and $T_xE_p$.
\end{thm} 
We say that a Riemannian metric on the vector bundle $\pi\colon E\to M$ is Euclidean on fibers when the metric induced on each fiber $E_p\subset E$ is a vector space metric, meaning that $E_p$ is linearly isometric to the Euclidean space $\R^s$, as Riemannian manifolds. 

Such metric always exists on any vector bundle. The natural way to construct one is by defining a metric on $M$, a vector bundle metric on $E$ and an Ehressmann connection $H$ for the bundle $\pi$, that is: a vector subbundle of $TE$, the \emph{horizontal} bundle, such that $d_q\pi|_{H_q}\colon H_q\to T_{\pi(q)}M$ is a bijection. Then, the metric on $E$ is defined by declaring the implied isomorphism $T_qE\cong T_{\pi(p)}M\oplus_{\perp} E_q$ to be an isometry. A metric defined with this procedure is Euclidean on fibers, but it also make $\pi\colon E\to M$ a Riemannian submersion.
\begin{defi}
Let $\pi\colon E\to M$ be a vector bundle, such that $E$ is endowed with a metric constructed via a connection, with the above procedure. Then, we say that $\pi$ is a \emph{connected Riemannian bundle} or that it has a \emph{connected Riemannian metric}. We will say \emph{linearly connected} if the connection is linear.\footnote{A connection $H\subset E$ is linear if $H_{\lambda q}=d_qL_{\lambda}(H_q)$ for every $\lambda \in\R$, where $L_\lambda$ denotes the scalar multiplication. in this case, the operator $\nabla\colon \mC^\infty(M|E)\to \mC^{\infty}(M|E\otimes T^*M)$ satisfies the Leibnitz rule and thus it defines a covariant derivative.}
\end{defi}

Notice that in the case of Theorem \ref{thm:maingau} it is easy to see that $\MW $ is diffeomorphic to $W$ and $\delta_{\Gam}$ is a continuous density on it, although $\delta_{X\in W}\in L^+(M)$. Thus, by endowing $E$ with a connected Riemannian metric, Theorem \ref{thm:megafica} implies the following more natural formula. 
\begin{cor}\label{cor:maingauW}
In the same setting of Theorem \ref{thm:maingau}, assume that $\pi\colon E\to M$ is endowed with a connected Riemannian metric. Let $V\subset E$ be any Borel subset, then there exists a smooth density $\delta_{\Gam}\in \mathscr{D}^\infty(W)$ such that
\bega\label{eq:gaudelgamma}
\E\#\left(W \cap V\right)&=\int_{W\cap V}\delta_{\Gam}\\
&=\int_{W\cap V}
\E\left\{J_pX\sigma_x(X,W)\bigg|X(p)=x\right\}\frac{e^{\left(-\frac{1}{2}K_p^{-1}\left\langle x,x\right\rangle\right)}}{\pi^\frac s2\sqrt{\det(K_p)}}dW(x).
\eega
\end{cor}
\begin{remark}\label{rem:connectionForm}
If moreover $W$ is parallel, for the given connection, that is: $T_qW^{\perp}\subset E_q$, then at a point $q=X(p)$ we have
\be 
J_pX\sigma_q(X,W)=J\left(\Pi_{T_qW^\perp}\circ d_pX\right)=|\det\left(\Pi_{T_qW_p^\perp}\circ\left(\nabla X\right)_p\right) |,
\ee
where $W_p=W\cap E_p$ and $(\nabla X)_p\colon T_pM\to E_p$ is the vertical projection of $d_pX$. 
\end{remark}
We are (ab)using the symbol $\nabla$, since this notion coincides with that of \emph{covariant derivative}, in the case in which the connection is linear. 
Given that $W$ is transverse to the fibers of $\pi$,  one can always define a horizontal space $h_q=T_qW\cap (T_qW\cap T_qE_{\pi(q)})^\perp$ for each point $q\in W$. Then, if $h$ can be extended to  the whole $E$, it defines a connection (non linear, in general) $H$ for which $W$ is parallel. This construction is possible whenever $W\subset E$ is closed, by Tietze's extension theorem, but in general, there can be problems at $\overline{W}\- W$.

A particularly special case is when the connection is $\nabla=\nabla^X$ and the bundle metric on $E$ are the ones naturally defined by $X$ (see \cite{Nicolaescu2016}), namely $K_p$ is the dual metric and
\be \label{eq:nablaX}
\nabla^Xs:=Ds-\E\{DX|X=s\},
\ee
for any other connection $D$.
Since $\nabla^X$ is a metric connection in this case, it follows that for any Riemannian metric on $M$, a non-degenerate Gaussian random section defines a connected Riemannian structure on $E$. Moreover, since $(\nabla^XX)_p$ and $X(p)$ are independent, the formula in this case becomes much simpler.
\begin{cor}\label{cor:connectedform}
In the same setting of Theorem \ref{thm:maingau}, assume that $\pi\colon E\to M$ is endowed with the connected Riemannian structure defined by $X$. Let $W\subset E$ be parallel for this structure. Then
\bega\label{eq:connectedgaudelgamma}
\E\#\left(W \cap V\right)&=\int_{W\cap V}
\E\left\{|\det\left(\Pi_{T_xW_p^\perp}\circ\left(\nabla^X X\right)_p\right) |\right\}\frac{e^{-\frac{|x|^2}{2}}}{(2\pi)^{\frac s2}}dW(x).
\eega
\end{cor}

Thanks again to the Probabilistic Transversality theorem from \cite{dtgrf}, the above result immediately generalizes to the case of a Whitney stratified submanifold (see \cite{GoreskyMacPherson}) $W\subset E$ of codimension $m$, simply because the probability that $X(M)$ intersects the lower dimensional strata is zero, therefore one can replace $W$ with its smooth locus $W^{(s)}$, namely the stratum of dimension $s$. In this case we still write $\delta_{X\in W}$ for the density, in place of $\delta_{X\in W^{(s)}}$.

\subsection{Finiteness and continuity}

\begin{enumerate}[\textbf{Q.}$1$:]
\item \emph{Is $\E\#_{X\in W}$ a Radon measure?} This question has positive answer precisely when the density is locally integrable, that is $\delta_{X\in W}\in L^1_{loc}(M)$. Theorem \ref{thm:maingau} leaves open the possibility that the density $\delta_{X\in W}(p)=\int_{W_p}\delta_{\Gam}(p,x)dx$ is even infinite. 
\item  \emph{Is the function $X\mapsto \E\#_{X\in W}(A)$ continuous?}  Understanding this is really useful in those situations where one is interested in the asymptotic behavior of things, for instance when dealing with Kostlan polynomials (see \cite{mttrp}). 
\end{enumerate}


From Corollary \ref{cor:maingauW} it is clear that $\delta_{X\in W}$ is finite at least in the case in which $W$ has finite volume. However, this would not be satisfying, since in many possible applications, $W$ has infinite volume. For instance, when $W\subset E$ is a vector subbundle of $E$, in fact, we will see that the density is finite in this case.  On the other hand, it should be clear that, due to the natural additivity of the formula:
$
\delta_{X\in \cup_n W_n}=\sum_n \delta_{X\in W_n},$
there are cases in which $\delta_{X\in W}(p)=+\infty$.
The intuition behind this is that if $W$ is too much ``concentrated'' over the fiber over a point $p_0$, then the probability that $X(p)\in W$ for some point near $p$ is too big, resulting in having $\E\#_{X\in W}(O)=+\infty$ for some neighborhood $O\subset M$ of $p_0$.

To express such concept, we introduce the notion of \emph{sub-Gaussian concentration}. This will allow us to compare the magnitude of $W$ with that of Gaussian sections, by passing through the linear structure of the bundle.
\begin{defi}\label{def:subexpvol}
Let $\pi\colon E\to M$ be a linearly connected Riemannian vector bundle. Let $B_R\subset E$ be the subset of vectors $e\in E$ with length at most $R$ for the given bundle metric.
We say that a smooth submanifold $W\subset E$ \emph{has sub-Gaussian concentration} if: for every compact $D\subset M$, the volume of $\pi^{-1}(D)\cap W\cap B_R$ (in the Riemannian manifold $W$) grows less than any Gaussian density, that is: $\forall \e>0$ $\exists C(\e)\ge 0$ such that $\forall R>0$,
\be 
\text{Vol}_W\left(
\pi^{-1}(D)\cap W\cap B_R
\right)
=\int_{\pi^{-1}(D)\cap W\cap B_R}dW\le C(\e)e^{\e R^2}.
\ee
If $W$ is a Whitney stratified submanifold, we say that it has sub-Gaussian concentration if its smooth locus has sub-Gaussian concentration.

\end{defi}
It turns out that the property of having sub-Gaussian concentration is local and it is independent from the choice of a metric. In fact, this condition can be checked by proving that $W$ has sub-Gaussian concentration in each chart $E|_U\cong\D^m\times \R^s$ of a trivialization atlas for the bundle $E\to M$, and with respect to the standard metric. This is proved in Lemma \ref{lem:subGwell}. 
For this reason, in the following results we won't need to mention the Riemannian structure at all.

\begin{thm}\label{thm:mainEgau}
Let $\pi\colon E\to M$.
Let $W=\sqcup_{i\in I}W_i\subset E$ be a smooth Whitney stratified subset of codimension $m$ such that $W^{(s)}\transv E_p$ for every $p\in M$, where $W^{(s)}$ is the union of the higher dimensional strata.
Assume that $W$ has sub-Gaussian concentration. 
\begin{enumerate}[]
\item  Let $X\in\goo \infty ME$ be a non-degenerate $\mC^\infty$ Gaussian random section. Then $\delta_{X\in W}$ is locally integrable, hence $\E\#_{X\in W}$ is an absolutely continous Radon measure on $M$. 
\item  Let $X_d,X_\infty\in\goo \infty ME$ be a sequence of non-degenerate $\mC^\infty$ Gaussian random sections such that $K_{X_d}\to K_{X_\infty}$ in the $\mC^2$ topology (weak Whitney), as $d\to +\infty$. Assume that the limit $X_\infty$ is also non-degenerate. Then
\be 
\lim_{d\to +\infty}\E\#_{X_d\in W}(A)=\E\#_{X_\infty \in W}(A)
\ee
for every relatively compact Borel subset $A\subset M$.
\end{enumerate}
\end{thm}
\subsection{Semialgebraic singularities}
Clearly, if $W$ is compact, or a linear subbundle, or a cylinder over a compact, then it has sub-Gaussian concentration. The example that we are most interested in, though, is the case in which $W\subset E$ is locally semialgebraic. By this we mean that every $p\in M$ has a neighborhood $U\subset M$ such that there is a trivialization of the bundle $E|_{U}\cong \R^s\times \R^m$ such that $W\cap E|_{U}$ is a semialgebraic subset of $\R^s\times \R^m$. In this case the volume of $W\cap E|_{U} \cap B_R$ evidently grows in a polynomial way and thus...
\begin{remark}
...if $W\subset E$ is locally semialgebraic, it has sub-Gaussian concentration.
\end{remark}
The reason why we put the accent on the semialgebraic case is that Theorems \ref{thm:maingau} and \ref{thm:mainEgau} can be used to study the expected number of singular points of a GRS. The meaning of ``singular point'' depends on the situation, but in general it is a point $p\in M$ where the section satisfies some condition involving its derivatives. A general model for that  (the same proposed in \cite{mttrp} and \cite{HaniehPaulAnto}) is to consider a subset $W\subset J^rE$ of the bundle of $r$ jets (if the derivatives involved are of order less than $r$) of sections of $E$ (see \cite{Hirsch} for a definition of the space of jets) and call \emph{singular points of class $W$} those points $p\in M$ such that the $r^{th}$ jet of $X$ at $p$ belongs to $W$. 
Examples are:
\begin{enumerate}[$\bullet$]
\item ``Zeroes'', when $W=M\subset E=J^0(E)$.
\item ``Critical points'', when $W\subset J^1(M,\R^k)$, is  such that $j^1_pX\in W$ if and only if $d_pX$ is not surjective.
\item 
Combining the two previous examples, one can consider
$W\subset J^1(E\times \R)$, such that given a function $g\colon M\to \R$ and a section $s\colon M\to E$, then
$(j_p^1(s,g))^{-1}(W)=\text{Crit}(g|_{s^{-1}(0)})$. This is useful in that it provides an upper bound for the total Betti number of the set of zeroes of $s$. Indeed, generically, by Morse theory the latter must be smaller than the \emph{number} of  singular points of class $W$.
\item The Boardman singularity classes: $W=\Sigma^{(i_1,\dots,i_r)}\subset J^rE$, see \cite{arnold2012singularities}.
\end{enumerate}
In all of the above examples, and in most natural situations, the singularity class is given by a locally semialgebraic subset $W\subset J^rE$.

Considered this, we rewrite the statements of Theorems \ref{thm:maingau} and \ref{thm:mainEgau} in the case when the vector bundle is $J^rE\to M$ and the Gaussian random section is holonomic, namely it is of the form $j^rX$.

\begin{cor}\label{cor:mainjgau}
Let  $\pi\colon E\to M$ a smooth vector bundle. Let $W\subset J^rE$, with $r\in \N$, be a smooth submanifold of codimension $m$ such that $W\transv J^r_pE$ for every $p\in M$. 
Let $A\subset M$ be any relatively compact Borel subset.
\begin{enumerate}[$1$]
\item Let $X\in\goo \infty ME$ be a $\mC^{\infty}$ Gaussian random section with non-degenerate $r^{th}$ jet. Then there exists a smooth density $\delta_{\Gamma(j^rX,W)}\in \mathscr{D}^\infty(W)$ such that, 
\be\label{eq:formjgaus}
\E\#\{p\in A\colon j^r_pX_d\in W\}= \int_{W\cap \pi^{-1}(A)}\delta_{\Gamma(j^rX,W)}.
\ee
Moreover, if $W$ has sub-Gaussian concentration, then the above quantity is finite.
\item Assume that $W$ has sub-Gaussian concentration. Let $X_d, X_\infty\in\goo {\infty}ME$ be a sequence of  $\mC^{\infty}$ Gaussian random sections with non-degenerate $r^{th}$ jet and assume that 
$ 
K_{X_d}\to K_{X_\infty}
$ in the $\mC^{2r+2}$ topology (weak Whitney), as $d\to+\infty$. Assume that the limit $X_\infty$ also has non-degenerate $r^{th}$ jet. Then
\be 
\lim_{d\to +\infty}\E\#\{p\in A\colon j^r_pX_d\in W\}=\E\#\{p\in A\colon j^r_pX_{\infty}\in W\}.
\ee
\end{enumerate}
\end{cor} 
 
\begin{remark}
Here the formula for $\delta_{\Gamma(j^rX,W)}$ is obtained from formula \ref{eq:gaudelgamma} by replacing $E$ with $J^rE$, $X$ with $j^rX$ and $K_X$ with the covariance tensor of $j^rX$. Notice that the latter can be derived from the jet of order $2r$ of $K_X$.
\end{remark}

%
%
%

\section{Expectation of other counting measures}\label{sec:weights}
Let $X\colon M\to N \supset W$ be a random $\mC^1$ map. In this manuscript we chose to focus on the expectation of the actual number of points of intersection of $X$ and $W$. However, all the discussion can be generalized with minimal effort to the case in which a different weight is assigned to each point, in the following way.

For any $\a\colon \mC^1(M,N)\times M\to \R^k$ Borel measurable and $A\subset M$, define
\be\label{eq:Na}
\#^\a_{X\in W}(A)=\sum_{p\in A\cap X^{-1}(W)}\a(X,p) \in \R^k,
\ee 
and a density $\delta^\a_{X\in W}\colon M \to \Delta M$, such that
\be\label{eq:deltaa}
\delta_{X\in W}^\a(p)=\int_{ S_p\cap W}\E\left\{\a(X,p)\delta_pX\frac{\sigma_q(X,W)}{\sigma_q(S_p,W)}\bigg|X(p)=q\right\}\rho_{X(p)}(q)d(S_p\cap W)(q).
\ee
The following result extends Theorem \ref{thm:maindens} (Compare with \cite[Theorem 12.4.4]{AdlerTaylor} and \cite[Proposition 6.5]{Wschebor}, in the standard case.).
\begin{thm}\label{thm:mainalph}
Let $(X,W)$ be a KROK couple. Then Theorem \ref{thm:maindens} holds for $\#^\a_{X\in W}$ and $\delta^\a_{X\in W}$: for any Borel subset $A\subset M$ we have
\be\label{eq:mainalph}
\E\#^\a_{X\in W}(A)=\int_A\delta^\a_{X\in W} 
\ee
When not finite, both sides take the same infinite value among $+\infty,-\infty$ or $\infty-\infty$.
\end{thm}
\subsection{The intersection degree}
\newcommand{\dg}{\text{deg}}
\newcommand{\lk}{\text{link}}
Let $M$ be oriented and let $W\subset N$ be a closed cooriented submanifold. Then, given $q\in W$ and a linear bijection $L\colon T_pM\to T_qN/T_qW$, the sign of $\det(L)$ is well defined, since both vector spaces are oriented. If $f\colon M\to N$ is a $\mC^1$ map such that $f\transv W$, then the \emph{intersection degree} of $f$ and $W$ is defined as $\text{deg}(f,W)=\#^\a_{f\in W}(M)\in \Z$, where
\be 
\a(f,p)=\text{sgn}\left(\frac{d_pf}{T_{f(p)}W}\right).
\ee
In this situation, we can incorporate the sign in the definition of the angle, by defining $\text{sgn}\sigma_q(f,W)=\a(f,p)\sigma_q(f,W)$, so that the formula \eqref{eq:mainalph} for the expected intersection degree of a KROK couple $(X,W)$ and an open subset $U\subset M$ becomes
\be\label{eq:Edeg}
\E\{\dg(f|_U,W)\}=\int_U \int_{ S_p\cap W}\E\left\{\delta_pX\frac{\text{sgn}\sigma_q(X,W)}{\sigma_q(S_p,W)}\bigg|X(p)=q\right\}\rho_{X(p)}(q)d(S_p\cap W)(q).
\ee
\begin{remark}\label{rem:anto}
This confirms a general idea suggested to the author by Antonio Lerario, according to which the general philosophy to deal with these kind of formulas should be:
\emph{To get the formula for the signed count, add the sign to both members of the formula for the normal count.} Theorem \ref{thm:mainalph} can be thought as an extension of this philosophy from the \emph{sign} to any choice of \emph{weight} $\a$.
\end{remark}
\begin{remark} If $U\subset M$ be an open set whose closure is a compact topological submanifold with boundary, such that $f(\de U)\subset N\- W$. It can be seen that $\dg (f|_U,W)$ actually depends only on homological data and thus it is defined and locally constant on the space of continuous functions $f\colon \overline{U}\to N$ such that $f(\de U)\subset N\- W$. Moreover, if the Poincaré dual of $W$ in $N$ vanishes, then one can define the \emph{linking number}\footnote{
Since $W\subset N$ is cooriented, there exists a Thom class $\tau\in H^m(N,N\smallsetminus W)$. By definition, $\dg(f|_U,W)=\int_Uf^*\tau$ in the case $f\transv W$, but now this identity remains true for any continuous $f$ such that $f(\de U)\subset N\smallsetminus W$.
Looking at the long exact sequence for the pair $(N,N\smallsetminus W)$, we see that $\tau\mapsto e$ maps to the Poincaré dual of $W$, so that if $e=0$ then there exists an element $\a\in H^{m-1}(N\smallsetminus W)$ that maps to $\tau$ (i.e. $d^*\a=\tau$), therefore $\int_Uf^*\tau=\int_{\de U} f^*\a$ by naturality (or Stokes theorem from De Rham's point of view). In such case, $\lk(f|_U,W)=\int_{\de U}f^*\a$. If $N$ is a tubular neighborhood of $W$, then $e,\tau,\a$ are  respectively the Euler class, the Thom class and the class of a closed global angular form (Compare with \cite{botttu}).
} 
$\lk(g,W)\in \Z$  for any $g\colon B\to N\-W$, where $B$ is a closed manifold of dimension $m-1$. In this case $\dg(f|_U,W)=\lk(f|_{\de U},W)$ depends only on the restriction to the boundary. For this reason, $\dg(X|_U,W)$ can be thought to be \emph{less random} than $\#_{X\in W}(U)$; in fact, often it ends up being deterministic.
\end{remark}
\subsection{Multiplicativity and currents}\label{sub:multcurr}
The formula \eqref{eq:Edeg} can be written as the integral of a differential form $\E\{\dg(f|_U,W)\}=\int_{M_+} \Omega_{X\in W}$ over the oriented manifold $M_+$, defined as
\be\label{eq:Omega}
\Omega_{X\in W}(p)=\int_{ S_p\cap W}\E\left\{{(X^*\nu)_p}\big|X(p)=q\right\}\rho_{X(p)}(q)\frac{d(S_p\cap W)(q)}{\sigma_q(S_p,W)} \in \wedge^m T^*_pM,
\ee
where $\nu\in \wedge^m TW^\perp$ is the volume form of the oriented metric bundle $TW^\perp$. This follows simply from remark \ref{rem:altforms} and the fact that if $M$ and $N$ are Riemannian, then
\be 
\text{sgn}\left(\frac{d_pf}{T_{f(p)}W}\right)=\text{sgn}\det\left(\Pi_{ T_qW^\perp}\circ d_pX\right).
\ee
What is interesting about this is that the form $\Omega_{X\in W}\in \Omega^r(M)$ is still defined if the codimension of $W$ is $r< m$. 

If the couple $(X,W)$ satisfies the KROK hypotheses except for the requirement on the codimension of $W$ in KROK.\ref{itm:krok:4}, then let us say that $(X,W)$ is semi-KROK.
Let us consider semi-KROK couples $(X_i,W_i)$ on the manifolds $M, N_i$ for $i=1,\dots,k$ where the codimensions $r_i$ of $W_i$ are such that $r_1+\dots+r_k=m$ and all the $W_i$s are cooriented. Then it is easy to see that the product map $X:=X_1\times \dots \times X_k\colon M\to N:=N_1\times \dots \times N_k$ and the submanifold $W=W_1\times \dots \times W_k$ form a KROK couple $(X,W)$, and formula \eqref{eq:Omega} gives
\be 
\Omega_{X\in W}= \Omega_{X_1\in W_1}\wedge \dots \wedge \Omega_{X_k\in W_k}.
\ee
This can be interpreted in the language of currents, in the same spirit of \cite{Nicolaescu2016}.

\begin{claim}\label{claimcurr}
The $r$-form $\Omega_{X_i\in W_i}\in \Omega^r(M)$ is equal, as a current, to the expectation of the random current $[[X_i^{-1}(W)]]$ defined by the integration over the oriented submanifold $X_i^{-1}(W)$. 
\end{claim}

This follows from to the fact that the intersection degree $\dg(X|_U,W)$ may be viewed as the evaluation of the $0$-dimensional current $[[X^{-1}(W)]]=[[X_1^{-1}(W_1)\cap\dots\cap X^{-1}_k(W_k)]]$ on the function $1_U$. And since the intersection of currents is linear, we have 
\bega \label{eq:doubt}
\Omega_{X_1\in W_1}\wedge &\dots \wedge \Omega_{X_k\in W_k}
\\
=\Omega_{X\in W}&
=\E[[X^{-1}(W)]]
\\
=\E[[X_1^{-1}(W_1)]]\cap &\dots \cap \E[[X_k^{-1}(W_k)]].
\eega
\begin{remark}
Claim \ref{claimcurr} is not a theorem yet, for it has to be proved that the validity of the equation \eqref{eq:doubt} for every family of semi-KROK couples, implies that $\Omega_{X_i\in W_i}=\E[[X_i^{-1}(W_i)]]$. We plan to elaborate more on this subject in a subsequent paper.
\end{remark}

\subsection{Euler characteristic}
A special case of intersection degree is when $N=E\to M$ is the total space of an oriented vector bundle over an oriented manifold $M$ and $W=0_E\subset E$ is the zero section. Then $\dg(f,0_E)=\chi(E)$ is the Euler characteristic. In this case we can present the formula \eqref{eq:mainalph} in the form of Remark \ref{rem:connectionForm}.
\bega
\E\#_{X\in W}(U) &=\int_U \E\left\{|\det(\nabla X)_p|\bigg|X(p)=0\right\}\rho_{X(p)}(0)dM(p),
\\
\E\dg(X|_U,0_E) &=\int_U \E\left\{\ \det(\nabla X)_p\ \bigg|X(p)=0\right\}\rho_{X(p)}(0)dM(p)
\\
(&= \chi(E) ,\quad \text{If $U=M$ and is compact}),
\eega
where $E$ is endowed with some connected Riemannian metric\footnote{
Actually here the connection is not needed, since if $X(p)=0$ then $(\nabla X)_p$ is independent from $\nabla$.
 }.
  Notice that here, the density $\det(\nabla X)\rho_{X(\cdot)}(0)dM\colon M\to \Delta M$ is actually an intrinsic object, independent from the chosen Riemannian structures. 
  
  Following the discussion in the previous subsection \ref{sub:multcurr}, we now view the intersection degree $\dg(X|_{(\cdot)},0_E)$ as a random current $[[X^{-1}(0_E)]]\in\Omega_c^{0}(M)^*$. Its expectation is thus given by $\E[[X=0]] = \Omega_{X\in 0_E}$, defined in \eqref{eq:Omega}.

  
  Suppose now that $X\in\goo \infty M E$ is a nondegenerate smooth Gaussian random section and $s\colon M\to E$ is any smooth section. Let $E$ be endowed with the bundle metric defined by $X$ and let $\nabla=\nabla^X$ be the metric connection naturally associated with $X$ (see \cite{Nicolaescu2016}). Let $\nabla^2\in \bigwedge^2T^*M\otimes \bigwedge^2 E$ be its curvature. 
  Assume that $m$ is even, then a formula for $\Omega_{X\in 0_E}$ was computed in \cite{Nicolaescu2016}.
 \bega
 \Omega_{X\in 0_E}=
\E\left\{ \det(\nabla X)\right\}\frac{\w_{M}}{(2\pi)^{\frac m2}}=\frac{\text{Pf}(\nabla^2)}{(2\pi)^{-\frac{m}{2}}
}  \eega
 where $\w_{M}\in\Omega^m(M)$ is the volume form of the oriented Riemannian manifold $M$.  
\begin{remark}  
  The result of Nicolaescu \cite{Nicolaescu2016}, extends this to vector bundles with arbitrary even rank $r$. He proves that the expectation of the random current $[[X=0]]\in \Omega_c^{m-r}(M)^*$ is the current defined by the $r$-form $e(E,\nabla)=(2\pi)^{-\frac{r}{2}}Pf(\nabla^2)$. (Our sign convention in the definition of the Pfaffian $\text{Pf}$ is the same as in \cite{milnor-stasheff}, which differs to that of \cite{Nicolaescu2016} by a factor $(-1)^\frac {m}{2}$.)
 \end{remark}
 \subsection{Absolutely continuous measures}
 \HP\  on $M$ and $N$. Let $\mu$ be a finite Borel signed measure on $\mC^1(M,N)$ that is absolutely continuous with respect to $X$. By the Radon-Nikodym theorem, this is equivalent to the existence of an integrable Borel function $\a\colon \mC^1(M,N)\to \R$, with $\E\{|\a(X)|\}<\infty$, such that 
 \be 
 \int_{\mC^1(M,N)} \mathcal{F}(f)d\mu(f)=\E\{\mathcal{F}(X)\a(X)\}.
 \ee
Considering the case in which the function $\a$ in Theorem \ref{thm:mainalph} does not depends on the point, we deduce that the generalized Kac-Rice formula holds for every such measure $\mu$.
\be 
\int_{\mC^1(M,N)}\#(f^{-1}(W)\cap A)d\mu(f)=\E\#^\a_{X\in W}(A)=\int_{A}\delta^\a_{X\in W}.
\ee

\section{Examples}\label{sec:ex}

\subsection{Poincaré Formula for Homogenous Spaces}
In this section we will use Theorem \ref{thm:main} to give a new proof of the following Theorem. It is a special case of the Poincaré Formula for homogenous spaces, see \cite[Th. 3.8]{Howard}. 
\begin{thm}\label{thm:howard}
Let $G$ be a Lie group and let $K\subset G$ be a compact subgroup. Assume that $G$ is endowed with a left-invariant Riemannian metric that is also right-invariant for elements of $K$ and define the metric on $G/K$ to be the one that makes the projection map $\pi\colon G\to G/K$ a Riemannian submersion. Let $M,W\subset G/K$ be two smooth submanifolds (possibly with boundary) of complementary dimensions. Then
\be\label{eq:howardpoincare}
\int_G\#(gM\cap W)dG(g)=\int_M\int_W \bar{\sigma}_K\left(T_xM^\perp,T_yW^\perp\right)\Delta_K(x)dM(x)dW(y).
\ee
\end{thm}
Here, $\Delta_K(y)$ is defined as follows. Let $y=\nu K=\pi(\nu)$, then $\Delta_K(y)$ is defined to be the jacobian of the right multiplication by $\nu^{-1}$ in $G$,
\be 
\Delta_K(y):=\Delta(\nu)=J(R_{\nu^{-1}}).
\ee
To see the this definition is well posed, observe that since $\Delta\colon G\to \R$ is a group homomorphism and $K$ is compact, the set $\Delta(K)$ must be a compact subgroup of $\R$, thus $\Delta(K)=\{1\}$. It follows that $\Delta$ factorizes to a well defined function $\Delta_K\colon G/K\to \R$.

Notice that the angle $\bar{\sigma}_K\left(T_xM^\perp,T_yW^\perp\right)$ takes two subspaces that do not belong to the same tangent space $T_x(G/K)\neq T_y(G/K)$, in general. In fact $\bar{\sigma}_K$ is not the usual angle of Definition \ref{defi:angle}, but it is defined as follows. 
If $x=\mu K$ and $y=\nu K$ and $U,V$ are vector subspaces of, respectively, $T_x(G/K)\subset T_\mu G$ and $T_y(G/K)\subset T_\nu G$, then $\mu^{-1}_*U$ and $\mu^{-1}_*V$ are both subspaces of $T_1G$ and we can compute the angle $\sigma(\mu^{-1}_*U,\nu^{-1}_*V)$ using Definition \ref{defi:angle}. Then $\bar{\sigma}_K\left(U,V\right)$ is obtained by taking the average among all choices of $\mu$ and $\nu$:
\be 
\bar{\sigma}_K(U,V)=\int_K \sigma_{T_1 G}(\mu^{-1}_*U,k^{-1}_*\nu^{-1}_*V)dK(k).
\ee
Observe that, by Proposition \ref{prop:angleperp}, we have that $\bar{\sigma}_K(U,V)=\bar{\sigma}_K(U^\perp,V^\perp)$.
\begin{proof}[Proof of Theorem \ref{thm:howard}]
Let $\Omega\subset G$ be an open subset with compact closure and with $\vol_G(\de \Omega)=0$. Define the smooth random map $X\colon M\to G/K$ such that
\be
\P\{X=(L_g)|_{M}, \text{ for some $g\in A$}\}=\frac{\vol_G(A\cap \Omega)}{\vol_G(\Omega)}.
\ee
In other words, $X$ is uniformly distributed on the set of left multiplications by elements of $\Omega\subset G$, that means that $X\colon M\to G/K$ is the map $X(p)=g\cdot p$, where $g$ is a uniform random element of $\Omega$. We want to apply Theorem \ref{thm:main} to the random map $X$, where $W\subset N=G/K$. To this end, let us show that the couple $(X,W)$ is KROK (see Definition \ref{def:krok}).

\begin{enumerate}[(i)]
\item Ok.
\item \label{itm:how3} Let us fix $p=\mu K\in M$. The support of $[X(p)]$ is the open set $\Omega \cdot p\subset G/K$. Let $A\subset G/K$
\bega
\P\{X(p)\in A\}&=\frac{1}{\vol_G(\Omega)}\vol_G\left(\{g\in \Omega\colon g\mu K\in A\}\right)\\
&=\frac{1}{\vol_G(\Omega)}\vol_G\left(\Omega\cap\left(\pi^{-1}(A)\cdot\mu^{-1}\right) \right)\\
&=\frac{\Delta(\mu)}{\vol_G(\Omega)}\vol_G\left( \Omega\cdot\mu\cap\pi^{-1}(A)\right)\\
&=\frac{\Delta_K(p)}{\vol_G(\Omega)}\int_A\left(\int_{\Omega\cdot\mu\cap \pi^{-1}(q)}d\left(\pi^{-1}(q)\right)\right)d\left(G/K\right)(q) 
\\
&= \frac{\Delta_K(p)}{\vol_G(\Omega)}\int_A\vol_K(K\cap \nu^{-1}\Omega \mu)\ d\left(G/K\right)(\nu K),
\eega
where in the fourth step, we used the Coarea Formula (Theorem \ref{thm:coarea}). 

At this point, we can define a continuos function $f_\Omega\colon G/K\times G/K\to \R$ such that 
\be 
f_\Omega(p,q)=\vol_K(K\cap q^{-1}\Omega p):=\vol_K(K\cap \nu^{-1}\Omega\mu)\ee
 for any $\nu\in\pi^{-1}(q)$ and $\mu\in\pi^{-1}(p)$. This definition is well posed since the metric, and hence the volume form, on $G$ is invariant with respect to elements of $K$, both on the left and on the right. To see that $f_\Omega$ is continuous it is enough to show that if $\nu_n\to \nu$ and $\mu_n\to \mu$ in $G$, then $\vol_K(K\cap \nu_n^{-1}\Omega\mu_n)\to \vol_K(\nu^{-1}\Omega\mu)$.
 We prove this via the dominated convergence theorem (here we use the fact that $\vol_G(\de\Omega)=0$), since 
 \be 
 \vol_K(K\cap \nu_n^{-1}\Omega\mu_n)=\int_K 1_{\Omega}(\nu_n a \mu_n^{-1})dK(a),
 \ee
where $1_{\Omega}(\nu_n a \mu_n^{-1})\to 1_{\Omega}(\nu a \mu^{-1})$ for every $a\notin  \nu^{-1}\de\Omega \mu$ and, eventually,
\be 
1_{\Omega}(\nu_n a \mu_n^{-1})\le 1_{B_1(\nu^{-1})\cdot\Omega\cdot B_1(\mu)}(a)\in L^1(G).
\ee
It follows that $[X(p)]$ is absolutely continuous on $S_p:=G/K$, with a continuous density 
\be \label{eq:howarddensity}
\delta_{X(p)}(q)=\rho_{X(p)}(q)\cdot d(G/K)(q)= \frac{\Delta_K(p)}{\vol_G(\Omega)}f_\Omega(p,q)\cdot d(G/K)(q).
\ee
\begin{remark}
Notice that $\rho_{X(p)}(q)\neq 0$ if and only if $q\in \Omega\cdot p$. However, it would be a bad idea to define $S_p=\Omega\cdot p$. Indeed, with that choice the set $\mathcal{S}$ would not be closed and thus KROK.\ref{itm:krok:7} would not hold.
\end{remark}
\item Let us consider the smooth map $F\colon \Omega\times M\to G/K$, given by $F(g,p)=g\cdot p$. Clearly $F(\cdot,p)$ is a submersion, for every $p\in M$, therefore $F\transv W$. This, by a standard argument (Parametric Transversality, see \cite[Theorem 2.7]{Hirsch}) implies that $(L_g)|_M=F(g,\cdot)\transv W$ for almost every $g\in \Omega$. We conclude that
\be 
\P\{X\transv W\}=1.
\ee
\item Since $S_p=G/K$, the transversality assumption $S_p \transv W$ is certainly satisfied for every $p\in M$.
\item Ok.
\item In this case we have that $\mathcal{S}=M\times (G/K)$ is, without a doubt, a closed submanifold of $M\times (G/K)$. Moreover,
$\MW =M\times W$.
\item By equation \eqref{eq:howarddensity}, we have that $\rho_{X(p)}(q)$ is continuous with respect to all $(p,q)\in M\times (G/K)$.
\item This is the most complicated condition to check. Let $\a\colon \mC^1(M,G/K)\to \R$ be a continuous function (the case in which $\a$ depends on $p$ follows automatically, because of Proposition \ref{prop:CazzoDuro}). We have to show that the function $e\colon M\times G/K\to \R$, defined as $e(p,q)=\E\left\{\a(X)\big|X(p)=q\right\}$, is continuous at all points of $M\times W$. Let $p,q\in G/K$ and let $\nu,\mu\in G$  be such that $p=\mu K$ and $q=\nu K$. Notice that if $X=(L_g)|_{M}$, then $X(p)=q$ if and only $g\in \nu K\mu^{-1}$. Therefore if $\w$ denotes a uniformly distributed random element of $\Omega$ and $\xi$ denotes a uniformly distributed random element of $K$, then
\bega\label{eq:howcondi}
\E\left\{\a(X)\big|X(p)=q\right\}
&= \E\left\{\a\left((L_\w)|_M\right)\big|\w\in \Omega \cap \nu K \mu^{-1} \right\}
\\
&= \E\left\{\a\left((L_{\nu \xi\mu^{-1}})|_M\right)\big|k\in K\cap \nu^{-1} \Omega \mu \right\}
\\
&= \int_{K\cap \nu^{-1} \Omega \mu}\a\left((L_{\nu k\mu^{-1}})|_M\right) \frac{dK(k)}{\vol_K\left(K\cap \nu^{-1} \Omega \mu\right)}
\\
&=.\frac{1}{f_\Omega(p,q)}\int_{K\cap \nu^{-1} \Omega \mu}f_\a(\nu k\mu^{-1}) dK(k),
\eega
where $f_\a\colon G\to \R$ is a continuous function, defined as $f_\a(g)=\a\left(L_g|_M\right)$. Notice, that the last integral depends only on $p,q$, for if $\mu'=\mu a$ and $\nu'=\nu b$ for some $a,b\in K$, then the change in the integral corresponds to the change of coordinates $k'= b^{-1}ka$. 
To see that $e$ is continuous it is enough to check the continuity of the composed function $e(\pi(\cdot),\pi(\cdot))$. Let $\mu_n\to \mu$ and $\nu_n\to \nu$ be converging sequences in $G$ and define, for any $n_0\in\N$, the number
\bega
s_{n_0}=\sup_{n\ge n_0}\sup_{k\in K}|f_\a(\nu_nk\mu_n^{-1})-f_\a(\nu k\mu^{-1})|.
\eega
Since $f_\a$ is continuous and $K$ is compact, we have that $s_n\to 0$. As a consequence we get that
\bega
\lefteqn{\limsup_{n\to \infty}|e(\mu_nK,\nu_n K)-e(\mu K, \nu K)|\le}
\\
& &\le\lim_{n_0\to \infty}\sup_{n\ge n_0}\ \int_{K\cap \nu^{-1} \Omega \mu}\frac{\left|f_\a(\nu_n k\mu_n^{-1})-f_\a(\nu k\mu^{-1})\right|}{\vol_K\left(K\cap \nu^{-1} \Omega \mu\right)}{dK(k)}
\\
& &\le \lim_{n_0\to \infty}s_{n_0}=0.
\eega
This proves that $e(\pi,\pi)\colon G\times G\to\R$ is continuous.
\end{enumerate}
At this point, we know that the couple $(X,W)$ is KROK, therefore the expected number of intersections of the submanifolds $X(M)$ and $W$ is given by the generalized Kac-Rice formula of Theorem \ref{thm:main}, where $S_p=G/K$.
\begin{multline} 
\E\{\#_{X\in W}(M)\}=\int_{M}\rho_{X\in W}dM(p)
\\
=\int_M\int_W\E\left\{J_pX\sigma_q\left(d_pX(T_pM),T_qW\right)\big|X(p)=q\right\}\rho_{X(p)}(q)dW(q)dM(p).
\end{multline}
Recall that $X$ ranges among left translations, which are isometries, thus $JX=1$ with probability one. Moreover, we already computed $\rho_{X(p)}(q)$ (see \eqref{eq:howarddensity}) and we already understood how to compute the conditional expectation $\E\{\cdot|X(p)=q\}$ (see \eqref{eq:howcondi}). 
\begin{multline}
\E\{\#_{X\in W}(M)\}= \\
\int_M\int_W\frac{\Delta_K(\mu K)}{\vol_G(\Omega)}\left(\int_{K\cap \nu^{-1} \Omega \mu}\sigma_{\nu K}\left(\nu k \mu^{-1}T_{\mu K}M,T_{\mu K}W\right) dK(k)\right)dW(\nu K)dM(\mu K)=
\\
\int_M\int_W\frac{\Delta_K(\mu K)}{\vol_G(\Omega)}\left(\int_{K\cap \nu^{-1} \Omega \mu}\sigma_{T_1G}\left(\mu^{-1}T_{\mu K}M,k^{-1}\nu^{-1}T_{\nu K}W\right) dK(k)\right)dW(\nu K)dM(\mu K).
\end{multline}
Let $\Omega_n\subset \Omega_{n+1}\subset G$ be a sequence of relatively compact subsets such that $\cup_{n}\Omega_n =G$ and with $\vol_G(\de \Omega_n)=0$, and let $X_n\colon M\to G/K$ be the random map defined as above, with $\Omega=\Omega_n$. We obtain the thesis \eqref{eq:howardpoincare} by monotone convergence:
\bega
\int_{G}\#(gM\cap W)dG(g) = 
\sup_{n\in \N}\int_{\Omega_n}\#(gM\cap W)dG(g)
= \sup_{n\in\N}\E\{\#_{X_n\in W}(M)\}\vol_G(\Omega_n)
\\
=\int_M\int_W\Delta_K(\mu K)\left(\int_{K}\sigma_{T_1G}\left(\mu^{-1}T_{\mu K}M,k^{-1}\nu^{-1}T_{\nu K}W\right) dK(k)\right)dW(\nu K)dM(\mu K)
\\
=\int_M\int_W\Delta_K(p)\bar{\sigma}_{K}\left(T_{p}M,T_{q}W\right)dW(q)dM(p).
\eega
\end{proof}
\subsection{Isotropic Gaussian fields on the Sphere}
Let $X\colon S^m\to \R^k$ be a $\mathcal{C}^1$ Gaussian random field on the sphere. Using the notation of section \ref{sec:gausscaseintro} (consistently with \cite{dtgrf}), we say that $X\in \g 1{S^m}k=\goo 1{S^m}E$, meaning that $X$ is a Gaussian random section of the trivial bundle $E=S^m\times \R^k$. We say that $X$ is \emph{isotropic} if $[X\circ R]=[X]$ for any rotation $R\in O(m+1)$. In particular, the covariance of $X(x)$ and $X(y)$ depends only on the angular distance $\a(x,y)=\arccos(\langle x,y\rangle)$ (because $S^m\times S^m/O(m+1)\cong [0,\pi]$):
\be 
K_X(x,y)=\E\left\{X(x)\cdot X(y)^T\right\}=K(\langle x,y\rangle)=F(\a(x,y)),
\ee
where $K\colon [-1,1]\to \R^{k\times k}$ and $F\colon \R\to \R^{k\times k}$ are $\mathcal{C}^2$ functions, such that $F$ is even and $2\pi-$periodic and
\be \label{eq:stropFK}
F(\a)=K(\cos\a).
\ee
As a consequence, the covariance structure of the first jet of $X$ at a given point $p\in S^m$, namely the couple $j^1_pX=(X(p),d_pX)$, is understood as follows. Define \be
\Sigma_0:=K(1)=F(0) \quad \text{and} \quad \Sigma_1=K'(1)=-F''(0).
\ee
Then given any orthonormal basis $\de_1,\dots, \de_m$ of $T_p S^m$, we have the following identities
\be\label{eq:istropId}
\E\left\{X(p)X(p)^T\right\}=\Sigma_0 ,
\quad 
\E\left\{\de_i X\cdot X(p)^T\right\}=0, \quad
\E\left\{\de_iX \cdot \de_jX^T\right\}=\Sigma_1\delta_{i,j}.
\ee
\begin{example}[Kostlan Polynomials]
Let $\psi_d\in\g 1{S^m}{}$, be defined as the restriction to $S^m$ of the random homogeneous Kostlan polynomial of degree $d$:
\be 
\psi_d(x)=\sum_{|\a|=d}{d\choose\a}^\frac12 \gamma_\a x^\a,
\ee
where $\gamma_\a\sim N(0,1)$ are independent normal Gaussian. Then $\psi_d$ is a smooth isotropic Gaussian field $X\in\g 1{S^m}k$ with $K(t)=t^d$. In fact, any isotropic Gaussian field for which the function $K(t)$ has the form
\be
K(t)=K_0+K_1t+\dots +K_dt^d,
\ee 
for some positive definite simmetric matrices $K_\ell$, is a linear combination of Kostlan fields. To see this, let $A_\ell$ be a $k\times k$ matrix such that $A_\ell A_\ell^T=K_\ell$ and define 
\be \label{eq:stropKostlanmix}
\tilde{X}=A_0\begin{pmatrix}
\psi_0^1 \\ \vdots \\ \psi_0^k
\end{pmatrix}+A_1\begin{pmatrix}
\psi_1^1 \\ \vdots \\ \psi_1^k
\end{pmatrix}
+\dots +
A_d\begin{pmatrix}
\psi_d^1 \\ \vdots \\ \psi_d^k
\end{pmatrix},
\ee
where $\{\psi_\ell^i\}_{i,\ell}$ are independent copies of Kostlan polynomials of degree $\ell=1,\dots,d$. Then $\tilde{X}$ is equivalent to $X$, since they have the same covariance function. In the particular case in which $X=(\psi_{d_1},\dots,\psi_{d_k})^T$ where $\psi_{d_\ell}$ are independent Kostlan polynomials of degree $d_\ell$, then $\Sigma_0=\mathbb{1}_k$ and $\Sigma_1=\text{diag}(d_1,\dots,d_k)$.
\end{example}
We recall that the density function of a Gaussian random vector $\xi\sim N(0,\Sigma)$ in $\R^k$ with nondegenerate covariance matrix $\Sigma$ is given by
\be\label{eq:isodensga}
\rho_{\Sigma}(y)=\frac{e^{-\frac12y^T\Sigma^{-1}y}}{(2\pi)^\frac{k}2(\det \Sigma)^\frac12}.
\ee
\begin{lemma}\label{lem:MWRn}
Let $M$ be a Riemannian manifold. Let $X\in \g \infty Mk$ and assume that $X(p)$ has nondegenerate covariance matrix  $\Sigma_0(p)$. Let $W\subset \R^k$ be any submanifold (possibly stratified) of codimension $m$. Then 
\begin{multline}
\E\{\#X^{-1}(W)\}=\int_M\delta_{X\in W}(p)
\\
=\int_M\int_W\E\left\{\left|\det\left(\Pi_{T_yW^\perp}\circ d_p X\right)\right|\Bigg| X(p)=y\right\}\rho_{\Sigma_0}(y) dW(y)dM(p).
\end{multline}
\end{lemma}
\begin{proof}
Let $\Gamma_X$ be the graph of $X$ and $\MW :=M\times W$. Then $\Gamma_X$ is a non-degenerate Gaussian random section of the trivial bundle $E=M\times \R^k$ and $\MW $ is clearly transverse to all fibers of the bundle: $M\times W\transv \{p\}\times \R^k$. Therefore we can apply Theorem \ref{thm:maingau} to obtain a density $\delta_{X\in W}=\delta_{\Gamma(X)\in \MW }\in L^+(M)$. Moreover, the trivial connection on $E=M\times \R^k$ makes it into a linearly connected Riemannian bundle, for which $\MW $ is a parallel submanifold, so that we can present the density $\delta_{X\in W}$ with the formula of Remark \ref{rem:connectionForm}, knowing that $\nabla \Gamma_X=dX$ and $\MW \cap E_p=W$.
\be 
\delta_{X\in W}(p)=\int_W\E\left\{\left|\det\left(\Pi_{T_yW^\perp}\circ d_p X\right)\right|\Bigg| X(p)=y\right\}\rho_{\Sigma_0}(y) dW(y)dM(p).
\ee
\end{proof}
Given a smooth submanifold $W\subset \R^k$ (possibly stratified) of codimension $m$, we say that a measurable map $\nu\colon W\to \R^{m\times k}$ is a \emph{measurable normal framing} for $W$, if for almost every $y\in W$ the columns of the matrix $\nu(y)$ form an orthonormal basis of $T_yW^\perp$ (if $W$ is stratified, then this has to hold only for almost every $y$ in the top dimensional stratum of $W$).
\begin{thm}\label{thm:mainIsotrop}
Let $X\colon S^m\to \R^k$ be an isotropic $\mathcal{C}^1$ Gaussian random field. Let $\Sigma_0$ and $\Sigma_1$ be the $k\times k$ matrices defined as above and assume that $\Sigma_0$ is nondegenerate. Let $W\subset \R^k$ be any submanifold (possibly stratified) of codimension $m$ and let $\nu \colon W\to \R^{m\times k}$ be a measurable normal framing. Then
\be
\E\left\{\# X^{-1}(W)\right\}=2\cdot \int_W\sqrt{\det(\nu(y)^T\Sigma_1\nu(y))}\cdot\frac{e^{- \frac12y^T\Sigma_0^{-1}y}}{(2\pi)^\frac{k-m}2(\det \Sigma_0)^\frac12}dW(y)
.
\ee
\end{thm}
\begin{proof}
By Lemma \ref{lem:MWRn}, we have that 
\be 
\E\{\#X^{-1}(W)\}=\int_W\E\left\{\left|\det\left(\Pi_{T_yW^\perp}\circ d_p X\right)\right|\right\}\rho_{\Sigma_0}(y) dW(y)dS^2(p).
\ee
We can omit the conditioning $X(p)=y$, since in this case $X(p)$ and $d_pX$ are independent. The fact that the field is isotropic implies that the measure $\E\#_{X\in W}$ is an invariant measure on $S^m$, so that 
\be 
\E\{\#X^{-1}(W)\}=\vol(S^m)\int_W\E\left\{\left|\det\left(\Pi_{T_yW^\perp}\circ d_p X\right)\right|\right\}\rho_{\Sigma_0}(y) dW(y),
\ee
where $p\in S^2$ is any point. Let $\de_1,\dots,\de_m$ be an orthonormal basis of $T_pM$. It remains only to compute the expectation of the determinant of the random matrix $A(y)$ with columns $A_1(y),\dots, A_m(y)$ defined as
\be 
A_i(y):=\nu(y)^T\de_i X=\left(\Pi_{T_yW^\perp}\circ d_p X\right)(\de_i).
\ee
By the third set of the identities in  \eqref{eq:istropId}, we deduce that the columns of $A(y)$ are independent Gaussian random vectors, each of them having covariance matrix $K(y)=\nu(y)^T\Sigma_1\nu(y)$. Therefore
\bega
\E\left\{\left|\det\left(\Pi_{T_yW^\perp}\circ d_p X\right)\right|\right\}=\E\left\{\left|\det A(y)\right|\right\}=\sqrt{\det \nu(y)^T\Sigma_1\nu(y)} \frac{m!\vol(\B^m)}{(2\pi)^{\frac{m}{2}}}.
\eega

To conclude, let us observe that $\vol(S^m)\vol(\B^m) m!=2(2\pi)^m$, because of a special property of the Gamma function: $\Gamma(z+\frac12)\Gamma(z)=2^{1-2z}\sqrt{\pi}\Gamma(2z)$.
\bega 
\vol(S^m)\vol(\B^m) m!&= \frac{\pi^\frac{m+1}2(m+1)}{\Gamma\left(\frac{m+1}2 +1\right)} \frac{\pi^\frac{m}2}{\Gamma\left(\frac{m}2 +1\right)}m!
=\frac{
\pi^{m+\frac12}(m+1)!
}{
2^{- (m+1)}\sqrt{\pi}\Gamma(m+2)
}
=2^{m+1}\pi^m.
\eega
\end{proof}
\begin{remark} In the case $\Sigma_0=\sigma^2\mathbb{1}_k$, we obtain a particularly nice formula
\be 
\E\{\#X^{-1}(W)\}=\int_W \E\{\#X^{-1}(T_yW)\}\rho_{\sigma^2\mathbb{1}_{k-m}}(y)
dW(y).\ee
This allows to reduce to the case $k=m$ and to the standard version of Kac-Rice formula. Indeed $X\in (T_yW)$ if and only if $\nu(y)^TX=0$. For completeness, we report two results that can be proved by applying the standard Kac-Rice formula (Corollary \ref{thm:stropzero} and \ref{thm:stropshubsmale} are not new results).
\end{remark}
\begin{cor}(Gaussian Isotropic Kac-Rice formula)\label{thm:stropzero}
Let $X\colon S^m\to \R^m$ be an isotropic $\mathcal{C}^1$ Gaussian random field. 
Let $\Sigma_0$ and $\Sigma_1$ be the $m\times m$ matrices defined by the identities \eqref{eq:istropId} and assume that $\Sigma_0$ is nondegenerate. Then, for any $y\in\R^m$,
\be 
\E\left\{\# X^{-1}(y)\right\}=2\cdot \sqrt{\frac{\det(\Sigma_1)}{\det(\Sigma_0)}}e^{-\frac12 y^T\Sigma_0^{-1}y}.
\ee
\end{cor}
\begin{cor}(Shub-Smale Theorem \cite{shsm})\label{thm:stropshubsmale}
 Let $\psi_{1},\dots ,\psi_{m}$ be independent Kostlan homogenous polynomials of degrees $d_1,\dots, d_m$ and denote by $Z\subset\RP^m$ the random subset defined by the equations $\psi_i=0$. Then $\E\{\#Z\}=\sqrt{d_1 \cdot\cdot\cdot  d_m}$.
\end{cor}
Notice that Corollary \ref{thm:stropzero} covers also the case in which the equations in the Shub-Smale Theorem \ref{thm:stropshubsmale} are dependant, as long as they are jointly orthogonally invariant, in particular in the case of a mixed Kostlan polynomial defined as in equation \eqref{eq:stropKostlanmix}, we have 
\be
\E\{\#X^{-1}(0)\}=2\sqrt{\frac{\det\left(A_0A_0^T+A_1A_1^T+\dots +A_dA_d^T\right)}{\det\left(A_1A_1^T+2A_2A_2^T+\dots +dA_dA_d^T\right)}}.
\ee
%


\section{Proof of Lemma \ref{lemma:meas}}\label{sec:proflemmameas}

\begin{proof}[Proof of Lemma \ref{lemma:meas}]
We can assume that $W$ is closed, by replacing $X\colon M\to N$, with the random map $X_W\colon M\times W\to N\times N$, such that $(p,w)\mapsto (X(p),w)$ and $W\subset N$ with the diagonal $\Delta\subset N\times N$, which is surely a closed submanifold. It is straightforward to see that $X\transv W$ if and only if $X_W\transv \Delta$ and that $\#_{X\in W}(A)=\#_{X_W\in \Delta}(A\times W)$, for every $A\subset M$, moreover if $A\in \mathcal{B}(M)$, then $A\times W\in \mathcal{B}(M\times W)$\footnote{A consequence of this trick is that $\mu(A\times B):=\E\#_{X\in W\cap B}(A)$ defines a Borel measure on the product space $M\times W$. We are going to develop this idea properly later, in Section \ref{sec:constru}.}.

Let $\mathcal{D}$ be the family of subsets $A\subset M$ such that the function $\#_{X\in W}(A)$ is measurable. The class $\mathcal{D}$ contains the subfamily $\mathcal{P}$ of all relatively compact open sets in $M$, which is closed under intersection, hence the idea is to prove that it is also a Dynkin class\footnote{Let $M$ be a nonempty set; a \emph{Dynkin class} $\mathcal{D}$ is a collection of subsets of $M$ such that:
\begin{enumerate}
    \item $M\in \mathcal{D}$;
    \item if $A, B\in \mathcal{D}$ and $A\subset B$, then $B\backslash A\subset \mathcal{D};$
    \item given a family of sets $\{A_k\}_{k\in \N}$ with $A_k\in \mathcal{D}$ and $A_k\subset A_{k+1}$, then $\bigcup_{k}A_k\in \mathcal{D}$.
\end{enumerate}
The Monotone Class Theorem (see \cite[p. 3]{Erhan}) says that if a family $\mathcal{D}$ is a Dynkin class which contains a family $\mathcal{P}$ closed by intersection, then it contains also the  $\sigma$-algebra generated by $\mathcal{P}$.} to conclude, by the Monotone Class Theorem (see \cite[p. 3]{Erhan}), that $\mathcal{D}$ contains the $\sigma$-algebra generated by $\mathcal{P}$, which is precisely the Borel $\sigma$-algebra $\mathcal{B}(M)$.  

Actually, to prevent $\#_{X\in W}(A)$ from taking infinite values, it is more convenient to consider a countable increasing family of relatively compact open subsets $M_i$ such that $\cup_i M_i=M$ and work with the class $\mathcal{D}_i=\{A\in\mathcal{D}\colon A\subset M_i\}$, since $\#_{X\in W}(M_i)$ is almost surely finite.

By previous considerations, $M_i\in \mathcal{D}_i$. If $A,B\in\mathcal{D}_i$ and $A\subset B$, then  $B\- A \in \mathcal{D}_i$, because since $\#_{X\in W}(B)$ is almost surely finite, we can write $\#_{X\in W}(B\- A)=\#_{X\in W}(B)-\#_{X\in W}(A)$ . Suppose that $A_k\in \mathcal{D}$ is increasing, then 
\be \label{eq:contbelow}
\#_{X\in W}(\cup_k A_k)=\lim_{k}\#_{X\in W}(A_k),
\ee
thus $A=U_k A_k\in \mathcal{D}$ because $\#_{X\in W}(A)$ is the pointwise limit of measurable functions, thus in particular if $A_k\in\mathcal{D}_i$, then $A\in \mathcal{D}_i$. It follows that $\mathcal{D}_i$ is indeed a Dynkin class, hence $\mathcal{D}\supset\mathcal{D}_i\supset \mathcal{B}(M_i)$. Now let $A\in\mathcal{B}(M)$, then $A$ is the union of the increasing sequence $A\cap M_i$ and since $A\cap M_i\in\mathcal{B}(M_i)\subset \mathcal{D}$, we can use again the formula in \eqref{eq:contbelow}, to conclude that $A\in\mathcal{D}$.

Clearly $\en$ is finitely additive and $\en(\emptyset)=0$, therefore to prove that $\en$ is a measure it is enough to show that it is continuous from below. This can be seen just by taking the mean value in \eqref{eq:contbelow} and using the Monotone Convergence theorem, since $\#_{X\in W}(A_k)$ is increasing for any increasing sequence $A_k\in\mathcal{B}(M)$.
\end{proof}

%
%

\section{General formula}\label{sec:gen}

This section is devoted to the proof of the following theorem. It is a more general result than the main Theorem \ref{thm:main}, but in fact it is too abstract to be useful on its own. Its role is to create a solid first step for the proof of the main theorem and to better understand its hypotheses.

\begin{figure}\begin{center}
\includegraphics[scale=0.2]{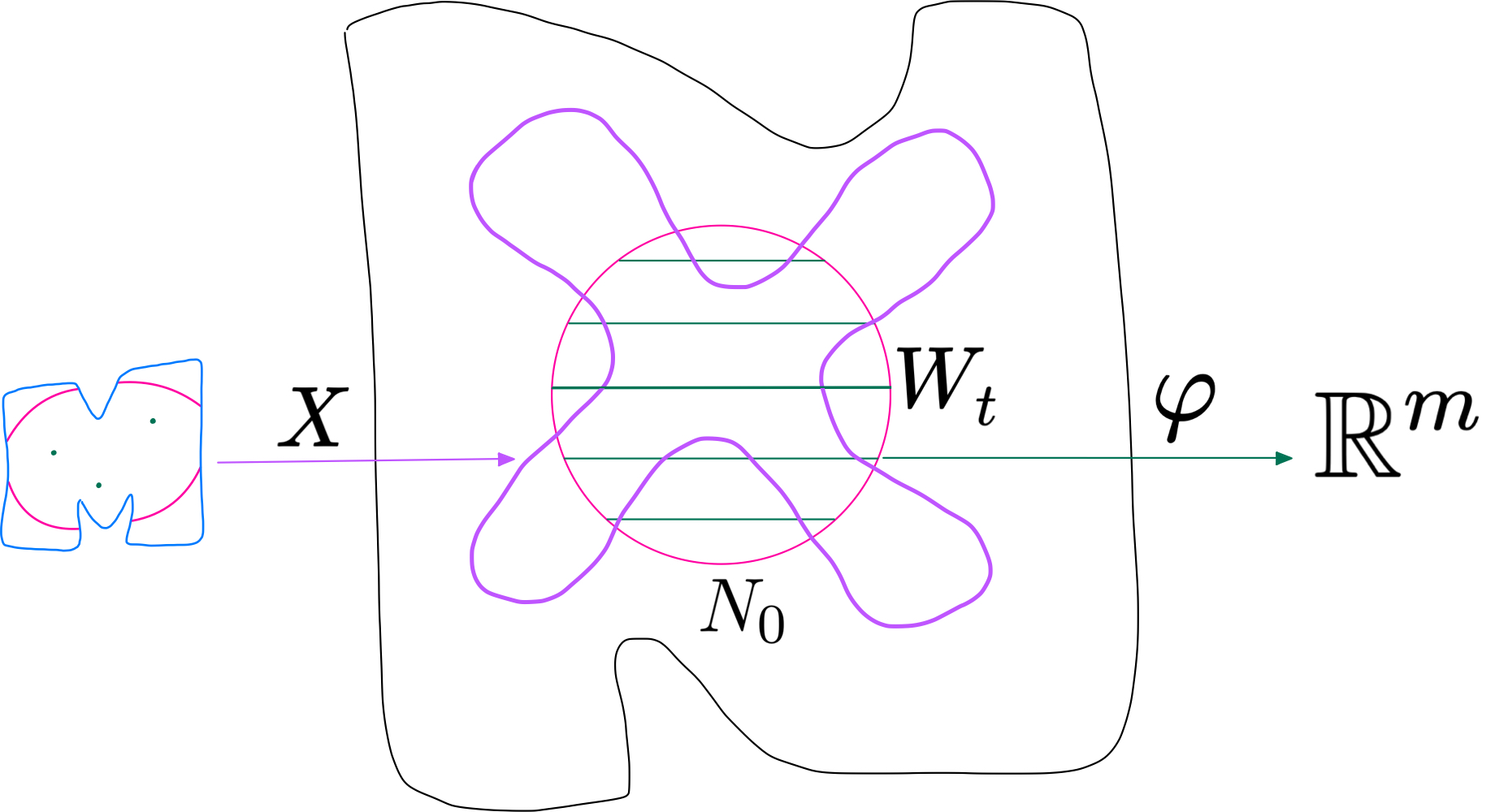}\caption{This figure is meant to give an idea of the set theoretic positions of the objects involved in Theorem \ref{thm:general}.}\label{fig:MXNR}
\end{center}
\end{figure}
\begin{thm}\label{thm:general}
Let $X\colon M\to N$ be a $\mC^1$ random map, such that $d[X(p)]=\rho_{X(p)}dS_p$\footnote{See point \ref{itm:krok:3} of Definition \ref{def:krok}.}, for some Riemannian submanifold $S_p\subset N$ and measurable function $\rho_{X(p)}\colon S_p\to [0,+\infty]$. Let $\{W_t\}_{t\in \R^m}$ be a smooth foliation of an open set $N_0\subset N$, defined by a submersion $\f\colon N_0\to \R^m$ via $W_t=\f^{-1}(t)$,  such that $W_t\transv S_p$ for all $p$ and $t$. Consider the density $\delta_{X\in W_t}(p)$ defined by the same formula given in \eqref{eq:formdel}:
\be
\delta_{X\in W_t}(p)=\int_{ S_p\cap W_t}\E\left\{J_pX\frac{\sigma_x(X,W_t)}{\sigma_x(S_p,W_t)}\bigg|X(p)=x\right\}\rho_{X(p)}(x)d(S_p\cap W_t)(x)dM(p),
\ee
and assume it to be a measurable function with respect to the couple $(p,t)$. Let $A\in \mathcal{B}(M)$ be any Borel subset of $M$.
\begin{enumerate}[$(i)$]
    \item\label{itm:general:1} If $A\in\mathcal{B}(M)$, 
then for almost every $t\in \R^m$
    \be\label{eq:askr}
    \E\#[X^{-1}(W_t)\cap A]=\int_A \delta_{X\in W_t}.
    \ee 
Equivalently, there is a full measure set $T\subset \R^m$, such that for all $t\in T$, the set function $A\mapsto \E\#_{X\in W_t}(A)=\E\#[X^{-1}(W_t)\cap A]$ is an absolutely continuous Borel measure on $M$ with density $\delta_{X\in W_t}$.
    \item\label{itm:general:3} Let $X\transv W_{t_0}$ almost surely and assume that there exists a density $\delta_{top}\in L^1_{loc}(M)$ such that $\limsup_{t\to t_0}\int_{K}\delta_{X\in W_t}\le \int_K{\delta_{top}}$ for every compact set $K\subset M$. Then the measure $\E\#_{X\in W_{t_0}}$ is an absolutely continuous Radon measure on $M$. 
    In this case the corresponding density $\delta\in L^1_{loc}(M)$ satisfies $\delta\le \delta_{top}$.
    
    In particular, if $\delta_{X\in W_t}\to \delta_{X\in W_{t_0}}$ in $L^1_{loc}(M)$, then 
    \be 
    \E\#_{X\in W_{t_0}}(A)=\int_{A}\delta\le\int_{A}\delta_{X\in W_{t_0}}.
   \ee
\end{enumerate}
\end{thm}
\begin{remark}\label{rem:toni}
The left hand side of equation \eqref{eq:askr} is well defined for almost every $t\in\R^k$. This can be seen as follows. Using Tonelli's theorem, we can prove that $\P\{X\transv W_t\}=1$ for almost every $t\in \R^m$:
\be 
\begin{aligned}
\int_{\R^m}\P\{X\not\transv W_t\}dt
&=d[X]\otimes \mathscr{L}^m(\left\{(f,t)\in \ci(M,N)\times \R^m:(\f\circ f)|_{f^{-1}(N_0)}\transv \{t\} \right\}) =\\
&= \E\{\mathscr{L}^m(\{\text{critical values of }(\f\circ X)|_{X^{-1}(N_0)}\})\}=0,
\end{aligned}
\ee
The last equality following from Sard's theorem: critical values of a $\ci$ map between two manifolds of the same dimension form a set of zero Lebesgue measure. Combining this fact with Lemma \ref{lemma:meas}, we deduce that there is a full measure set $T\subset \R^m$ such that the set function
\be 
\mathcal{B}(M)\ni \quad A\mapsto \E\#_{X\in W_t}(A):=\E\#[X^{-1}(W_t)\cap A]\quad \in [0,+\infty]
\ee
is well defined and is a Borel measure.
\end{remark}
The density $\delta_{X\in W_t}$ appearing above is the same that appears in the statement of Theorem \ref{thm:maindens},
where $M$ and $N$ are endowed with auxiliary Riemannian metrics. Since the conditional expectation  
\be 
x\mapsto \E\left\{J_pX\frac{\sigma_x(X,W_t)}{\sigma_x(S_p,W_t)}\bigg|X(p)=x\right\}
\ee
is defined (for every $p$) up to almost everywhere equivalence and $S_p\cap W_t$ has Lebesgue measure zero, it follows that the value of $\delta_{X\in W_t}(p)$ is not uniquely determined. By saying that  $\delta_{X\in W_t}(p)$ should be measurable in $p$ and $t$, we mean that we assume to have chosen a representative of the conditional expectation above in such a way that the function $(p,t)\mapsto \delta_{X\in W_t}(p)$ is measurable. In the KROK case (Definition \ref{def:krok}), however, there is no such ambiguity (see Subsection KROK.\ref{itm:krokEcont}).

Moreover, notice that the value of $\delta_{X\in W_t}(p)$ depends on the choices of $S_p$ and $\rho_{X(p)}$ rather than just $[X(p)]$. 
Indeed, the choice of the submanifold $S_p$ such that $d[X(p)]=\rho_{X(p)}dS_p$ is not unique in general. In fact, $S_p$ can even be replace with $S_p\-W_0$, so that $\delta_{X\in W_0}=0$. This is not in contradiction with the theorem, because the identity 
\eqref{eq:askr} is valid for all $t$ out of a measure zero set. Again, in the KROK case we don't have to worry about that, because, by point \ref{itm:krok:7} of \ref{def:krok}, $S_p\subset N$ is required to be closed.
\begin{remark}\label{rem:sorry}
There is a treacherous measurability issues, that the author wasn't able to solve.
 \emph{Given a random $\mC^1$ map $X\colon M\to N$ and a map $\f\colon N_0\subset N\to \R^m$ satisfying the hypotheses of Theorem \ref{thm:general}, is it true that there exists a version of the density $\delta_{X\in W_t}$ that is measurable with respect to $(p,t)$?} Measurability is crucial to use Fubini-Tonelli's Theorem, thus in Theorem \ref{thm:general} we assume that we are in a situation where the previous question has a positive answer, though such hypothesis may be redundant.
\end{remark}

Before going into the proof of Theorem \ref{thm:general}, let us prove two important preliminary results. The first (see Corollary \ref{cor:measurable}) ensures that the number $\#[X^{-1}(W_t)\cap A]$ is measurable in $(X,t)$ for every Borel set $A$. The second (see Lemma \ref{lem:sette}) gives an alternative expression for $\delta_{X\in W_t}$, that is convenient to use in the proof of the Theorem.
\begin{lemma}\label{lem:fsigmeas}
Let $M,N$ and $T$ be metrizable topological spaces and let $\f\colon K\to T$ be a continuous function,where $K\subset N$ is a closed subset. Then, for every compact set $A\subset M$, the function below is Borel
\be 
\coo{0}{M}{N}\times T \ni \quad (f,t)\mapsto \#(f^{-1}(\f^{-1}(t))\cap A)\quad\in \N\cup \{\infty\}.
\ee
\end{lemma}
\begin{proof}

Fix $\e>0$ and define, for each subset $S\subset M$ the number $\#_{\e}(S)$ to be the minimum number of open subsets of diameter smaller than $\e$ that are needed to cover $S$\footnote{It corresponds to the set function $\mathcal{H}_\e^0(S)$ used in the construction of the Hausdorff measure. }. Observe that $\#(S)=\sup_{\e>0}\#_\e(S)$, therefore we can conclude the proof by showing that the function $(f,t)\mapsto \#_\e(f^{-1}(\f^{-1}(t)))$ is Borel for every $\e$.

Let us consider two convergent sequences $f_n\to f$ in $\coo0MT$ and $t_n\to t$. Assume that $B_1,\dots, B_k$ are open balls in $M$ of diameter smaller than $\e$ such that $\cup_{i=1}^kB_i\supset f^{-1}(t)$. We claim that for $n$ big enough, we have an inclusion
\be 
\bigcup_{i=1}^kB_i\supset f_n^{-1}(\f^{-1}(t_n)).
\ee 
If not, there would be a sequence $p_n\in f_n^{-1}(\f^{-1}(t_n))$ such that $p_n\notin \cup_{i=1}^kB_i$; now, by the compactness of $M$, we can assume that $p_n\to p \notin \cup_{i=1}^kB_i$, but then we find a contradiction as follows. We have $f_n(p_n)\in K$ for all $n$, so that $f(p)\in K$, thus
 $f(p)\in\f^{-1}(t)$ because
$t_n=\f(f_n(p_n))\to \f(f(p))$. 
Hence $p\in f^{-1}(\f^{-1}(t))\subset \cup_{i=1}^kB_i$, which contradicts a previous statement.

It follows that $\#_\e(f^{-1}(\f^{-1}(t)))$ is an upper semicontinuous function:
\be 
\limsup_{(f_n,t_n)\to (f,t)}\#_\e(f_n^{-1}(\f^{-1}(t_n)))\le \#_\e(f^{-1}(\f^{-1}(t)))
\ee
and therefore it is measurable.
\end{proof}
\begin{cor}\label{cor:measurable}
Let $X\colon M\to N\supset W_t\subset  N_0\xrightarrow{\f}\R^m$ satisfy the hypotheses of Theorem \ref{thm:general}. Let $A\subset \mathcal{B}(M)$ be any Borel set. Then the function $u_A(f,t):= \#[f^{-1}(W_t)\cap A]$ is a measurable function on the completion of the measure space $\left(\mC^1(M,N)\times \R^m,\mathcal{B},[X]\otimes \mathcal{L}^m\right)$.
\end{cor}
\begin{proof}
Let us consider the Borel set $\mathcal{F}=\{(f,t)\colon f\transv W_t\}\subset \mC^1(M,N)\times \R^m$. Remark \ref{rem:toni} implies that $\mathcal{F}$ is a full measure subset of $\left(\mC^1(M,N)\times \R^m,\mathcal{B},[X]\otimes \mathcal{L}^m\right)$, therefore we can conclude by showing that the restriction $(u_A)|_{\mathcal{F}}$ is Borel for every $A\in\mathcal{B}(M)$. 

Let $K_i\subset N_0$, for $i\in \N$, be a sequence of increasing closed subsets (in $N$) whose union is equal to $N_0$ and define
\be 
u_A^{(i)}(f,t)=\#[f^{-1}(W_t\cap K_i)\cap A].
\ee
Then $u_A(f,t)=\sup_{i\in \N} u_A^{(i)}(f,t)$, and we already know that $u_A^{(i)}$ is Borel for any compact subset $A\subset M$ because of Lemma \ref{lem:fsigmeas}. Moreover, if $(f,t)\in\mathcal{F}$, we have that $f^{-1}(W_t\cap K_i)$ is a closed discrete subset of $M$, 
 because $f_t\transv W_t$ and $W_t\cap K_i=(\f|_{K_i})^{-1}(t)$ is closed. Therefore $(u_A^{(i)})\big|_{\mathcal{F}}$ is finite whenever $A$ is contained in a compact set.

Thanks to this observation, we can argue exactly as in the proof of Lemma \ref{lemma:meas}: write $M=\cup_j M_j$ as an increasing union of compact subsets; show that the family $\mathcal{D}_j$ of subsets $A\subset M_j$ such that $U_A^{(i)}|_{\mathcal{F}}$ is measurable forms a Dynkin class; conclude with the Monotone Class Theorem (see \cite{Erhan}) that $(u^{(i)}_A)|_{\mathcal{F}}$ is Borel for any $A\in\mathcal{B}(M)$.
\end{proof}
The following Lemma helps to rewrite the candidate formula \eqref{eq:formdel} for the Kac-Rice density into something that is more directly comparable with the Coarea formula (see Theorem \ref{thm:coarea}), which will be one of the main ingredient in the proof of Theorem \ref{thm:general}.
\begin{lemma}\label{lem:sette}
Let $X\colon M\to N\supset W_t\subset  N_0\xrightarrow{\f}\R^m$ satisfy the hypotheses of Theorem \ref{thm:general}. Then 
\be 
\delta_{X\in W_t}(p)=\int_{S_p\cap W_t}\E\left\{\frac{\delta_p(\f\circ X)}{J_x(\f|_{S_p})}\bigg| X(p)=x\right\}\rho_{X(p)}(x)d(S_p\cap W_t)(x).
\ee
\end{lemma}
\begin{proof}
 It is sufficient to prove the following identity (the Definitions of the object in play are in the Appendices \ref{app:densities} and \ref{app:angle}):
\be 
\frac{\delta_p(\f\circ X)}{J_x(\f|_{S})}= \delta_pX \frac{\sigma_x(W,d_pX)}{\sigma_x(W,S)},
\ee
where $X(p)=x$, $\f(x)=t$ and $W=\f^{-1}(t)$. Let us take an orthonormal basis $(\tau w \eta)$ of $T_xN$ such that the first vectors $\tau$ form a basis of $T_xW\cap T_xS$ and the $\eta$ form a basis of $T_xW^\perp$. The matrix of $d_x\f$ in such basis has the form $(00A)$ for some $m\times m$ invertible matrix $A$ and the space $T_xS$ (written in terms of the basis $(\tau w \eta)$) is spanned by the columns of a matrix of the form
\be 
\begin{pmatrix}
\mathbb{1} & 0\\
0 & B \\
0 & C
\end{pmatrix}.
\ee
Without loss of generality we can choose $\begin{pmatrix} B^T & C^T\end{pmatrix}^T$ to be an orthonormal frame. 
Notice that $\det C\neq 0$ because $T_xS\transv T_xW$ and that $\ker (d_u\f|_{S})$ is spanned by the first $m$ columns, since they correspond to the basis $\tau$, hence $J_x(\f|_{S})=|\det A C|$.

For any $u\in M=\D^m$ we have, by definition, that $\delta_{u}(\f\circ X)=J_u(\f\circ X)du$.
\be 
\begin{aligned}
\frac{J_p(\f\circ X)}{J_x(\f|_{S})}&=
\left|\frac{\det\left(d_p\f\frac{\de X}{\de u}\right)}{\det(A)\det(C)} \right|=\\
&=
\frac{|\det(A)| \vol\left(\Pi_{T_xW^\perp}\left(\frac{\de X}{\de u}\right)\right)}{|\det(A)|\vol\left(\Pi_{T_xW^\perp}\begin{pmatrix}
0\\ B \\ C
\end{pmatrix}\right)}  =\\
&= \frac{\sigma\left(T_xW, \text{span}\left(\frac{\de X}{\de u}\right)\right)\cdot \vol\left(\frac{\de X}{\de u}\right)}{\sigma\left(T_xW, T_xS\right)} .
\end{aligned}
\ee
\end{proof}
\subsection{Proof of the general formula: Theorem \ref{thm:general}}
\begin{proof}[Proof of Theorem \ref{thm:general}]
\ref{itm:general:1} Let $A\in\mathcal{B}(M)$. In what follows, let us keep in mind that, by Corollary \ref{cor:measurable}, the function $u_A(f,t)=\#(f^{-1}(W_t)\cap A)$ is measurable on the completion of the measure space $\left(\mC^1(M,N)\times \R^m,\mathcal{B},[X]\otimes \mathcal{L}^m\right)$. This, together with the hypothesis that $(p,t)\mapsto \delta_{X\in W_t}(p)$ is measurable, means that we don't need to worry about measurability issues.

Let $a\in B^+(\R^m)$ and apply the Area formula (see Theorem \ref{thm:area}) to the (deterministic) map $(\f\circ X)|_{X^{-1}(N_0)}\colon X^{-1}(N_0)\to \R^k$, with 
\be 
g(p)=a\left(\f\left(X(p)\right)\right)1_A(p)1_{N_0}\left(X(p)\right).
\ee
We obtain an identity, valid for all $X\in \mathcal{C}^1(M,N)$:
\be \label{eq:detform}
\int_{\R^m}a(t)\#[X^{-1}(W_t)\cap A] dt=\int_{A} \left( (a\circ\f\circ X)\cdot (1_{N_0}\circ X)\right) \delta(\f\circ X).
\ee
Using the Coarea formula \ref{thm:coarea}, we deduce a second identity, as follows:
\be\label{eq:appcoarea}
\E\left\{\left( (a\circ\f\circ X)\cdot (1_{N_0}\circ X)\right) \delta(\f\circ X)\right\}|_p= \ee
\be
\begin{aligned}
&= \E\left\{ a(\f(X(p))1_{N_0}(X(p))\delta_p(\f\circ X)\right\}=\\
&= \int_{N_0} a(\f(x))\E\{\delta_p(\f\circ X)|X(p)=x\}d[X(p)](x) =\\
&= \int_{S_p\cap N_0}a(\f(x))\E\{\delta_p(\f\circ X)|X(p)=x\}\rho_{X(p)}(x)dS_p =\\
&= \int_{\R^m}a(t)\int_{S_p\cap \f^{-1}(t)}\E\{\delta_p(\f\circ X)|X(p)=x\}\rho_{X(p)}(x)\frac{d\left((\f|_{S_p})^{-1}(t)\right)}{J_x(\f|_{S_p})}dt=\\
&=\int_{\R^m}a(t)\left(\int_{S_p\cap W_t}\E\left\{\frac{\delta_p(\f\circ X)}{J_x(\f|_{S_p})}\bigg| X(p)=x\right\}\rho_{X(p)}(x)d(S_p\cap W_t)(x)\right)dt\\
&=\int_{\R^m}a(t)\delta_{X\in W_t}(p) dt.
\end{aligned}
\ee
The Coarea formula was applied to the function $\f|_{S_p\cap N_0}\colon S_p\cap N_0\to \R^k$ in the fourth line.

Taking the expectation on both sides of \eqref{eq:detform} and repeatedly interchanging the order of integration via Tonelli's theorem\footnote{It is possible because the functions involved are positive and measurable.}, we obtain
\be 
\begin{aligned}
\int_{\R^m}a(t)\E\left\{\#[X^{-1}(W_t)\cap A]\right\}dt
&=
\E\left\{\int_{\R^m}a(t)\#[X^{-1}(W_t)\cap A] dt\right\}=\\
&=\E\left\{\int_{A} \left( (a\circ\f\circ X)\cdot (1_{N_0}\circ X)\right) \delta(\f\circ X)\right\}\\
&=\int_A \int_{\R^m}a(t)\delta_{X\in W_t}(p)dt\\
&= \int_{\R^m}a(t)\left(\int_A \delta_{X\in W_t}(p)\right)dt.
\end{aligned}
\ee
Given the arbitrariness of $a$ this proves \ref{itm:general:1}.


\ref{itm:general:3} Lemma \ref{lemma:meas} guarantees that in this case the set function $\mu(A)=\E\#[X^{-1}(W_0)\cap A]$ is a well defined Borel measure on $M$. 

Let $K\subset M$ be a compact set and let $O$ be its interior. If $X\transv W_0$, then $0$ is a regular value for the map $(\f\circ X)\colon X^{-1}(N_0)\to \R^{m}$, hence there is an $\e(X)>0$ such that, for any $0<\e<\e(X)$, the set $(\f\circ X)^{-1}(B_\e)\cap O$ is a disjoint union of balls $U_i$ with the property that $(\f\circ X)|_{U_i}\colon U_i\to B_\e$ is a $\ci$ diffeomorphism, therefore the number $\#[X^{-1}(W_t)\cap O]$ is constant for all $|t|<\e(X)$. 

Let $a_\e\in \mathcal{C}^\infty_c(\R^m)$ be a non negative function supported in the ball $B_\e$ of radius $\e>0$ and with $\int_{\R^m} a_\e =1$. In this case, formula \eqref{eq:detform} implies that
\be 
\#[X^{-1}(W_0)\cap O]= \lim_{\e\to 0^+}\int_{O} \left( (a_\e\circ\f\circ X)\cdot (1_{N_0}\circ X)\right) \delta(\f\circ X).
\ee

Taking expectation on both sides we have
\be
\begin{aligned}
\mu(O)
&=\E\left\{\lim_{\e\to 0^+}\int_{O} \left( (a_\e\circ\f\circ X)\cdot (1_{N_0}\circ X)\right) \delta(\f\circ X)\right\} \\
&\le \liminf_{\e\to 0}\E\left\{\int_{O} \left( (a_\e\circ\f\circ X)\cdot (1_{N_0}\circ X)\right) \delta(\f\circ X)\right\}\quad \textrm{(Fatou)}\\
&= \liminf_{\e\to 0}\int_O\int_{\R^m}a_\e(t)\delta_{X\in W_t} dt\quad\textrm{(Tonelli and \eqref{eq:appcoarea})}\\
&=\liminf_{\e\to 0}\int_{B_\e}a_\e(t)\left(\int_O\delta_{X\in W_t} \right)dt\quad\textrm{(Tonelli again)}\\
&\le \int_K \delta_{top}.
\end{aligned}
\ee


In particular, if $O\subset M$ belongs to the class $\mathcal{S}$ of relatively compact open subsets such that $\de O$ has measure zero, then $\nu(O)\le \int_O\delta_{top}$.
It follows that the measure $\mu$ is absolutely continuous with respect to the measure $\nu=\int \delta_{top}$. Indeed if $A\in \mathcal{B}(M)$ is such that $\nu(A)=0$, then
\be \label{eq:munu}
\mu(A)\le \inf_{B_i \in\mathcal{S},\ A\subset \cup_{i}B_i}\sum_i\mu(B_i)\le \inf_{B_i \in\mathcal{S},\ A\subset \cup_{i}B_i}\sum_i\int_{B_i}\delta_{top}=\nu(A)=0.
\ee
By the Radon-Nikodym theorem (see \ref{thm:RadoNiko}), this implies the existence of a measurable density $\delta\in L^1_{loc}(M)$ such that $\mu(A)=\int_A \delta$ for every $A\in\mathcal{B}(M)$. Moreover, since $\mu(A)\le \nu(A)$ by equation \eqref{eq:munu}, $\delta$ satisfies almost everywhere the inequality: $\delta\le \delta_{top}$.
\end{proof}



\section{Proof of the main Theorem}\label{sec:mainproof}
The goal of this section is to specialize the abstract result of Theorem \ref{thm:general} to the, more friendly, KROK situation. Let us consider a $\mC^1$ random map $X\colon M\to N$ and a submanifold $W\subset N$ such that the couple $(X,W)$ is KROK, that is, it satisfies the conditions i-ix of Definition \ref{def:krok}. 

To facilitate the next proofs, we will show that, without loss of generality we can make some further assumptions on $X$ and $W$.
\begin{figure}\begin{center}
\includegraphics[scale=0.15]{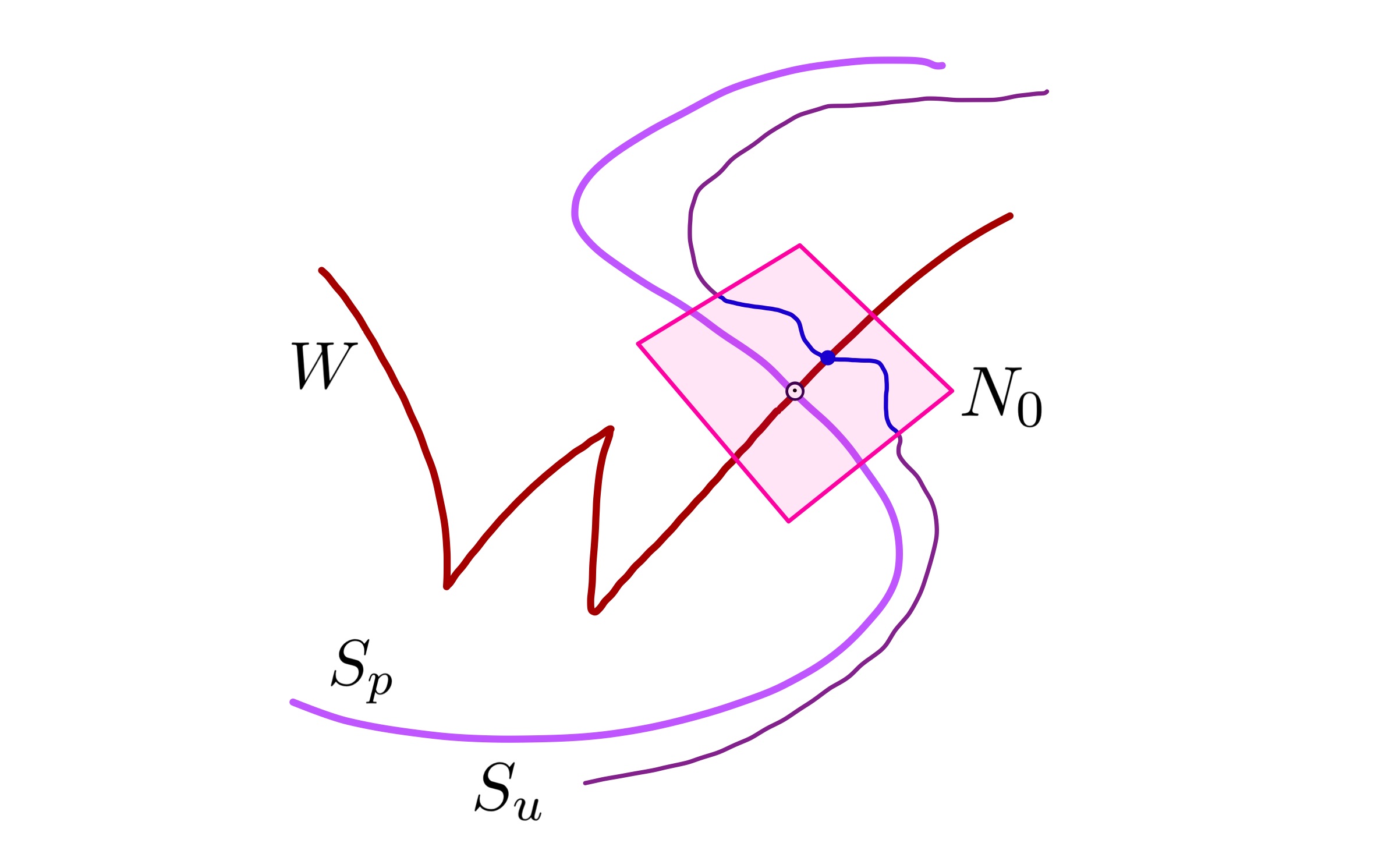}\caption{An illustration of Lemma \ref{lem:dieci}.}\label{fig:WSP}
\end{center}
\end{figure}

\begin{defi}\label{defi:krokcoords}
Let $(X,W)$ be a KROK couple. Let
\be 
    \phi\colon \overline{N_0}\xrightarrow{\sim} \D^m\times\D^{s-m}\times \D^{n-s}=\{ (t,y,z)\}
    \ee 
be a chart of $N$. We say that $\phi$ is a KROK chart at $(p,q)\in \MW $ if the following assumptions are satisfied.
\begin{enumerate}[(i)]
    \item $N_0\subset N$ is a relatively compact open subset such that $W\cap N_0= \phi^{-1}\{t=0\}$.
    \item  There is an open neighborhood $O_p\subset M$ of $p$ and a smooth map $g\in \coo\infty {O_p\times \D^m\times\D^{s-m}}{\D^{n-s}}$, such that
        \be 
        \phi(S_u\cap \overline{N_0})= \text{graph}(g(u,\cdot))=\{(t,y,g(u,t,y))\colon (t,y)\in \D^m\times\D^{s-m}\}.
        \ee
\end{enumerate}
In this case we say that the tuple $ (O_p,N_0,\phi, g)$ is a KROK model for $(X,W)$.
\end{defi}
\begin{lemma}\label{lem:dieci}
\HP. Then for all $(p,q)\in \MW $, there are open sets $p\in O_p\subset M$ and $q\in N_0\subset N$ and a KROK model $(O_p,N_0,\phi,g)$ for $(X,W)$.
\end{lemma}
\begin{proof}
Consider the point $q\in S_p\cap W$. Since $S_p\transv W$, there is a chart $\phi\colon \overline{N_0}\to\D^m\times\D^{s-m}\times \D^{n-s}=\{ (t,y,z)\}$ centered at $q$, such that $W\cap N_0=\{(0,y,z)\colon y,z\}$ and $S_p\cap N_0=\{(t,y,0)\colon t,y\}$. 
Now, let us consider the set $\mathcal{S}=\{(u,v)\in M\times N\colon u\in S_v\}$ (see Definition \ref{def:krok}) and observe that $T_{v}S_u=T_{(u,v)}\mS\cap T_v N$. Since 
$\mathcal{S}$ is a submanifold of $M\times N$ by Definition \ref{def:krok}, we deduce that for any $u,v$ close to $p,q$ in $M\times N$, the tangent spaces $T_vS_u$ and $T_qS_p$ are close to each other. This implies that for all $u$ in a neighborhood $O_p$ of $p$, de submanifold $S_u\cap N_0$ can be parametrized as the graph of a function $g_u(t,y)$ and that this function depends smoothly on $u$.
\end{proof}
\begin{lemma}\label{lem:undici}
\HPK. Let $(X^1,X^2,X^3):=\phi\circ X$ and define 
\be 
q\colon M\times \D^m\times \D^{s-m}\to N_0,
\ee
\be q(u,t,y):=\phi^{-1}(t,y,(g(u,t,y)))\in S_u\cap W_t\cap N_0.\ee Then the expression of $\delta_{X\in W_t}=\rho_{X\in W_t}du$ in the coordinates $t,y,z$ is the following.
\begin{multline}\label{eq:densieasy}
\rho_{X\in W_t}(u) =\\ 
=\int_{ \D^{s-m}}\E\left\{\left|\det\left(\frac{\de X^1(u)}{\de u}\right)\right|\Bigg| X(p)=q(u,t,y)\right\}\rho_{X(u)}(q(u,t,y))\sqrt{\det G(u,t,y)} dy,
\end{multline}
where $G(u,t,y)=\langle  \begin{pmatrix}\frac{\de q}{\de t} & \frac{\de q}{\de y}
\end{pmatrix}^T,\begin{pmatrix}\frac{\de q}{\de t} & \frac{\de q}{\de y}
\end{pmatrix} \rangle$ is the Gram matrix of $S_u$ with respect to the coordinates $t,y$.
\end{lemma}
\begin{proof}
It is convenient to write $\rho_{X\in W_t}$ with the formula of Lemma \ref{lem:sette}, where $\f\colon N_0\to \R^m$ is the function defined by $\f(\phi^{-1}(t,y,z))=t$. Thus, it is sufficient to show that 
\be 
\frac{d(S_u\cap W_t)}{J_v\left(\f|_{S_u}\right)}=\sqrt{\det G}\ dy.
\ee
Let us start by looking at the most tedious piece, namely the jacobian. The matrix of $d_u\f$ in the coordinates $t,y$ is $\begin{pmatrix}
\mathbb{1} & 0
\end{pmatrix}$, so that formula \eqref{eq:jacobidet} in Appendix, yelds
\be 
\begin{aligned}
J_u\left(\f|_{S_u}\right)&=\det\left(\begin{pmatrix}
\mathbb{1}& 0
\end{pmatrix}G^{-1}\begin{pmatrix}
\mathbb{1}\\ 0
\end{pmatrix}\right)^{\frac12}
=\det\left( S^{-1}\right)^{\frac12}= \left( \frac{\det G_{2,2}}{\det G}\right)^{\frac12},
\end{aligned}
\ee
where $S=G_{11}-G_{12}G_{22}^{-1}G_{12}$\footnote{$S$ is called the Schur complement of the block $G_{22}$ in $G=\begin{pmatrix}
G_{11} & G_{12} \\ G_{21} & G_{22}
\end{pmatrix}$.}. The last two equalities can be deduced from the identity: 
\be 
\begin{pmatrix}
G_{11} & G_{12} \\ G_{21} & G_{22}
\end{pmatrix}\cdot \begin{pmatrix}
\mathbb{1} & 0 \\ -G_{22}^{-1}G_{21} & \mathbb{1}
\end{pmatrix}=\begin{pmatrix}
S & G_{12} \\ 0 & G_{22}
\end{pmatrix}.
\ee
Now, since $\phi(N_0\cap S_u\cap W_t)=\{(t,y)\colon y\in \D^{s-m}\}$, the Gram matrix of $S_u\cap W_t$ in the coordinates $y$ is exactly $G_{22}$. Therefore, we conclude
\be 
\begin{aligned}
\frac{d(S_u\cap W_t)}{J_v\left(\f|_{S_u}\right)}
&=
\frac{\sqrt{\det G_{22}}dy}{\sqrt{ \frac{\det G_{2,2}}{\det G}}}
=\sqrt{\det G}\ dy.
\end{aligned}
\ee
\end{proof}
\begin{remark}\label{rem:independencemetrix}
Formula \eqref{eq:densieasy} can be rewritten in a form that doesn't involve the metrics on $M$ and $N$. Define a random element $X^1(u)\in \R^m\cup\{\arian\}$ such that $X^1(u)=\phi^1(X(u))$ if $X(u)\in N_0$ and $X^1(u)=\arian$ whenever $X(u)\notin N_0$. Then, if $A\subset \R^m$,
\be
\begin{aligned}
\P\{X^1(u)\in A\}&= \P\{X(u)\in N_0\cap \phi^{-1}(A\times \D^{s-m})\}\\
&=\int_{A}\left(\int_{\D^{s-m}}\rho_{X(u)}(q(u,t,y))\sqrt{\det G(u,t,y)} dy\right)dt,
\end{aligned}
\ee
hence the restriction of the measure $[X^1(u)]$ to $\R^m$ is absolutely continuous,
so that if we denote its density by $\rho_{X^1(u)}(t)dt$ we obtain an equivalent expression to that in formula \eqref{eq:densieasy}.
\be\label{eq:indepensity} 
\delta_{X\in W_t}(u)=\E\left\{\left|\det\left(\frac{\de X^1(u)}{\de u}\right)\right|\Bigg| \begin{aligned}&X(u)\in N_0 \\&X^1(u)=t\end{aligned}\right\}\rho_{X^1(u)}(t) du.
\ee
Notice that $\delta_{X\in W_t}=\delta_{X^1\in \{t\}}$. The above formula is completely independent from the Riemannian structures of $M$ and $N$.
\end{remark}
\subsection{A construction}\label{sec:constru}

The purpose of this section is to show that Theorems \ref{thm:main} and \ref{thm:maingraph} are actually equivalent. In fact, although the latter is evidently a more general result, it can be proved with a particularly simple application of the former. To understand this,
let us define a new $\mC^1$ random map $\XW\colon \MW \to \NW$, such that 
 \bega\label{eq:XW}
 \XW\colon \MW =\bigsqcup_{p\in M}\{p\}&\times (S_p\cap W) \longrightarrow \NW=\bigsqcup_{(p,q)\in \MW }\{(p,q)\}\times S_p,
 \\
 &(p,q) \mapsto (p,q,X(p)).
 \eega
 \begin{remark}
The fact that $\mathcal{S}$ is closed guarantees that $X(p)\in S_p$ for all $p\in M$, with probability one, hence this definition is well posed.  Indeed given a dense countable subset $D\subset M$, then 
\be 
\P\{X(p)\in S_p\ \forall p\in D\}=\P\left(\bigcap_{p\in D}\{X(p)\in S_p\}\right)=1
\ee 
and if the $X(p)\in S_p$ for all $p\in D$, then $(p,X(p))\in \overline{\mathcal{S}}=\mathcal{S}$ for all $p\in M$, by density and continuity.
\end{remark}
 Define also the diagonal submanifold $\DElta=\{(p,q,q)\colon (p,q)\in \MW \}\subset \NW$.
 Now, the random set $\Gam$ from Theorem \ref{thm:maingraph}, defined for a KROK couple $(X,W)$, can be interpreted as 
 \be 
\Gam=\left\{(p,q)\in M\times W\colon q=X(p)\right\}=\XW^{-1}(\DElta) ,
\ee
 therefore
 $\#_{\Gam}(V)=\#_{\XW\in \DElta}(V\cap \MW )$ for any $V\subset M\times N$.


Observe that if $\dim S_p=s$, then the dimensions of $\MW , \NW, \DElta$ are $s,2s,s$.

\begin{claim}
The couple $(\XW,\DElta)$ is KROK and
\be
d[\XW(p,q)](x,y,y')=\delta_p(x)\delta_q(y)\rho_{X(p)}(y')dZ_{p,q}\footnote{Here $\delta_p$ denotes the $\delta-$measure on the point $p$.}
\ee
where $Z_{p,q}=\{p,q\}\times S_p$. Moreover $ Z_p\cap \DElta=\{(p,q,q)\}$ is always finite.
\end{claim}
\begin{proof}
It is enough to observe that if $X\transv W$, then $\XW\transv \DElta$, and that if $S_p\transv W$, then $Z_{p,q}\transv \DElta$.
\end{proof}
It follows that the expression $V\mapsto \E\#_{\XW\in \DElta}(V)$  defines a Borel measure on $\MW $ and putting $V=A\times B$, where $A\in \mathcal{B}(M)$ and $B\in\mathcal{B}(N)$, we have
\be 
\#_{X\in B\cap W}(A)=\#_{\XW\in \DElta}(A\times B\cap \MW ).
\ee
In other words, we can consider the measure $\E\#_{X\in W}(\cdot)$ as the section $\mu(W\times(\cdot))$ of a Borel measure $\mu=\E\#_{\XW\in \DElta}(\MW \cap (\cdot))$ on $M\times N$, with $\text{supp}(\mu)\subset \MW $\footnote{It is a strict inclusion, when $\rho_{X(p)}(q)= 0$ on a not negligible set of points $p,q$.}.

Moreover, the fact that $Z_{p,q}\cap \DElta$ is always equal to the point $(p,q,q)\in \NW$, permits to get rid of the integral in the formula for $\delta_{\XW\in \DElta}$:
\be
\delta_{\XW\in \DElta}(p,q)=\E\left\{\delta_{(p,q)}\XW\frac{\sigma_{(p,q,q)}(\XW,\DElta)}{\sigma_{(p,q,q)}(Z_{p,q},\DElta)}\bigg|X(p)=q\right\}\rho_{X(p)}(q).
\ee

\begin{lemma} Let $\XW\colon \MW  \to \NW$ be the map defined above. Then the density element $\delta_{\XW\in \DElta}(p,q)\in \Delta_{(p,q)}\MW$ can be written as follows.
\bega
\delta_{\XW\in \DElta}(p,q)&=
\E\left\{J_pX\frac{\sigma_x(X,W)}{\sigma_x(S_p,W)}\bigg|X(p)=x\right\}\rho_{X(p)}(x)d\MW (p,q)
\eega
where $\MW $ is endowed with the Riemannian metric induced by the isomorphism $T_{(p,q)}\MW \cong T_pM\oplus T_q(W\cap S_p)$, by declaring it to be an orthogonal splitting.
\end{lemma}
\begin{proof}
\HPK . Define $q(u,t,y)=\phi^{-1}(t,y,g(u,t,y))$. In the rest of the proof we will identify $O_p\cong \D^{m}$, so that in particular $p=0$ and $q=q(0,0,0)$. 
 Let $N_0^\mathcal{X}$ be the open neighborhood of $(p,q)$ in $\NW$, defined as the image of the map
\be 
\D^s\times \D^s\ni(t^\mathcal{X},z^\mathcal{X})=\left((t',y'),(u,y)\right)\mapsto \left(u,q(u,0,y),q(u,t',y+y')\right).\footnote{After a rescaling of the coordinate $y$, we can assume that $\D^s\times\D^s\subset \phi(N_0)$. }
\ee
We call $\phi^\mathcal{X}\colon \overline{N^\mathcal{X}_0}\to\D^s\times \D^s$ the inverse of the above map, which provides a coordinate chart for $N_0^\mathcal{X}$. 
Let us consider the following small open subset in $\MW $ 
\be O^\mathcal{X}_{(p,q)}=\{(u,q(u,0,y))\colon (u,y)\in \D^{s}\}\cong \D^s\ee
and let us define the coordinate $u^\mathcal{X}=(u,y)$ on it.

We can see now that $(O_{(p,q)}, N_0^\mathcal{X},\phi^\mathcal{X},g^\mathcal{X})$ is a KROK model for $(\XW,\DElta)$, where $ g^\mathcal{X}(u^\mathcal{X},t^\mathcal{X})=u^\mathcal{X}$. Indeed $\DElta\cap N_0^\mathcal{X}=\{t^\mathcal{X}=0\}$ and $\phi^\mathcal{X}(Z_{u^\mathcal{X}}\cap N_0^\mathcal{X})=\{(t^\mathcal{X}, z^\mathcal{X})\colon z^\mathcal{X}=u^\mathcal{X}\}$ is equal to the graph of the map $ g^\mathcal{X}$.
It follows that $\delta_{\XW\in \DElta}=\rho_{\XW\in \DElta}du^\mathcal{X}$ can be represented with the formula of Lemma \ref{lem:sette}, where $\f(p,q,q')=t^\mathcal{X}$.
\bega 
\rho_{\XW\in \DElta}(p,q)
&=\E\left\{\left|\det\left(\frac{\de \XW^1}{\de u^\mathcal{X}}\right)\right|\Bigg| \XW(p,q)=(p,q,q)\right\}\rho_{\XW(p,q)}(p,q,q)\sqrt{\det G^\mathcal{X}}\\
&= 
\E\left\{\bigg|\det
\begin{pmatrix}
\frac{\de X^1}{\de u} & 0 \\ \frac{\de X^2}{\de u} & -\mathbb{1}
\end{pmatrix}
\bigg|\Bigg|X(p)=q\right\}\rho_{X(p)}(q) \vol\left(\frac{\de (\phi^\mathcal{X})^{-1}}{\de t^\mathcal{X}}\right)\\
&=
\E\left\{\left|\det\left(\frac{\de X^1}{\de u}\right)\right|\Bigg|X(p)=q\right\}\rho_{X(p)}(q)\vol\left(\frac{\de }{\de (t',y')}q(u,t',y+y')\right)\\
&=
\E\left\{\left|\det\left(\frac{\de X^1}{\de u}\right)\right|\Bigg|X(p)=q\right\}\rho_{X(p)}(q)\sqrt{\det G}.
\eega 
Here we used that, according to our definition, $\XW^1(u,y)=(X^1(u),X^2(u)-y)$.

In terms of the density,  recalling Lemma \ref{lem:undici} we just showed that
\bega 
\delta_{\XW\in \DElta}&=\E\left\{\left|\det\left(\frac{\de X^1}{\de u}\right)\right|\Bigg|X(p)=q\right\}\rho_{X(p)}(q)\sqrt{\det G}dydu\\
&=
\E\left\{\delta_pX\frac{\sigma_q(X,W)}{\sigma_q(S_p,W)}\bigg|X(p)=q\right\}\rho_{X(p)}(q)d(S_p\cap W)du.
\eega
and since, by definition, $d(S_p\cap W)du=d\MW $, we conclude.
\end{proof}
In other words, the relation between the absolutely continuous Borel measures with densities $\delta_{X\in W}=\rho_{X\in W}dM\in L^+(M)$ and $\delta_{\XW\in \DElta}=\rho_{\XW\in \DElta}d\MW \in L^+(\MW )$ mirrors the relation between the Borel measures $\E\#_{X\in W}$ on $M$ and $\E\#_{\XW\in \DElta}$ on $\MW$. Indeed, observing that $d\MW (p,q)=dM(p)\otimes d(S_p\cap W)(q)$, we have
\be 
\delta_{X\in W}(p)=\int_{W\cap S_p}\delta_{\XW\in \DElta}(p,q)dq.
\ee
Moreover, the hypotheses II and III, ensure that $\delta_{\XW\in \DElta}$ is a continous density on $\MW $.

\subsection{Proof of Theorems \ref{thm:main} and \ref{thm:maingraph}}
\begin{thm}
\HP. Then for any $A\in\mathcal{B}(M)$
\be 
\E\#_{X\in W}(A)=\int_{A}\delta_{X\in W}.
\ee
\end{thm}
\begin{proof}
Because of the construction of Section \ref{sec:constru}, we can assume that $S_p\cap W=\{q(p)\}$ is a point, for every $p\in M$.

 By Lemma \ref{lem:dieci} we can cover $M$ with a countable collection of embedded $m$-disks $D_{i}$, such that for each $i$, there is an open set $N_i\subset N$ containing $q(D_i)$ and such that there exists a KROK model $(D_i,N_i,\phi_i,g_i)$ for each $i$ (see Definition \ref{defi:krokcoords}). in particular $X(O_i)\subset N_i$. If we assume that the theorem holds for each $X_i=X|_{O_i}$, then for every $A\subset O_i$ we have:
\be 
\E\#_{X\in W}(A)=\E\#_{X_i\in W_i}(A)=\int_{A}\delta_{X_i\in W_i}=\int_{A}\delta_{X\in W}.
\ee 
The last equality is due to the fact that for every $p\in A$ the point $q(p)$ is already in $N_i$.
This implies that the two Borel measures $\E\#_{X\in W}$ and $\int \delta_{X\in W}$ coincide for every $A$ small enough ($A\subset O_i$ for some $i$), thus they are equal.

For this reason, we can assume that there is a global KROK model $(M,N,\phi,g)$. In this case the variable $y$ is not needed because $\dim(S_p)=m$ and we have, from Lemma \ref{lem:undici}, that $\delta_{X\in W_t}=\rho_t du$ with
\be 
\rho_t(u):=\E\left\{\left|\det\left(\frac{\de X^1(u)}{\de u}\right)\right|\Bigg|X(u)=q(u,t)\right\}\rho_{X(u)}(q(u,t))\sqrt{\det G(u,t)},
\ee
where $q(u,t)=\phi^{-1}(t,g(u,t))$ and $q(u,0)=q(u)$.
The KROK assumptions ensure that the function $(u,t)\mapsto \rho_{t}(u)$ is continuous at $M\times 0$, thus $\delta_{X\in W_t}\to \delta_{X\in W_0}$ in $L^1_{loc}$, so
 that from point \ref{itm:general:3} of Theorem \ref{thm:general}, applied with $\f(t,z)=t$, it follows that there exists a measurable function $g\colon M\to [0,1]$, such that $\E\#_{X\in W} =\int g du$ and $g\le \rho_0$.
To end the proof it is sufficient to show that $g(p)\ge \rho_0(p)$ for almost every $p\in M$. 

Let us consider the subset $P\subset M$ (recall that we are assuming $M=\D^{m}$) of all Lebesgue points for $g$, so that $p\in P$ if and only if
\be 
g(p)=\lim_{r\to 0}\frac{\E\#_{X\in W}(B_r)}{\mathscr{L}^m(B_r)}.
\ee
 We will prove that if $p\in P$, then $g(p)\ge \rho_0$. After that, the proof will be concluded since, by Lebesgue Differentiation theorem, $P$ is a full measure set in $M$.

Let $B_r$ be a closed ball of radius $r>0$ in $M$ centered at $p\in P$. 
Let $f^1=\f\circ f$ and consider the set
\be 
\mathscr{A}_s=\left\{f\in \mathcal{C}^1(M,N)\colon \begin{aligned}
& f(B_s)\subset N_0, \\
&f^1|_{B_s}\colon B_s\hookrightarrow \R^m\text{ is an embedding}
\end{aligned}
\right\}.
\ee
It is straightforward to see that $\mathscr{A}_s$ is an open set in $\mathcal{C}^1(M,N)$. Define a family of continuous functions $\a_{s,\eta}\colon \coo1MN\to [0,1]$ such that $a_{s,\eta}\nearrow\mathbb{1}_{\mathscr{A}_s}$, when $\eta\to +\infty$. 

Let $r\le s$. Now, the random variable $\#_{X\in W_t}(B_r)\a_{s,\eta}(X)$ is bounded by $1$ and converges almost surely as $t\to 0$, because $\E\#_{X\in W}(\de B_r)=0$.
As a consequence, by dominated convergence, we have that
\be
\E\left\{\#_{X\in W_0}(B_r)\a_{s,\eta}(X)\right\}=\lim_{t\to 0}\E\left\{\#_{X\in W_t}(B_r)\a_{s,\eta}(X)\right\}.
\ee
Moreover, arguing as in the proof of point \ref{itm:general:1} of Theorem \ref{thm:general}, we have the following equality for almost every $t\in \R^m$ and every $\eta$:
\be 
\E\left\{\#_{X\in W_t}(B_r)\a_{s,\eta}(X)\right\}=\int_{B_r}\rho_{X\in W_t}^{s,\eta}(u)du,
\ee
where 
\be 
\rho_{X\in W_t}^{s,\eta}(u):=\E\left\{\left|\det\left(\frac{\de X^1(u)}{\de u}\right)\right|\a_{s,\eta}(X)\Bigg|X(u)=q(u,t)\right\}\rho_{X(u)}(q(u,t))\sqrt{\det G(u,t)}.
\ee
The important point here is that $\rho_{X\in W_t}^{s,\eta}(u)$ is continuous at $t=0$\footnote{This is why we defined the continuous functions $\a_{s,\eta}$, instead of simply using the characteristic functions $\mathbb{1}_{{\mathcal{A}_s}}$.} because of the KROK. \ref{itm:krokEcont}, so that we are allowed to do the last step in the following sequence of inequalities. 
\be 
\begin{aligned}
\E\#_{X\in W}(B_r)&\ge 
\E\{\#_{X\in W_0}(B_r)\a_{r,\eta}(X)\} 
\\
&=\lim_{t\to 0}\int_{B_r}\E\{\#_{X\in W_t}\a_{r,\eta}(X)\}\\
&=\lim_{t\to 0}\int_{B_r}\rho_{X\in W_t}^{r,\eta}(u)du\\
&=\int_{B_r}\rho_{X\in W_0}^{r,\eta}(u)du.
\end{aligned}
\ee
Now, because $p\in P$ and $\rho_{X\in W_0}^{s,\eta}$ is continuous, we have that
\be
\begin{aligned}
g(p)&=\lim_{r\to 0}\frac{\E\#_{X\in W}(B_r)}{\mathscr{L}^m(B_r)}
\ge
\lim_{r\to 0}\frac{\int_{B_r}\rho_{X\in W_0}^{s,\eta}(u)du}{\mathscr{L}^m(B_r)}
=\rho_{X\in W_0}^{s,\eta}(p).
\end{aligned}
\ee
Taking the supremum over all $\eta$ we obtain that
\be\label{eq:daje}
g(p)
\ge  \E\left\{\left|\det\left(\frac{\de X^1}{\de u}(p)\right)\right|\mathbb{1}_{\mathscr{A}_s}(X)\Bigg|X(p)=q\right\}\rho_{X(p)}(q)\sqrt{\det G(p,0)},
\ee
where $q=q(p,0)\in W$.
Observe that $\mathscr{A}_s$ is an increasing family (as $s\to 0$) whose union is equal to the set $\mathscr{A}_0$ consisting of all functions $f\in \coo1MN$, such that $f(p)\in N_0$ and such that $d_pf^1$ is invertible. 
Therefore, recalling that $X\transv W$ almost surely, we see that taking the supremum with respect to $s$ in \eqref{eq:daje} we
obtain the thesis, valid for all $p\in P$:
\be 
g(p)\ge \rho_0(p).
\ee 
\end{proof}


%
%
%
%

\subsection{The case of fiber bundles: Proof of Theorem \ref{thm:megafica}}

\begin{proof}[Proof of Theorem \ref{thm:megafica}]
(We refer to Appendix \ref{app:angle} for the notations with frames.)
Let $q\in W$ and $p=\pi(q)$.
Let us take an orthonormal basis $\de_u$ of $T_pM$ and an orthonormal basis $\de_y$ of $T_qS_p\cap T_qW$. Let us complete the latter to an orthonormal basis of $T_qS_p$ by adding an orthonormal frame $\de_t$. Let us denote by $h(de_u)\in T_q S_p^\perp$ the frame such that $d_q\pi(h(\de_u))=\de_u$.
Finally, observe that $T_qW\cap (T_qS_p\cap T_qW)^\perp$ is contained in the space generated by $(h(\de_u),\de_t)$ and it projects surjectively on $T_pM$, hence it has a (not orthonormal) basis of the form $\de_z=h(\de_u)+A\de_t$, for some $m\times m$ matrix $A$.
 (the letters are coherent with the KROK coordinates, see Definition \ref{defi:krokcoords}).
 
 Now, by construction and by Definition \ref{defi:angle}, we have
 \bega
 \frac{d(S_p\cap W)(q) dM(p)}{\sigma(S_p,W)}&=\frac{\vol(\de_z)}{\vol\begin{pmatrix}
 \de_t & \de_z 
 \end{pmatrix}} dy du
 \\
 &= \frac{1}{\vol{\begin{pmatrix}
 \mathbb{1}_m & A \\
 0 & \mathbb{1}_m
 \end{pmatrix}}}\left(\vol\begin{pmatrix}
 \de_y & \de_z
 \end{pmatrix}dydu\right)
\\
&=dW(p).
\eega
\end{proof}
\subsection{Other counting measures: Proof of Theorem \ref{thm:mainalph}}
\begin{proof}[Proof of Theorem \ref{thm:mainalph}]
It is sufficient to prove the case in which $\a$ is continuous, bounded and positive. Indeed, then the result can be extended by monotone convergence to any positive Borel function and finally to any Borel function $\a$ by presenting it as $\a=\a^+-\a^-$.
If $\a$ is continuous, then the hypothesis KROK.\ref{itm:krokEcont} ensures that one can repeat the whole proof of Theorem \ref{thm:maindens} with $\#^\a_{X\in W}$ and $\delta^\a_{X\in W}$. The only thing to check is the very first step, which is provided by the Area and Coarea formulas at equations \eqref{eq:detform} and \eqref{eq:appcoarea} in the proof of the general formula of Theorem \ref{thm:general}. 
For the weighted case, we have to apply the Area formula to the function $\f\circ X$, with
\be 
g(p)=\a(X,p) a\left(\f\left(X(p)\right)\right)1_A(p)1_{N_0}\left(X(p)\right),
\ee
to get a generalization of identity \eqref{eq:detform}:
\bega
\int_{\R^m}a(t)\#^\a_{X\in W_t}(A)dt &=\int_{\R^m}a(t)\left(\sum_{p\in A\cap X^{-1}(W_t)}\a(X,p) \right)dt
\\
&=\int_{A} \a(X,p) a(\f\circ X(p)) 1_{N_0}( X(p)) \delta(\f\circ X)(p).
\eega
On the other hand, via the Coarea formula, the identity \eqref{eq:appcoarea} becomes.
\bega
&\ \E\left\{\a(X,p) a(\f(X(p))1_{N_0}(X(p))\delta_p(\f\circ X)\right\}=
\\
&= \int_{S_p\cap N_0}a(\f(x))\E\{\a(X,p)\delta_p(\f\circ X)|X(p)=x\}\rho_{X(p)}(x)dS_p 
\\
&=\int_{\R^m}a(t)\left(\int_{S_p\cap W_t}\E\left\{\a(X,p)\frac{\delta_p(\f\circ X)}{J_x(\f|_{S_p})}\bigg| X(p)=x\right\}\rho_{X(p)}(x)d(S_p\cap W_t)(x)\right)dt
\\
&=\int_{\R^m}a(t)\delta^\a_{X\in W_t}(p) dt.
\eega
\end{proof}
\section{The Gaussian case: Proofs}
\label{sec:proofgauss}

\subsection{Proof of Theorem \ref{thm:maingau} and Corollaries \ref{cor:maingauW} and \ref{cor:connectedform}}
\begin{proof}[Proof of Theorem \ref{thm:maingau}]
It suffices to show that the couple $(X,W)$ is KROK and use Theorem \ref{thm:main}. The only conditions of Definition \ref{def:krok} that have to be checked are KROK.\ref{itm:kroktransvas} and KROK.\ref{itm:krokEcont}.
 The former is a consequence of the Probabilistic Transversality theorem from \cite{dtgrf} and the latter follows from next Lemma.
\end{proof}

\begin{lemma}\label{lem:gioiellino}
Let $X=(X^0,X^1)\colon \R^m\to\R^k\times\R^h$ be a $\mC^1$ Gaussian random field and assume that $X^0$ is non-degenerate Let $\a\colon \mC^1(\R^m,\R^{k+h})\times\R^m\to \R$ be a continuous function such that
\be\label{eq:subpolynomial}
|\a\left(f,p\right)|\le N\left(\left|f(p)\right|+\left|\de_1 f(p)\right|+\dots +\left|\de_m f(p)\right|\right)^N+N,
\ee
for some constant $N>0$. Then
 the next function is continuous
 \be
 \R^m\times\R^k\ni (p,q)\mapsto \E\left\{\a(X,p)\big|X^0(p)=q\right\}\in \R.
 \ee
\end{lemma}
 \begin{proof}
Let us fix two points $u,p\in\R^m$. By using a standard argument we can find a Gaussian random vector $Y(u,p)\in \R^{k+h}$ that is independent from $X^0(p)$ and such that 
\be\label{eq:gausdegress}
X(u)=A(u,p)X^0(p)+Y(u,p)
\ee
For some $(k+h)\times k$ matrix $A(u,p)$.
Then, defining $K^{*,0}_X(u,p)=\E\{X(u)X^0(p)^T\}$, we have 
\be\label{eq:covagausdegress}
K^{*,0}_X(u,p)\left(K_X^{0,0}(p,p)\right)^{-1}=A(u,p),
\ee
from which we deduce that $A$ is a $\mC^1$ function of $u,p$. 

Consider now the $\mC^1$ Gaussian random field $Y_p=Y(\cdot,p)\colon \R^m\to \R^{k+h}$ defined, for each $p\in \R^m$, by the identity \eqref{eq:gausdegress} above. 
The previous computation ensures that the random vectors $Y_p(u)$ and $X^0(p)$ are independent for every $u\in\R^m$ and therefore $Y_p(\cdot)$ is independent from the vector $X^0(p)$ as a field. It follows that
\be\label{eq:expgausdegress}
\E\{\a(X,p)|X^0(p)=q\}=\E\{\a(Y_p(\cdot)+A(\cdot,p)q,p)\}=\E\{\a(Z_{p,q},p)\},
\ee
for any $q\in\R^k$, where $Z_{p,q}$ is the $\mC^1$ (noncentered) Gaussian random field \bega
Z_{p,q}(\cdot)&=Y_p(\cdot)+A(\cdot,p)q\colon \R^m\to \R^{k+h}\\
&=
X(\cdot)+A(\cdot,p)\left(q-X^0(p)\right)
\eega

To end the proof, let us show that the right hand side of equation \eqref{eq:expgausdegress} depends continuously on $p,q$. Given two converging sequences: $p_n\to p$ in $\R^m$ and $q_n\to q$ in $\R^k$, it is clear that the random map $Z_{p_n,q_n}$ converges almost surely to $Z_{p,q}$ in the space of $\mC^1$ functions, thus
\be\label{eq:fatou1lastlemma}
\E\{\a(Z_{p,q},p)\}\le \liminf_{n\to+\infty} \E\{\a(Z_{p_n,q_n},p_n)\}
\ee
by Fatou's lemma. Now observe that, since $Z_{p,q}(p),\de_1 Z_{p,q}(p),\dots, \de_mZ_{p,q}(p)$ are jointly Gaussian and their covariance matrices depend continuously on $p,q$ because of equation \eqref{eq:covagausdegress}, then the function $\f\colon\R^m\to \R$:
\be
\f(p,q)=\E\{N\left(|Z_{p,q}(p)|+|\de_1Z_{p,q}(p)|+\dots +|\de_m Z_{p,q}(p)|\right)^N+N\}
\ee
is continuous. The hypothesis \eqref{eq:subpolynomial} allows us to apply Fatou's Lemma once again to obtain that $\limsup_{n\to+\infty}\E\{\a(Z_{p_n,q_n},p_n)\}\le \E\{\a(Z_{p,q},p)\}$:
\begin{multline}
\f(p,q)-\E\{\a(Z_{p,q},p)\}
 \\
=\E\left\{\left(N\left(|Z_{p,q}(p)+|\de_1Z_{p,q}(p)|+\dots +|\de_m Z_{p,q}(p)|\right)^N+N\right)-\a(Z_{p,q},p)\right\}\\
\le
\liminf_{n\to+\infty}\left(\f(p_n,q_n)-\E\{\a(Z_{p_n,q_n},p_n)\}\right)\\
=
\f(p,q)-\limsup_{n\to+\infty}\E\{\a(Z_{p_n,q_n},p_n)\}.
\end{multline}
This, together with \eqref{eq:fatou1lastlemma}, implies that $\E\{\a(Z_{p_n,q_n},p_n)\}\to\E\{\a(Z_{p,q},p)\}$.
\end{proof}
\begin{proof}[Proof of Corollary \ref{cor:maingauW}]
Follows directly from the version of the formula for fiber bundles.
The only novelty to prove here is that the density $\delta_{\Gam}\in \mathscr{D}^0(W)$ is smooth in this case. This follows from Lemma \ref{lem:gagagaga} below.
\end{proof}
\begin{proof}[Proof of Corollary \ref{cor:connectedform}]
Since $W\subset E$ is parallel for the structure, we can express the formula of Corollary \ref{cor:maingauW}, in the form described in Remark \ref{rem:connectionForm}, so that $\E\#\left(W \cap V\right)=\int_{W\cap V}\delta_{\Gam}$ where
\be 
\delta_{\Gam}=
\E\left\{|\det\left(\Pi_{T_xW_p^\perp}\circ\left(\nabla^X X\right)_p\right) \Big| X(p)=x\right\}\frac{e^{-\frac{|x|^2}{2}}}{(2\pi)^{\frac s2}}dW(x).
\ee
Here, we already used the fact that the metric structure $g$ is defined by $X$, meaning that $K_p^{-1}\langle x,x\rangle=g(x,x)$ and $\det K_p=1$. Furthermore, from the definition of $\nabla^X$ given by the identity \eqref{eq:nablaX} (compare with \cite{Nicolaescu2016}) we observe that, in a trivialization chart, for any $i,j$ and $v\in T_pM$, we have
\be 
\E\{(\nabla_v^X X)^i_p X^j(p)\}=\E\left\{(D_vX)^i_p X^j(p)-\E\{(D_v X)^i_p |X(p)\}X^j(p)\right\}=0,
\ee
meaning that $(\nabla^X X)_p$ and $X(p)$ are uncorrelated and thus, being Gaussian, independent. Therefore the conditioning can be removed:
\be 
\E\left\{|\det\left(\Pi_{T_xW_p^\perp}\circ\left(\nabla^X X\right)_p\right) \Big| X(p)=x\right\}=\E\left\{|\det\left(\Pi_{T_xW_p^\perp}\circ\left(\nabla^X X\right)_p\right) \right\}
\ee
\end{proof}
\subsection{Proof that the notion of sub-Gaussian concentration is well defined}
\begin{lemma}\label{lem:subGwell}
Let $\pi\colon E=\D^m\times\R^s\to \D^m$ and let $W\subset E$ be a smooth submanifold. If $W\subset E$ has sub-Gaussian concentration for some linearly connected Riemannian metric on $\pi$, then the same holds for any other.
\end{lemma}
\begin{proof}
In this case $E=\D^m\times\R^s\to\D^m$,  $E_u=\{0\}\times \R^s$, a connected Riemannian metric is any Riemannian metric whose matrix is of the form
\be\label{eq:lcRmetric}
g_E(u,x)=
\begin{pmatrix}
g_M(u) & 0
\\ 0 & g_v(u)
\end{pmatrix}+
\begin{pmatrix}
\Gamma(u,x)^Tg_v(u)\Gamma(u,x) & \Gamma(u,x)^Tg_v(u)
\\ g_v(u)\Gamma(u,x) & 0
\end{pmatrix},
\ee
where $g_M(u)$ is the Riemannian metric on $\D^m$, while $\Gamma\colon E\to \R^{m\times s}$ is some smooth function. Moreover, in this case, we have that the horizontal space $H_{(u,x)}=(T_{(u,x)}E_u)^\perp$ is the graph of the linear function $\R^m\ni\dot{u}\mapsto -\Gamma(u,x)\dot{u}\in\R^s$.

The connection is linear precisely when $\Gamma(u,x)$ depends linearly on $x$, so that we can define the Christoffel symbols\footnote{They are the Christoffel symbols of the corresponding covariant derivative: given a smooth section $s(u)=(u,s_v(u))$,  the vertical projection of $d_us$ is given by
\be (\nabla s)_u=d_us_v+\Gamma(u,s_v(u))\in\R^{m\times s}.\ee } $\Gamma_{k,i}^j\in \mC^\infty(\D^m)$ as follows
\be 
\Gamma_{k}^j(u,x)=\sum_{i=1}^s\Gamma_{k,i}^j(u)x^i, \quad \forall k=1,\dots, m \text{ and } j=1,\dots, s.
\ee
Let $g_0$ and $g_1$ be any two linearly connected metrics on $E$ and assume that $W$ has sub-Gaussian concentration for $g_0$. The corresponding vertical norms on $E_u$
\be 
\|x\|_i=\sqrt{g_i\langle\begin{pmatrix}
0 \\ x
\end{pmatrix},\begin{pmatrix}
0 \\ x
\end{pmatrix}\rangle}
\ee
are equivalent and since $u\in \D^m$ there exists a uniform constant $N>0$ such that
\be 
\frac 1 {N}\|\cdot\|_0\le \|\cdot\|_1\le N\|\cdot\|_0.
\ee
Moreover, given the linearity of $\Gamma$  in equation \eqref{eq:lcRmetric}  for both $g_0$ and $g_1$, we can assume that
\be 
dW_1(u,x)\le N(1+R^2)dW_0(u,x),
\ee
for all $x\in B_R^1\cap W$ and every $R>0$. It follows that
\bega 
\vol_{W_1}\left(W\cap B^1_R\right)&=\int_{W\cap B^1_R}dW_1 
\\
&\le \int_{W\cap B_R^1}N(1+R^2)dW_0
\\
&\le \int_{W\cap B_{NR}^0}N(1+R^2)dW_0
\\
&\le N(1+R^2)C_0(\delta)\exp(\delta N^2R^2).
\eega
the last inequality is true for any $\delta>0$ because $W$ has sub-Gaussian concentration with respect to the metric $g_0$. Now, for every fixed $\e>0$,  there exists $C_1(\e)>0$ such that
\be 
N C_0\left(\frac12 \e N^{-2}\right)(1+R^2)\le C_1(\e)\exp\left(\frac12 \e R^2\right),
\ee
simply because $1+R^2=o(e^{\frac12 \e R^2})$. We obtain that $W$ has sub-Gaussian concentration for the metric $g_1$ as well, since for every fixed $\e>0$, setting $\delta=\frac12 \e N^{-2}$, we get
\be 
\vol_{W_1}\left(W\cap B^1_R\right)\le C_1(\e)\exp\left(\frac12 \e R^2\right)\exp(\delta N^2R^2)=C_1(\e)\exp\left( \e R^2\right).
\ee
\end{proof}

\subsection{Proof of theorem \ref{thm:mainEgau}}
Since the statement is local we can assume that $E=\D^m\times\R^s$ and $X\colon \D^m\to\D^m\times \R^m$ is a smooth Gaussian random section of
the trivial bundle, so that $X(u)=(u,V(u))$ for some smooth Gaussian random function $V\colon \D^m\to\R^s$. Let us define the smooth Gaussian random field 
\be 
j^1X\colon \D^m\to \R^{s}\times \R^{m\times s}, \quad u\mapsto j^1_uX=(V(u),d_uV),
\ee
Since $X$ is smooth and nondegenerate, the covariance matrix of $j^1_uV$ 
\bega
K_{j^1X}\colon \D^m \to \mathcal{U}, \quad
K_{j^1_uX}= \begin{pmatrix}
K_0(u) & K_{01}(u)\\ K_{10}(u) & K_1(u)
\end{pmatrix}
\eega
is a smooth function taking values in the open set 
\be 
\mathcal{U}=\left\{K=\begin{pmatrix}
K_0 & K_{01}\\ K_{10} & K_1
\end{pmatrix}\in \R^{(m+ms)\times (m+ms)}\colon \begin{matrix}
\text{$K$ is symmetric, } \\
\text{semipositive}
\\
\text{and $\det(K_0)>0$}
\end{matrix}\right\}.
\ee 
For the proof of the first statement of the theorem, it is sufficient to prove the following Lemma.
\begin{lemma}\label{lem:gagagaga}
Let $E=\D^m\times \R^s$, be endowed with the standard Euclidean metric. Let $E_u:=\{u\}\times\R^s$, 
and let $W\subset E$ be a smooth submanifold of codimension $m$ such that $W\transv E_u$ for every $u\in\D$.\footnote{Notice that $T_{(u,x)}W$ is a well defined $m$ dimensional subspace of $\R^m\times\R^s$ even if $u\in\de\D^m$. In such case the transversality $E_u\transv T_{(u,x)}W$ is still meant in the space $\R^m\times\R^s$.}
There exists a continuous function
\be 
f_W\colon W\times \mathcal{U}\to \R,
\ee
such that for any nondegenerate smooth Gaussian random section $X\colon\D^m\to\D^m\times\R^s$, with $X(u)=(u,V(u))$, we have
\be 
\delta_{\Gam}(u,x)=f_W(u,x,K_{j^1_uX})dW(u,x).
\ee
Moreover, there exists a smooth function $N\in\mathcal{U}\to\R_+$ such that, $\forall (u,x,K)\in W\times\mathcal{U}$,
\be 
f_W(u,x,K)\le N(K)\exp\left(-\frac{1}{N(K)}\|x\|^2\right).
\ee
\end{lemma}
\begin{proof}
By Equation \eqref{eq:gaudelgamma}, we have
\be 
\delta_\Gam=\E\left\{\a(u,x,d_uV)|V(u)=x\right\} \rho_{X(u)}(x)dW(u,x),
\ee
where $\rho_{X(u)}$ is the density of the Gaussian random vector $X(u)$, while where $\a\colon W\times \R^{m\times s}\to \R_+$ is the function defined as follows.
\be \label{eq:alpha}
\a(u,x,J):=\vol\begin{pmatrix}
\mathbb{1}_m \\ J
\end{pmatrix}\sigma\left(\text{Im}\begin{pmatrix}
\mathbb{1}_m \\ J
\end{pmatrix},T_{(u,x)}W\right)=\vol(\Pi_{T_{(u,x)}W^\perp}\begin{pmatrix}
\mathbb{1}_m \\ J
\end{pmatrix}),
\ee
where $\vol$ and $\sigma$ are taken with respect to the Euclidean Riemannian structure on $E=\D^m\times\R^s$ (in the last equality we used Proposition \ref{prop:angleperp} from Appendix \ref{app:angle}). 

From formula \eqref{eq:dgauvec} it is clear that $\rho_{X(u)}(x)=\rho(u,x,K_{j^1_uX})$ with $\rho\colon E\times \mathcal{U}\to\R_+$ being the smooth function
\be \label{eq:rhocazzo}
\rho (x,K)=
\frac{
\exp
\left(
-\frac{1}{2}x^TK_0^{-1}x
\right)
}{
\pi^\frac s2\sqrt{\det\left(K_0\right)}
}
\le N_\rho(K)\exp(-\frac{\|x\|^2}{N_\rho(K)}).
\ee
Clearly there exists a continuous function $N_\rho\colon \mathcal{U}\to\R_+$ such that the inequality \eqref{eq:rhocazzo} holds.

Let us consider the function $\E\{\a(u,x,d_uV)|V(u)=x\}$. We will show that it is a continuous function of $u,x$ and $K_{j^1X}$. 

First, observe that $\a$ is smooth because $W$ is a smooth submanifold. Like we did in the proof of lemma \ref{lem:gioiellino}, let us define a new Gaussian random vector $J(u)$ such that $ 
d_uV=J(u)+A(u)V(u)
$, for some smooth function $A(u)\in\R^{m\times ms}$, linear in $x\in \R^s$, and such that $J(u)$ and $V(u)$ are independent. Then $K_{J(u)}=K_J(K_{j^1_uX})$ and $A(u)=A(K_{j^1_uX})$, where $K_J,A\in\mC^0(\mathcal{U})$ are smooth functions:
\bega\label{eq:dependance}
K_{J}(K)&=K_1-K_{10}K_0^{-1}
K_{01};
\\
A(K)&=K_{10}K_0^{-1}.
\eega
Consider the expression
$E(u,x,K_J,A)=\E\{\a(u,x,J-Ax)\}$. Since $\a$ is of the form \eqref{eq:alpha}, all of its derivatives are bounded by a polynomial in $J$, when $J\to+\infty$ with $u,x$ fixed. From this it follows that
$\E\{\a(u,x,J-Ax)\}$
is finite and depends smoothly on $u,x,A,K_J$, where $K_J$ is the covariance matrix of $J$. Therefore, the function $F(u,x,K)=E\{\a(u,x,K_J(K),A(K)\}$ is smooth in $u,x,K$ and we have
\be 
\E\{\a(u,x,d_uV)|V(u)=x\}=F\left(u,x,K_{j^1_uX}\right).
\ee

Moreover, let $J(K)\in\R^{m\times s}\cong \R^{ms}$ be the Gaussian random matrix corresponding to the Gaussian random vector having covariance matrix $K_J(K)$. Let $\|\cdot\|$ be any norm on $\R^{m\times s}$. Then
\bega
F(u,x,K)&=\E\{\a(u,x,J(K)-A(K)x)\}
\\
&\le \E\left\{\vol\begin{pmatrix}
\mathbb{1}_m \\ J(K)+A(K)x
\end{pmatrix}\right\}
\\
&\le \left(\E\left\{\left\|\vol\begin{pmatrix}
\mathbb{1}_m \\ J(K)
\end{pmatrix}\right\|^m\right\}^\frac1m
+ \| A(K)x\|
\right)^m
\\
&\le N_\a(K)\left(1+\|x\|\right)^m.
\eega
for some big enough continuous function $N_\a\colon \mathcal{U}\to \R_+$. This is because the $m^{th}$ moment of a Gaussian random vector depends continuously on its covariance matrix. 
Combining this with \eqref{eq:rhocazzo}, we see that the function 
\be 
f_W(u,x,K):=F(u,x,K)\rho(x,K).
\ee 
has all the properties that we wanted to show.
\end{proof}

Let us show the continuity statement for a sequence of GRSs $X_d$, with limit $X_\infty$.
 By Theorem \ref{thm:maingau} and Lemma \ref{lem:gagagaga} we have
\be 
\E\#_{X_d\in W}(A)=\int_{\pi^{-1}(A)\cap W}\delta_{\Gamma(X_d,W)}dW
=
\int_{\pi^{-1}(A)\cap W} f_W(u,x,K_d)dW(u,x),
\ee
where $K_d=K_{j^1X_d}$ converges to $K_\infty=K_{j^1X_{\infty}}$ because they depend linearly on the $2^{nd}$ jet of $K_{X_d}$ and $K_{X_\infty}$.
From Lemma \ref{lem:gagagaga} we deduce that $f_W(u,x,K_d)\to f_W(u,x,K_{\infty})$, therefore we could conclude by the dominated convergence theorem if we show that $f_W(u,x,K_d)$  is uniformly bounded by an integrable (on $W$) function. Since $K_d\in\mathcal{U}$ is a convergent sequence, we have that $N(K_d)$ is uniformly bounded by some constant $N\in\N$, thus 
\bega
f_W(u,x,K_d)&\le Ne^{-\frac{|x|^2}{N}}.
\eega
Now, the fact that $W$ has sub-Gaussian concentration implies that the latter function is integrable on $W$:
\bega
\int_W Ne^{-\frac{|x|^2}{N}}dW(u,x) &\le \sum_{R\in\N}\int_{W\cap B_R} Ne^{-\frac{(R-1)^2}{N}}dW(u,x)
\\
&\le 
\sum_{R\in\N}\vol_W\left(W\cap B_R\right) Ne^{-\frac{(R-1)^2}{N}}
\\
&\le 
\sum_{R\in\N}C(\e)e^{\e R^2} Ne^{-\frac{(R-1)^2}{N}}<+\infty,
\eega
for all $\e<\frac{1}{N}$. this concludes the proof of theorem \ref{thm:mainEgau}.

\subsection{Proof of Corollary \ref{cor:mainjgau}}
\begin{proof}[Proof of Corollary \ref{cor:mainjgau}]
Corollary \ref{cor:mainjgau} follows directly from Theorem \ref{thm:mainEgau} and the following observation. The GRS $j^rX$ is a Gaussian random section of the vector bundle $J^rE\to M$. Its first jet $j^1(j^rX)$ is equivalent to $j^{r+1}X$ and thus its covariance tensor is determined by the $(2r+2)^{th}$ jet of $K_X$. 
\end{proof}
%
%
%
%
\appendix 
\section{Densities}\label{app:densities}
Let $M$ be a smooth manifold. The \emph{density bundle} or, as we call it, the \emph{bundle of density elements} $\Delta M$ is the vector bundle:
\be \Delta M=\wedge^m(T^*M)\otimes L,\ee
where $L$ is the orientation bundle (see \cite[Section 7]{botttu}); $\Delta M$ is a smooth real line bundle and the fiber can be identified canonically with
\bega\label{eq:apdensi}
\Delta_pM 
=\{\delta\colon (T_pM)^m\to \R\colon \delta=\pm |\w|, \text{ for some $\w\in \wedge^mT^*_pM$}\}.
\eega
We call an element $\delta\in \Delta_pM$, a \emph{density element} at $p$.

 Given a set of coordinates $x^1,\dots,x^m$ on $U$, we denote
\be
dx=dx^1\dots dx^m=|dx^1\wedge\dots\wedge dx^m|.
\ee
Using the language of \cite{botttu}, $dx$ is the section $(dx_1\wedge\cdots \wedge dx_m)\otimes e_U$, where $e_U$ is the section of $L|_{U}$ which in the trivialization induced by the chart $(x_1, \ldots, x_m)$ corresponds to the constant function $1$.

We have, directly from the definition, that 
\be 
dx=\left|\det\left(\frac{\de x}{\de y}\right)\right|dy.
\ee
for any other set of coordinates $y^1,\dots,y^m$. It follows that an atlas for $M$ with transition functions $g_{a,b}$ defines a trivializing atlas for the vector bundle $\Delta M$, with transition functions for the fibers given by $|\det(Dg_{a,b})|$. 

Each density element $\delta$, is either positive or negative, respectively if $\delta=|\w|$ or $\delta=-|\w|$, for some skew-symmetric multilinear form $w\in \wedge^mT^*_pM$. In other words, the bundle $\Delta M$ is canonically oriented and thus trivial (but not canonically trivial) and we denote the subbundle of positive elements by $\Delta^+ M\cong M\times [0,+\infty)$ (again, not canonically).

The \emph{modulus} of a density element is defined in coordinates, by the identity $|(\w(x)dx)|=|\w(x)|dx$. This defines a continuous map $|\cdot|\colon \Delta M\to \Delta^+ M$.

The sections of $\Delta M$ are called \emph{densities} and we will usually denote them as maps  $p\mapsto \delta(p)$. We define $\mathscr{D}^r(M)$ to be the space of $\mC^r$ densities and by $\mathscr{D}^r_c(M)$ the subset of the compactly supported ones. For smooth densities, we just write $\mathscr{D}=\mathscr{D}^\infty$.
From the formula for transitition functions it is clear that there is a canonical linear function 
\be 
\int_M \colon \mathscr{D}_c(M)\to \R \qquad \int_M \delta =:\int_{M}\delta(p)dp.
\ee
If $N$ is Riemannian and $f\colon M\to N$ is a $\ci$ map with $\dim M\le \dim N$, we define the \emph{jacobian density} $\delta f\in\mathscr{D}(M)$ by the following formula, in coordinates:
\be\label{eq:indensity}
\delta_u f=
\sqrt{\det\left(\frac{\de f}{\de u}^Tg(f(u))\frac{\de f}{\de u}\right)}du\in \Delta_p^+M,
\ee
In particular $\delta(\text{id}_M)$ is the Riemannian volume density of $M$ and we denote it by $dM$. Similarly, if $f$ is a Riemannian inclusion $M\subset N$, then $\delta f=dM$. 

Given any Riemannian metric on the manifold $M$, one can make the identification $\Delta_p M= \R (dM(p))$ and treat densities as if they were functions. Moreover this identification preserves the sign, since $dM$ is always a positive density: $dM(p)\in\Delta_p^+M$, for every $p\in M$. 
We denote by $B(M)$ the set of all Borel measurable functions $M\to [-\infty,+\infty]$ and by $L(M)=B(M)dM$ the set of all measurable, not necessarily finite, densities (the identification with $B(M)$ depends on the choice of the metric). Let $B^+(M)$ be the set of positive measurable functions $M\to [0,+\infty]$ and $L^+(M)$ be the set of densities of the form $\rho dM$, for some $\rho\in B^+(M)$. In other words 
\be 
L^+(M)=\left\{ \text{measurable functions }M\ni p\mapsto \delta(p)\in\Delta^+_pM \cup \{+\infty\}\right\}
\ee 
is the set of all nonnegative, non necessarily finite measurable densities. 
The integral can be extended in the usual way (by monotone convergence) to a linear function $\int_M\colon L^+(M)\to \R\cup\{\infty\}$. Similarly, we define the spaces $L^1(M),L^\infty(M),L^1_{loc}(M),L^\infty_{loc}(M)$ and their respective topologies by analogy with the standard case $M=\R^m$. 

\begin{defi}\label{def:abscont} We say that a real Radon measure (see \cite{ambrofuscopalla}) $\mu$ on $M$ is \emph{absolutely continuous} if $\mu(A)=0$ on any zero measure subset $A\subset M$. In other words $\mu$ is absolutely continuous if and only if $\f_*\left(\mu|_{U}\right)$ is absolutely continuous with respect to the Lebesgue measure $\mathscr{L}^m$ for any chart $\f\colon U\subset M\to \R^m$.
\end{defi}
In this language, the Radon-Nikodym Theorem takes the following form.
\begin{thm}\label{thm:RadoNiko} Let $\mu$ be an absolutely continuous real Radon measure on $M$. Then there is a density $\delta\in L^1_{loc}(M)$ such that, for all Borel subsets $A\subset M$,
\be 
\mu(A)=\int_A\delta.
\ee
\end{thm}
\section{Angle between subspaces}\label{app:angle}
Let $E, \langle\cdot,\cdot \rangle$ be a euclidean vector space (i.e. a finite dimensional real Hilbert space).
\begin{defi}
Let $f=(f_1,\dots,f_k)$ be a tuple (in row) of vectors $f_i\in E$. We define its volume as
\be 
\vol(f)=\sqrt{\det\langle f^T, f\rangle}.
\ee
If the vectors $f_1,\dots,f_k$ are independent, then $f$ is called a \emph{frame}. The \emph{span} $\text{span}(f)$ of the tuple $f$ is the subspace of $E$ spanned by the vectors $f_1,\dots,f_k$. A tuple $f$ is a \emph{basis} of a subspace $V\subset E$ if and only if it is a frame and its span is $V$. Given two frames $v,w$, we can form the tuple $(v,w)$.
\end{defi}
\begin{defi}\label{defi:angle}
Let $V,W\subset E$ be subspaces. Let $v$ and $w$ be frames in $E$ such that:
\begin{enumerate}[$\bullet$]
    \item $\text{span}(v)=V\cap (V\cap W)^\perp$;
    \item $\text{span}(w)=W\cap (V\cap W)^\perp$.
\end{enumerate}
We define the \emph{angle} 
 between $V$ and $W$ as
\be 
\sigma(V,W)=\begin{cases}
\frac{\vol(v w)}{\vol(v)\vol(w)} \quad &\text{ if $V\not\subset W$ and $W\not\subset V$;}\\
1 &\text{ otherwise.} 
\end{cases}
\ee
\end{defi}
It is easy to see that the definition is well posed, independently from the choices of the frames.  This definition corresponds to Howard's \cite{Howard} in the case when $V\cap W=\{0\}$.
Observe that $\sigma$ is symmetric and that we have $\sigma(V,W)=1$ if and only if $V=A\oplus_\perp B$ and $W=A\oplus_\perp C$ with and $B\perp C$.

When $V$ and $W$ are one dimensional, $\sigma(V,W)=|\sin\theta |$, where $\theta$ is the angle between the two lines. In general $\sigma(V,W)\in [0,1]$ is equal to the product of the sines of the nontrivial principal angles between $V$ and $W$ (see \cite{JordAngles,Zhu_2013_ang,MIAO199281_ang}). In particular, notice that $\sigma(V,W)>0$ always.

It is important to notice that $\sigma\colon \text{Gr}_k(E)\times \text{Gr}_h(E)\to [0,1]$ is not a continuous function. However the restriction to the subset of the pairs of subspaces $(V,W)$ such that $\dim (V+W)=n$ is continuous.
\begin{prop}\label{prop:chi}
Assume that $W\not\subset V$. Let $w$ be a basis for $W\cap (V\cap W)^\perp$, as in Definition \ref{defi:angle}, then
\be 
\sigma(V,W)=\frac{\vol\left(\Pi_{V^\perp}(w)\right)}{\vol(w)}
\ee
\end{prop}
\begin{proof}
First, observe that the projected frame $\Pi_{V^\perp}(w)$ is a basis of the space $(V+W)\cap V^\perp$. Let $\nu$ be an orthonormal basis of the same space, let $v$ be an orthonormal basis for $V\cap (V\cap W)^\perp$ and let $\tau$ be a basis for $V\cap W$. It follows that $(\tau, v,\nu)$ is a basis for $V+W$. Therefore, there is an invertible matrix $B$ and a matrix $A$ such that
\be w=\begin{pmatrix}\tau & v &\nu\end{pmatrix}\begin{pmatrix} 0\\ A\\ B\end{pmatrix}.\ee
Then, by Definition \ref{defi:angle}, we have
\be 
\begin{aligned}
\sigma(V,W)=\frac{\vol(v,w)}{\vol(v)\vol(w)}
=\frac{\left|\det\begin{pmatrix}\mathbb{1} & A \\ 0 & B \end{pmatrix}\right|}{\vol(w)}=\frac{\vol\left(\Pi_{V^\perp}(w)\right)}{\vol(w)}.
\end{aligned}
\ee
\end{proof}
\begin{prop}\label{prop:angleperp}
$\sigma(V^\perp,W^\perp)=\sigma(V,W).$
\end{prop}
\begin{proof}
The statement is trivially true if $V\subset W$ or $W \subset V$, so let us assume that this is not the case. Let $\nu$ be an orthonormal basis of the space $(V+W)\cap V^\perp$ and $v$ be an orthonormal basis of $V\cap (V\cap W)^\perp$. Besides, let $(\tau,w)$ be an orthonormal basis of $W$, such that $\tau$ is a basis for $V\cap W$. We have 
\be 
\begin{aligned}
\sigma(V^\perp,W^\perp)&=\sigma(W^\perp,V^\perp)\\
&=\vol\left(\Pi_{W}(\nu)\right)=\\
&=\det\langle w^T,\nu\rangle=\\
&=\det\langle \nu^T,w\rangle=\\
&= \vol \left(\Pi_{V^\perp}(w)\right)=\sigma(V,W).
\end{aligned}
\ee
\end{proof}
\section{Area and Coarea formula}\label{app:coarea} 
\begin{defi}\label{def:jacob}
Let $\f\colon M\to N$ be a $\ci$ map between $\ci$ Riemannian manifolds. The \emph{Jacobian} (often called \emph{normal Jacobian} when $f$ is a submersion) of $\f$ at $p\in M$ is
\be 
J_p\f:=\begin{cases}
0 & \text{if rank$(d_p\f)$ is not maximal;} \\
\frac{\vol_{N}\left(d_p\f(e)\right)}{\vol_{M}(e)} & \text{otherwise;}
\end{cases}
\ee
where $e=(e_1,\dots, e_k)$ is any basis of $\ker(d_p\f)^\perp\subset T_pM$. 

If $L\colon V_1\to V_2$ is a linear map between metric vector spaces, then we write $J L:=J_0L$ (Clearly $J_p\f=Jd_p\f$). If $V_1,V_2$ have the same dimension, then, to stress this fact, we may write $|\det L|:=JL$, although the sign of $\det L$ is not defined, unless we specify orientations.
\end{defi}
In particolar, let $\f\colon (\R^m;g_1)\to (\R^n;g_2)$ with differential $\frac{\de \f}{\de u}(u)=A$ having maximal rank, then
\be \label{eq:jacobidet}
J_u\f:=\begin{cases}
\sqrt{\frac{\det(A^Tg_2A)}{\det(g_1)}} & \text{if $m\le n$;} \\
\sqrt{\frac{\det(Ag_1^{-1}A^T)}{\det(g_2)^{-1}}} & \text{if $m\ge n$.}
\end{cases}
\ee
\begin{remark} In the case $m\le n$, the density induced on $M$ by a map $\f$, defined in \eqref{eq:indensity} corresponds to
$\delta_p\f=(J_p\f) dM$.
\end{remark}
\begin{thm}[Area formula]\label{thm:area}
Let $f\colon M\to N$ be a Lipschitz map between $\ci$ Riemannian manifolds, with $\dim M=\dim N$. Let $g\colon M\to [0,+\infty]$ be a Borel function, then
\be 
\int_M g(p)(J_pf) dM(p)=\int_N\left[\sum_{p\in f^{-1}(q)}g(p)\right]dN(q).
\ee
\end{thm}
\begin{proof}
See \cite[Theorem 3.2.3]{federer1996}.
\end{proof}
The Area formula is actually much more general than this, in that it holds for $\dim M\le \dim N$ and with the Hausdorff measure instead of $dN$. However, this simplified statement is all that we need in this paper. It also can be thought as a generalization of the following, in the case $\dim M=\dim N$.
\begin{thm}[Coarea formula]\label{thm:coarea}
Let $f\colon M\to N$ be a $\mC^1$ submersion between smooth Riemannian manifolds, with $\dim M\ge \dim N$. Let $g\colon M\to [0,+\infty]$ be a Borel function, then
\be 
\int_M g(p) (J_pf) dM(p)=\int_N\int_{f^{-1}(q)}g(p) d\left(f^{-1}(q)\right)(p) dN(q).
\ee
\end{thm}
\begin{proof}
See \cite{chavel} or deduce it from \cite[Theorem 3.2.12]{federer1996}.
\end{proof}

%
%
%
\bibliographystyle{plain}
\bibliography{KacRiceFormulaForTransverseIntersections.bib}

\end{document}